\documentclass[10pt,a4paper]{article}
\usepackage{amsmath, amsthm}
\usepackage{amsfonts}
\usepackage{amssymb}
\usepackage{mathtools}

\usepackage{cases}
\usepackage{empheq}
\usepackage{enumerate}
\usepackage{enumitem}
\usepackage{times}

\usepackage[hyperref, usenames, svgnames,x11names, dvipsnames]{xcolor}
\usepackage[ocgcolorlinks, colorlinks=true, citecolor=Cerulean, linkcolor=Bittersweet, urlcolor=blue]{hyperref}
\usepackage{enumitem}

\usepackage[left=2cm,right=2cm,top=2cm,bottom=2cm]{geometry}
\usepackage{comment}

\newcommand{\norm}[2]{\left\lVert #1\right\rVert_{#2}}

\newcommand{\weaklyto}{\rightharpoonup}
\newcommand{\weaklystar}{\overset{\ast}{\rightharpoonup}}

\newcommand{\cl}{\mathrm{cl}}

\DeclareMathOperator*{\esssup}{ess\,sup}

\newcommand{\cts}{\hookrightarrow}
\newcommand{\ctsCompact}{\xhookrightarrow{c}}

\newtheorem{theorem}{Theorem}[section]
\newtheorem{prop}[theorem]{Proposition}
\newtheorem{lem}[theorem]{Lemma}
\newtheorem{cor}[theorem]{Corollary}

\newtheorem{remark}[theorem]{Remark}
\newtheorem{remarks}[theorem]{Remarks}
\newtheorem{ass}[theorem]{Assumption}

\theoremstyle{definition}

\begin{document}
\title{Existence, iteration procedures and directional differentiability for parabolic QVIs}
\author{Amal Alphonse\thanks{Weierstrass Institute, Mohrenstrasse 39, 10117 Berlin, Germany ({\tt alphonse@wias-berlin.de},  {\tt hintermueller@wias-berlin.de}, {\tt rautenberg@wias-berlin.de}).} \and Michael Hinterm\"{u}ller\footnotemark[1]\hspace{0.07cm} \thanks{Department of Mathematics, Humboldt-University of Berlin, Unter den Linden 6, 10099 Berlin, Germany ({\tt hint@math.hu-berlin.de}, {\tt carlos.rautenberg@math.hu-berlin.de}).} 
        \and Carlos N. Rautenberg\footnotemark[1]\hspace{0.07cm} \footnotemark[2]    }
\maketitle

\begin{abstract}We study parabolic quasi-variational inequalities (QVIs) of obstacle type. Under appropriate assumptions on the obstacle mapping, we prove the existence of solutions of such QVIs by two methods: one by time discretisation through elliptic QVIs and the second by iteration through parabolic variational inequalities (VIs). Using these results, we show the directional differentiability (in a certain sense) of the solution map which takes the source term of a parabolic QVI into the set of solutions, and we relate this result to the contingent derivative of the aforementioned map. We finish with an example where the obstacle mapping is given by the inverse of a parabolic differential operator.
\end{abstract}


\section{Introduction}\label{sec:intro}
Quasi-variational inequalities (QVIs) are versatile models that are used to describe many different phenomena from fields as varied as physics, biology, finance and economics.  QVIs were first formulated by Bensoussan and Lions \cite{BensoussanLionsArticle, Lions1973} in the context of stochastic impulse control problems and since then have appeared in many models where nonsmoothness and nonconvexity are present, including superconductivity \cite{KunzeRodrigues,BarrettPrigozhinSuperconductivity,Prigozhin,MR3335194,MR1765540}, the formation and growth of sandpiles \cite{PrigozhinSandpile,BarrettPrigozhinSandpile,Prigozhin1996,Prigozhin1994,MR3335194} and networks of lakes and rivers \cite{Prigozhin1996,Prigozhin1994,MR3231973}, generalised Nash equilibrium problems \cite{HARKER199181, Facchinei2007, Pang2005}, and more recently in thermoforming \cite{AHR}. 

In general, QVIs are more complicated than variational inequalities (VIs) because solutions are sought in a constraint set which depends on the solution itself. This is an additional source of nonlinearity and nonsmoothness and creates considerable issues in the analysis and development of solution algorithms to QVI and the associated optimal control problems.

We focus in this work on parabolic QVIs with constraint sets of obstacle type and we address the issues of existence of solutions and directional differentiability for the solution map taking the source term of the QVI into the set of solutions. This extends to the parabolic setting our previous work \cite{AHR} where we provided a differentiability result for solution mappings associated to elliptic QVIs. Literature for evolutionary QVIs is scarce in the infinite-dimensional setting: among the few works available, we refer to \cite{Hintermueller2012} for a study of parabolic QVIs with gradient constraints, \cite{HRSQuasi} for existence and numerical results in the hyperbolic setting,  \cite{ASNSP} for QVIs arising in hydraulics, \cite{MR1677798} for state-dependent sweeping processes, and evolutionary superconductivity models in \cite{MR1765540}, as well as the work \cite{MR0415069}.  Differentiability analysis for parabolic (non-quasi) VIs was studied in \cite{Jarusek} and \cite{Christof}.

Let us now enter into the specifics of our setting. Let $V \subset H \subset V^*$ be a Gelfand triple of separable Hilbert spaces with $V \ctsCompact H$ a compact embedding. Furthermore, we assume that there exists an ordering to elements of $H$ via a closed convex cone $H_+$ satisfying $H_+ = \{h \in H : (h,g) \geq 0 \quad \forall g \in H_+\}$; the ordering then is $\psi_1 \leq \psi_2$ if and only if $\psi_2-\psi_1 \in H_+$. An example to have in mind is $H=L^2(\Omega)$ with $H_+$ the set of almost everywhere non-negative functions in $H$. This also induces an ordering  for $V$ and $V^*$ as well as for the associated Bochner spaces $L^2(0,T;H)$, $L^2(0,T;V)$ and so on. We define $V_+ := \{ v \in V : v \geq 0\}$. We write $v^+$ for the orthogonal projection of $v \in H$ onto the space $H_+$ and we have the decomposition $v = v^+ - v^-$. Suppose that $v \in V$ implies that $v^+ \in V$ and that there exists a constant $C >0$ such that for all $v \in V$,
\[\norm{v^+}{V} \leq C\norm{v}{V}.\]
An example of such a space $V$ is $V=W^{1,p}(\Omega)$ for $1 \leq p \leq \infty$;  we refer to \cite{Adams} for a definition of the Sobolev space $W^{1,p}(\Omega)$ over a domain $\Omega$.

Let $A\colon V \to V^*$ be a linear, symmetric, bounded and coercive operator which is T-monotone, which means that 
\[\langle Av^+, v^- \rangle_{V^*,V} \leq 0 \text{ for all $v \in V$,}\]
and let $\Phi\colon L^2(0,T;H) \to L^2(0,T;V)$ be a mapping which is increasing, i.e., 
\[\text{if $\psi_1 \leq \psi_2$, then $\Phi(\psi_1) \leq \Phi(\psi_2)$}.\]
Further assumptions will be introduced later as and when required. We consider parabolic QVIs of the following form.

\medskip

\noindent\textbf{QVI Problem}: Given $f \in L^2(0,T;H)$ and $z_0 \in V$, find $z  \in L^2(0,T;V)$ with $z' \in L^2(0,T;V^*)$ such that for all $v \in L^2(0,T;V)$ with $v(t) \leq \Phi(z)(t)$,
\begin{equation}\label{eq:qvi}
\begin{aligned}
z(t) \leq \Phi(z)(t) : \int_0^T \langle z'(t) + Az(t) - f(t), z(t) - v(t) \rangle_{V^*,V} &\leq 0,\\
z(0) &= z_0.
\end{aligned}
\end{equation}
We write the solution mapping taking the source term into the (weak or strong) solution as $\mathbf{P}_{z_0}$ so that \eqref{eq:qvi} reads $z \in \mathbf{P}_{z_0}(f)$ (sometimes we omit the subscript). In this paper, we contribute three main results associated to this QVI:
\begin{itemize}
\item \textit{Existence of solutions to \eqref{eq:qvi} via time-discretisation} (Theorem \ref{thm:discretisationExistence}): we show that solutions to \eqref{eq:qvi} can be formulated as the limit of a sequence constructed from considering time-discretised elliptic QVI problems. This result makes use of the theory of sub- and supersolutions and the Tartar--Birkhoff fixed point method.
\item \textit{Approximation of solution to \eqref{eq:qvi}  by solutions of parabolic VIs} (Theorem \ref{thm:approximationOfQVIByVIIterates}): we define an iterative sequence of solutions of parabolic VIs and show that the sequence converges in a monotone fashion to a solution of the parabolic QVI; this is another QVI existence result. The existence for the aforementioned parabolic VIs comes from either Theorem \ref{thm:discretisationExistence} or from a certain result of Brezis (which will be given in the relevant section), giving rise to two different sets of assumptions under which the theorem holds.
\item \textit{Directional differentiability for the source-to-solution mapping $\mathbf{P}$} (Theorem \ref{thm:directionalDifferentiability}): we prove that the map $\mathbf{P}$ is directionally differentiable in a certain sense using Theorem \ref{thm:approximationOfQVIByVIIterates} and some technical lemmas related to the expansion formulae for parabolic VI solution mappings.
\end{itemize}
Thus we give two existence results and a differential sensitivity result. It should be noted that the differentiability result essentially gives a characterisation of the contingent derivative (a concept frequently used in set-valued analysis) of $\mathbf{P}$ (between appropriate spaces) in terms of a parabolic QVI; see Proposition \ref{prop:contingentDerivative} for details.
\subsection{Notations and layout of paper}
We shall usually write the duality pairing between $V$ and $V^*$ as $\langle \cdot, \cdot \rangle$ rather than $\langle \cdot, \cdot, \rangle_{V^*,V}$ for ease of reading. We use the notations $\cts$ and $\ctsCompact$ to represent continuous and compact embeddings respectively. Let us define the Sobolev--Bochner spaces
\begin{align*}
W(0,T) &:= L^2(0,T;V) \cap H^1(0,T;V^*),\\
W_s(0,T) &:= L^2(0,T;V) \cap H^1(0,T;H),
\end{align*}
and defining the linear parabolic operator ${L}\colon W(0,T) \to L^2(0,T;V^*)$ by ${L}v := v'+ Av$, we also define the following space
\[W_{\mathrm r}(0,T) := \{ w \in W(0,T) : {L}w \in L^2(0,T;H)\}.\]
Note the relationships $W_s(0,T) \cts W(0,T)$ and $W_{\mathrm r}(0,T) \subset W(0,T)$.

If $z \in W(0,T)$ satisfies \eqref{eq:qvi}, we say that it is a \textit{weak solution} or simply a \textit{solution}. A weaker notion of solution is given by transferring the time regularity of the solution onto the test function and requiring only $z \in L^2(0,T;V) \cap L^\infty(0,T;H)$ to satisfy 
\begin{equation}\label{eq:qviWeak}
\begin{aligned}
z(t) \leq \Phi(z)(t) : \int_0^T \langle v'(t) + Az(t) - f(t), z(t) - v(t) \rangle &\leq 0 \quad \forall v \in W(0,T) : v(t) \leq \Phi(z)(t)\\
v(0) &= z_0
\end{aligned}
\end{equation}
and we call $z$ a \textit{very weak solution}. Note that the initial data also has been transferred onto the test function (indeed, the weak solution is not sufficiently regular to have a prescribed initial data).

The paper is structured as follows. In \S \ref{sec:discretisation} we consider the existence of solutions to \eqref{eq:qvi} via the method of time discretisation: we characterise a solution of the parabolic QVI as the limit of solutions of elliptic QVIs. In \S \ref{sec:approximations}, we approximate solutions of the QVI by a sequence of solutions of parabolic VIs that are defined iteratively. In \S \ref{sec:expansion} we consider parabolic VIs and directional differentiability with respect to perturbations in the obstacle and we make use of this in \S \ref{sec:DDMain} to prove that the source-to-solution map $\mathbf{P}$ is directionally differentiable in a particular sense. We highlight some possible alternative approaches in \S \ref{sec:otherApproaches} and finish with an example in \S \ref{sec:example}.


\section{Existence for parabolic QVIs through time discretisation}\label{sec:discretisation}
We prove existence to \eqref{eq:qvi} by the method of time discretisation via elliptic QVIs of obstacle type. This is in contrast to \cite{HRSQuasi} where the discretisation for evolutionary QVI was done in such a way as to yield a sequence of elliptic VIs, and as far as we are aware, our approach is novel in these type of problems.

We make the following basic assumption.
\begin{ass}\label{ass:OnPhiForTimeDisc}
Let $\Phi(0) \geq 0$ and
\begin{equation}
\Phi\colon H \to C^0([0,T];V)\label{ass:PhiVInC0V}.
\end{equation} 
\end{ass}
Let $N \in \mathbb{N}$, $h^N:=T/N$ and for $n=0, 1, ..., N$, $t_n^N := nh^N$. This divides the interval $[0,T]$ into $N$ subintervals of length $h^N$; we will usually write $h$ for $h^N$. We approximate the source term by 
\[f_n^N := \frac 1h \int_{t_{n-1}^N}^{t_n^N} f(t)\;\mathrm{d}t\]
and we consider the following elliptic QVI problem.

\medskip

\noindent \textbf{Discretised problem}: Given $z_0^N:=z_0$, find $z_n^N \in V$ such that
\begin{equation}\label{eq:approximatingQVI}
\begin{aligned}
z_n^N \leq \Phi(z_n^N)(t_{n-1}^N)  : \left\langle \frac{z_n^N - z_{n-1}^N}{h} + Az_n^N - f_n^N, z_n^N - v\right\rangle \leq 0 \quad \forall v \in V: v \leq \Phi(z_n^N)(t_{n-1}^N).
\end{aligned}
\end{equation}
This problem is sensible since by \eqref{ass:PhiVInC0V}, for fixed $t$, the mapping $v \mapsto \Phi(v)(t)$ is well defined from $H$ into $V$, since we may consider $H \subset L^2(0,T;H)$ (elements of $H$ can be thought of constant-in-time elements of $L^2(0,T;H)$), and $\Phi(v)$ can be evaluated pointwise in time. We consider first the existence of solutions to \eqref{eq:approximatingQVI}.
\subsection{Existence and uniform estimates for the elliptic approximations}
The inequality \eqref{eq:approximatingQVI} can be rewritten as
\begin{equation}\label{eq:approximatingQVI2}
\begin{aligned}
\left\langle z_n^N + hAz_n^N - hf_n^N- z_{n-1}^N , z_n^N - v\right\rangle \leq 0 \quad \forall v \leq \Phi(z_n^N)(t_{n-1}^N).
\end{aligned}
\end{equation}
We write the solution of \eqref{eq:approximatingQVI} or \eqref{eq:approximatingQVI2} as $\mathbf{Q}_{t_{n-1}^N}(hf_n^N + z_{n-1}^N) \ni z_n^N$. 
Related to \eqref{eq:approximatingQVI2} is the following VI:
\[ \text{Find } v \in V, v \leq \Phi(\psi)(t) : \langle v + hAv - g, v-\varphi \rangle \leq 0 \qquad \forall v \in V : v \leq \Phi(\psi)(t);\]
denote by $E_{t,h}(g,\psi)=v$ the solution of this problem. For fixed $t$ and $h$ and sufficiently smooth data, this mapping is well defined since the obstacle $\Phi(\psi)(t) \in V$. 
To ease notation, we write $E_{t_n^N}(g,\psi)$ instead of $E_{t_n^N, h^N}(g,\psi)$ (where $h^N$ is the step size) because specifying $h^N$ is redundant. 

Let us make an observation which follows from the theory of Birkhoff--Tartar \cite{tartar1974inequations, Birkhoff}. Suppose there exists a subsolution $z_{\mathrm{sub}}$ and a supersolution $z^{sup}$ to \eqref{eq:approximatingQVI}, i.e., $z_{\mathrm{sub}} \leq E_{t_{n-1}^N}(hf_n^N+z_{n-1}^N, z_{\mathrm{sub}})$ and $z^{sup} \geq E_{t_{n-1}^N}(hf_n^N+z_{n-1}^N, z^{sup})$. Then the QVI problem \eqref{eq:approximatingQVI} has a solution $z_n^N = E_{t_{n-1}^N}(hf_n^N+z_{n-1}^N, z_n^N) \in [z_{\mathrm{sub}},z^{sup}]$. The next lemma applies this idea.
\begin{remark}
In fact, the theory of Birkhoff and Tartar tells us not only that there exist in general multiple solutions in $[z_{\mathrm{sub}},z^{sup}]$ but also that there exist minimal and maximal solutions (a \emph{minimal solution} $z_{m}$ is a solution that satisfies $z_m \leq z$ for every solution $z$; maximal solutions are defined with the opposite inequality). 
\end{remark}
%
First, let us set $A_h w := w + hAw$ (this $A_h$ is the elliptic operator appearing in \eqref{eq:approximatingQVI2}) and 
define $\bar z_{n,N}$ as the solution of the following elliptic PDE:
\[A_h \bar z_{n,N} = hf_n^N + z_{n-1}^N.\]
\begin{lem}\label{lem:unIncreasing}Suppose that
\begin{align}
&\text{$f \geq 0$ is increasing, $z_0 \geq 0$}\tag{D1}\label{ass:dataNonnegativeAndfInc},\\
&\text{$z_0 \leq \Phi(z_0)(0)$ and $\langle Az_0 - f(t), v\rangle \leq 0$ for all $v \in V_+$ and a.e. $t$}\tag{D2}\label{ass:initialDataNewAss},\\
&t \mapsto \Phi(v)(t) \text{ is increasing for all $v \in V_+$}\label{eq:PhiIncreasingInTimeForNonNegV}.
\end{align}
Then the approximate problem \eqref{eq:approximatingQVI} has a non-negative solution $z_n^N \in [z_{n-1}^N,\bar z_{n,N}]$. Thus the sequence $\{z_n^N\}_{n \in \mathbb{N}}$ is increasing.
\end{lem}
\begin{proof}
We first show that $z_0^N=z_0$ is a subsolution for the QVI for $z_1^N$. Indeed, let $w := E_{t_0^N}(hf_1^N + z_0^N, z_0)$ which satisfies
\[w \leq \Phi(z_0)(0) : \langle w-z_0 + hAw - hf_1^N, w- v \rangle \leq 0 \quad \forall v \in V : v \leq \Phi(z_0)(0).\]
The choice $v = w + (z_0-w)^+$ is a feasible test function due to the upper bound on the initial data from assumption \eqref{ass:initialDataNewAss}. This leads to
\begin{align*}
(z_0-w, (z_0-w)^+)_H &\leq h\langle Aw - f_1^N, (z_0-w)^+\rangle\\
&= h\langle Aw - Az_0 + Az_0 - f_1^N, (z_0-w)^+\rangle\\
&\leq h\langle Az_0 -f_1^N, (z_0-w)^+ \rangle \tag{by T-monotonicity of $A$}\\
&= h\langle Az_0 - \frac 1h\int_0^{t_1^N}f(s)\;\mathrm{d}s, (z_0-w)^+\rangle\\
&= \int_0^{t_1^N}\langle Az_0 - f(s), (z_0-w)^+\rangle\;\mathrm{d}s\\
&\leq 0,
\end{align*}
again by assumption \eqref{ass:initialDataNewAss}. This shows that $z_0 \leq E_{t_0^N}(hf_1^N+z_0^N, z_0)$ is indeed a subsolution.

The function $\bar z_{1,N}$ defined through $A_h \bar z_{1,N} = hf_1^N + z_0$
is a supersolution since $\bar z_{1,N}=E_{t_{n-1}^N}(hf_1^N + z_0, \infty) \geq E_{t_{n-1}^N}(hf_1^N + z_0, \bar z_{1,N})$ (thanks to the fact that $\Phi$ is increasing). Then we apply the theorem of Birkhoff--Tartar to obtain existence of $z_1^N \in V$ with $z_1^N \in [z_0^N,\bar z_{1,N}]$. 

Suppose that $z_n^N \in [z_{n-1}^N, \bar z_n^N] \cap \mathbf{Q}_{t_{n-1}^N}(f_n^N + z_{n-1}^N$). Observe that
\[z_n^N = E_{t_{n-1}^N}(f_n^N + z_{n-1}^N, z_n^N) \leq E_{t_{n-1}^N}(f_{n+1}^N + z_n^N, z_n^N) \leq E_{t_{n}^N}(f_{n+1}^N + z_n^N, z_n^N)\]
with the final inequality because the obstacle associated to $t_{n}^N$ is greater than or equal to the obstacle associated to $t_{n-1}^N$ by assumption \eqref{eq:PhiIncreasingInTimeForNonNegV}. We also have
\[\bar z_{n+1}^N = E_{t_{n}^N}(f_{n+1}^N + z_n^N, \infty) \geq E_{t_{n}^N}(f_{n+1}^N + z_n^N, \bar z_{n+1}^N)\]
that is, $z_n^N$ and $\bar z_{n+1}^N$ are sub- and super-solutions respectively for $\mathbf{Q}_{t_{n}^N}(f_{n+1}^N + z_n^N)$ (and the supersolution is greater than the subsolution). Therefore, by Birkhoff--Tartar there exists a $z_{n+1}^N \in \mathbf{Q}_{t_{n}^N}(f_{n+1}^N + z_n^N)$ in the interval $[z_n^N, \bar z_{n+1}^N]$. 
\end{proof}
It appears that we may select \textit{any} solution $z^N_n$ as given by the Birkhoff--Tartar theorem in the previous lemma, regardless of how we choose $z^N_{n-1}$. For example, we may choose $z^N_{n-1}$ to be the minimal solution on its corresponding interval (with endpoints given by the sub- and supersolution) and $z^N_n$ to be the maximal solution on its corresponding interval, with no effect on the resulting analysis (though of course, different choices may lead to different solutions of the original parabolic QVI in question in the end).

We now obtain some bounds on the sequence $\{z_n^N\}$. For this, we use the fact that if $f \in L^2(0,T;H)$, then 
\begin{align}\label{eq:boundednessOfFn}
\sum_{n=1}^N h\norm{f_n^N}{H}^2 \leq \norm{f}{L^2(0,T;H)}^2.
\end{align}
\begin{lem}\label{lem:boundsOnunN}
Under the assumptions of Lemma \ref{lem:unIncreasing}, the following uniform bounds hold: 
\begin{align}
\norm{z_n^N}{H} &\leq C,\label{eq:bdInH}\\
h\sum_{i=1}^n\norm{z_i^N}{V}^2 &\leq C,\label{eq:bdinhSumV}\\
\frac{1}{h}\sum_{i=1}^n\norm{z_i^N-z_{i-1}^N}{H}^2 &\leq C\label{eq:bdhinvDifferencesInH}.
\end{align}
(The final bound needs symmetry of $A$ and the increasing property of the sequence $\{z^N_n\}_{n \in \mathbb{N}}$).
\end{lem}
\begin{proof}In this proof, we omit the superscript $N$ in various quantities for clarity.

We follow the argumentation of \cite[\S 6.3.3]{Glowinski1981}. Testing the QVI \eqref{eq:approximatingQVI2} with $v=0$ (which is valid since $0 \leq \Phi(0)(t_{n-1}) \leq \Phi(z_n)(t_{n-1})$ by assumption on the non-negativity at zero and the second inequality by the increasing property of $\Phi$ and the fact that $z_n \geq 0$) gives
\[\left\langle z_n^N - z_{n-1}^N  + hAz_n^N - hf_n^N, z_n^N \right\rangle \leq 0,\]
 and using the relation 
\begin{equation*}
(a-b,a)_H = \frac 12 \norm{a}{H}^2 - \frac 12 \norm{b}{H}^2 + \frac 12 \norm{a-b}{H}^2,
\end{equation*}
we find
\begin{align*}
\frac{1}{2}\left(\norm{z_n}{H}^2 - \norm{z_{n-1}}{H}^2 + \norm{z_n-z_{n-1}}{H}^2 \right) + hC_a\norm{z_n}{V}^2 &\leq h\langle f_n, z_n \rangle\\
&\leq h\norm{f_n}{V^*}\norm{z_n}{V}\\
&\leq \frac{h}{2C_a}\norm{f_n}{V^*}^2 + \frac{hC_a}{2}\norm{z_n}{V}^2
\end{align*}
with the last line due to Young's inequality. This leads to
\begin{align*}
\norm{z_n}{H}^2 - \norm{z_{n-1}}{H}^2 + \norm{z_n-z_{n-1}}{H}^2 + hC_a\norm{z_n}{V}^2 
&\leq \frac{h}{C_a}\norm{f_n}{V^*}^2,
\end{align*}
whence, summing up between $n=1$ and $n=m$ for some $m \in \mathbb{N}$, and using \eqref{eq:boundednessOfFn},
\begin{align*}
\norm{z_m}{H}^2 - \norm{z_0}{H}^2 + \sum_{n=1}^m \norm{z_n-z_{n-1}}{H}^2 + hC_a\sum_{n=1}^m\norm{z_n}{V}^2 
&\leq C.
\end{align*}
This leads to the first two bounds stated in the lemma. 
For the final bound, testing \eqref{eq:approximatingQVI} with $z_{n-1}$, which is valid since by Lemma \ref{lem:unIncreasing}, $z_{n-1} \leq z_n  \leq \Phi(z_n)(t_{n-1})$, we find
\[\left\langle \frac{z_n - z_{n-1}}{h} + Az_n - f_n, z_n - z_{n-1}\right\rangle \leq 0, \]
which leads to
\begin{align*}
\frac 1h\norm{z_n-z_{n-1}}{H}^2 + \langle Az_n, z_n - z_{n-1} \rangle 
&\leq  \frac{h}{2}\norm{f_n}{H}^2 + \frac{\norm{z_n-z_{n-1}}{H}^2}{2h},
\end{align*}
and here we use
\begin{align*}
&\langle Az_n, z_n - z_{n-1} \rangle\\
 &= \frac 12 \langle Az_n-Az_{n-1}, z_n - z_{n-1}\rangle  + \frac 12 \langle Az_n-Az_{n-1}, z_n - z_{n-1} \rangle + \langle Az_{n-1}, z_n - z_{n-1} \rangle\\
&= \frac 12 \langle Az_n-Az_{n-1}, z_n - z_{n-1}\rangle  + \frac 12 \langle Az_n, z_n \rangle + \frac 12 \langle Az_{n-1}, z_{n-1} \rangle - \langle Az_n, z_{n-1} \rangle + \langle Az_{n-1}, z_n  \rangle\\
&\quad - \langle Az_{n-1},  z_{n-1} \rangle\\
&= \frac 12 \langle Az_n-Az_{n-1}, z_n - z_{n-1}\rangle  + \frac 12 \langle Az_n, z_n \rangle - \frac 12 \langle Az_{n-1}, z_{n-1} \rangle  
\end{align*}
to get
\begin{align*}
\frac{1}{h}\norm{z_n-z_{n-1}}{H}^2 + \langle Az_n-Az_{n-1}, z_n - z_{n-1}\rangle  + \langle Az_n, z_n \rangle - \langle Az_{n-1}, z_{n-1} \rangle  &\leq  h\norm{f_n}{H}^2.
\end{align*}
Neglecting the second term and summing this up from $n=1$ to $n=m$ and using \eqref{eq:boundednessOfFn}, we obtain
\begin{align*}
\frac{1}{h}\sum_{n=1}^m\norm{z_n-z_{n-1}}{H}^2 + \langle Az_m, z_m \rangle  
&\leq \langle Az_0, z_0 \rangle + C
\end{align*}
as desired.
\end{proof}
%
%

\subsection{Interpolants}
We define the piecewise constant interpolants
\[z^N(t) := \sum_{n=1}^{N}z_n^N\chi_{[t^N_{n-1}, t^N_{n})}(t) \quad \text{and}\quad z^N_-(t) := \sum_{n=1}^{N}z_{n-1}^N\chi_{[t^N_{n-1}, t^N_{n})}(t),\]
where $\chi_A$ represents the characteristic function on the set $A$.
In order to ease the presentation, we often use the notation $T_n^{N}:= [t_n^N, t_{n+1}^N).$ 
\begin{cor}\label{cor:interpolantBounds}Under the assumptions of Lemma \ref{lem:boundsOnunN}, $\{z^N\}$ and $\{z^N_-\}$ are bounded uniformly in $L^2(0,T;V) \cap L^\infty(0,T;H)$.
\end{cor}
\begin{proof}
Since the $T_n^N$ are disjoint, we see that
\begin{align*}
\norm{z^N}{L^\infty(0,T;H)} 
= \esssup_{t \in [0,T]} \sum_n \norm{z^N_n}{H}\chi_{T^N_{n-1}}(t)
\leq C\esssup_{t \in [0,T]} \sum_n \chi_{T^N_{n-1}}(t)
= C 
\end{align*}
by \eqref{eq:bdInH}, and
\begin{align*}
\norm{z^N}{L^2(0,T;V)}^2 &= \int_0^T \norm{\sum_{n=1}^N z^N_n \chi_{T^N_{n-1}}(t)}{V}^2 
= \sum_{n=1}^N \int_{T^N_{n-1}} \norm{z^N_n}{V}^2 
= h\sum_{n=1}^N\norm{z^N_n}{V}^2
\leq C
\end{align*}
by \eqref{eq:bdinhSumV}. A similar argument leads to the bounds on $z^N_-$.
\end{proof}
To be able to handle the time derivative, it is useful to construct the interpolant 
\begin{align*}
\hat z^N(t) &:= z_0 + \int_0^t \sum_{n=1}^N \frac{z_n^N - z_{n-1}^N}{h}\chi_{[t^N_{n-1}, t^N_n)}(s)\;\mathrm{d}s\\
&= z_{n-1}^N + \frac{z_n^N - z_{n-1}^N}{h}(t-t^N_{n-1})\quad \text{if $t \in [t^N_{n-1},t^N_n)$},
\end{align*}
known as Rothe's function, which also has the time derivative 
\[\partial_t \hat z^N(t) = \frac{z_n^N - z_{n-1}^N}{h}\quad \text{if $t \in [t^N_{n-1},t^N_n)$}.\]
\begin{cor}\label{cor:boundOnDiscTimeDeriv}
Under the assumptions of Lemma \ref{lem:boundsOnunN}, $\{\hat z^N\}$ is bounded uniformly in $W_s(0,T)$. 
\end{cor}
\begin{proof}
We see that
\begin{align*}
\norm{\hat z^N}{L^2(0,T;V)}^2 &= \sum_{n=1}^N \int_{t^N_{n-1}}^{t^N_n}\norm{\hat z^N(t)}{V}^2\\ 
&\leq 2\sum_{n=1}^N\int_{t^N_{n-1}}^{t^N_n} \lVert z_{n-1}^N\rVert_V^2 +  \frac{2}{h^2} \sum_{n=1}^N\int_{t^N_{n-1}}^{t^N_n}(t-t^N_{n-1})^2\lVert z_n^N - z_{n-1}^N\rVert_V^2\\
&= 2h\sum_{n=1}^N\lVert z_{n-1}^N\rVert_V^2  +\frac{2}{3h^2} \sum_{n=1}^N[(t-t^N_{n-1})^3]^{t^N_n}_{t^N_{n-1}}\lVert z_n^N - z_{n-1}^N\rVert_V^2\\
&\leq C_1 + \frac{2}{3h^2} \sum_{n=1}^N(t^N_n-t^N_{n-1})^3\lVert z_n^N - z_{n-1}^N\rVert_V^2\\
&=C_1 + \frac{2h}{3} \sum_{n=1}^N\lVert {z_n^N - z_{n-1}^N}\rVert_V^2\\
&\leq C_2,
\end{align*}
with the last two inequalities by \eqref{eq:bdinhSumV}. Regarding the time derivative, using \eqref{eq:bdhinvDifferencesInH}, we find
\begin{align*}
\norm{\partial_t \hat z^N}{L^2(0,T;H)}^2 
&= \sum_{n=1}^N \int_{T^N_{n-1}} \norm{\partial_t z^N(t)}{H}^2\\
&= \sum_{n=1}^N \int_{T^N_{n-1}} \norm{\frac{z_n^N-z_{n-1}^N}{h}}{H}^2\\
&= \frac{1}{h}\sum_{n=1}^N \norm{z_n^N-z_{n-1}^N}{H}^2\\
&\leq C_3.
\end{align*}
\end{proof}

\subsection{Passing to the limit in the interpolants}
By the previous subsection, we have the existence of $z, \hat z$ such that, for a relabelled subsequence, the following convergences hold:
\begin{equation}\label{eq:listOfConvergences}
\begin{aligned}
z^N &\weaklystar z \quad &&\text{in $L^\infty(0,T;H)$},\\
z^N &\weaklyto z &&\text{in $L^2(0,T;V)$},\\
\hat z^N &\weaklyto \hat z &&\text{in $W_s(0,T)$}.\\
\end{aligned}
\end{equation}
\begin{lem}\label{lem:identificationOfHatuAndu}We have $\hat z \equiv z$. Furthermore, 
\[z^N_- \weaklyto z \quad\text{in $L^2(0,T;V)$ and weakly-star in $L^\infty(0,T;H)$}.\]
\end{lem}
\begin{proof}
Observe that
\begin{align*}
\int_0^T \norm{z^N_-(t)-z^N(t)}{H}^2 &= \sum_n \int_{T_{n-1}^N}\norm{z^N_{n-1}-z^N_n}{H}^2 = h\sum_n \norm{z^N_{n-1}-z^N_n}{H}^2 \leq Ch^2
\end{align*}
by \eqref{eq:bdhinvDifferencesInH}. Thus as $N \to \infty$, $z^N_- - z^N \to 0$ in $L^2(0,T;H)$. Since $z^N \weaklyto z$ in $L^2(0,T;V)$, $z^N \to z$ in $L^2(0,T;H)$ and we obtain $z^N_- \to z$ in $L^2(0,T;H)$ and weakly in $L^2(0,T;V)$.

Now consider
\begin{align*}
\int_0^T \norm{z^N_-(t)-\hat z^N(t)}{H}^2 &= \int_0^T \norm{\sum_n\chi_{[t_{n-1},t_n)}(t)(t-{t_{n-1}})\frac{z_n^N-z_{n-1}^N}{h}}{H}^2\\
&= \sum_n\int_{t_{n-1}}^{t_n} \norm{\chi_{[t_{n-1},t_n)}(t)(t-{t_{n-1}})\frac{z_n^N-z_{n-1}^N}{h}}{H}^2\\
&\leq \frac{1}{h^2}\sum_n\int_{t_{n-1}}^{t_n} (t-{t_{n-1}})^2\norm{z_n^N-z_{n-1}^N}{H}^2\\
&= \frac{h}{3}\sum_n\norm{z_n^N-z_{n-1}^N}{H}^2\tag{see the proof of Corollary \ref{cor:boundOnDiscTimeDeriv}}\\
&\leq Ch^2
\end{align*}
with the final line by \eqref{eq:bdhinvDifferencesInH}. This shows that $z^N_- - \hat z^N \to 0$ in $L^2(0,T;H)$, allowing us to identify $\hat z = z$ as desired.
%
\end{proof}
The convergence results above obviously imply that $\hat z^N \to z$ in $C^0([0,T];H)$ (due to Aubin--Lions), so that $z_0 = \hat z^N(0) \to z(0)$, i.e., $z$ has the right initial data. Let us now see that $z$ is feasible.
\begin{lem}[Feasibility of the limit]Let the following conditions hold: 
\begin{align}
&\text{$\{v_n\} \subset V_+$} \implies \textstyle \sum_{n=1}^N \Phi(v_n)(t)\chi_{T_{n-1}}(t) \leq \Phi\left(\textstyle \sum_{n=1}^N v_n\chi_{T_{n-1}}(\cdot)\right)(t)\label{ass:forFeasibility},\\
&\nonumber \text{$v_n \weaklyto v$ in $L^2(0,T;V)$ and weakly-* in $L^\infty(0,T;H)$ with $v_n(t) \leq \Phi(v_{n})(t)$}\\
&\quad\quad\quad\quad\quad\quad\quad\quad\quad\quad\quad\quad\quad\quad\quad\quad\quad\quad\quad\quad\quad\quad\quad\quad\implies v(t) \leq \Phi(v)(t)\label{ass:ensuresFeasibility2}.
\end{align}
Then $z(t) \leq \Phi(z)(t)$ for a.e. $t$. 
\end{lem}
\begin{proof}
Since by \eqref{eq:PhiIncreasingInTimeForNonNegV}, for $t \in T_{n-1}^N$, $z^N_n \leq \Phi(z^N_n)(t^N_{n-1}) \leq \Phi(z^N_n)(t)$, we have
\begin{align*}
z^N(t) 
\leq \sum_{n=1}^N \Phi(z^N_n)(t)\chi_{[t^N_{n-1},t^N_n)}(t).
\end{align*}
Using the assumption \eqref{ass:forFeasibility} applied to the right-hand side above, we have
\begin{align*}
z^N(t) &\leq \Phi\left(\sum_{n=1}^N z^N_n\chi_{[t^N_{n-1},t^N_n)}(\cdot)\right)(t) = \Phi(z^N)(t).
\end{align*}
Passing to the limit here using \eqref{ass:ensuresFeasibility2} gives the result.
\end{proof}
Regarding the assumptions of the previous lemma, \eqref{ass:ensuresFeasibility2} is a mild continuity requirement on $\Phi$ whereas for \eqref{ass:forFeasibility}, consider the superposition case $\Phi(v)(t) := \hat \Phi(t,v(t))$. Then if $\{v_n\}_{n \in \mathbb{N}} \subset V$, we have
\begin{align*}
\sum_n \Phi(v_n)(t)\chi_{T_{n-1}^N}(t) &= \sum_n \hat \Phi(t,v_n)\chi_{T_{n-1}^N}(t),
\end{align*}
and now supposing $t \in T_{j-1}^N$ for some $j$, this becomes
\begin{align*}
\sum_n \Phi(v_n)(t)\chi_{T_{n-1}^N}(t) &=  \hat\Phi(t,v_j)
=\hat\Phi\left(t,\sum_n v_n\chi_{T_{n-1}^N}(t)\right)
= \Phi\left(\sum_n v_n\chi_{T_{n-1}^N}\right)(t)
\end{align*}
and since $j$ is arbitrary, this holds for all $t$. Hence \eqref{ass:forFeasibility} holds with equality.

In order to pass to the limit in the inequality satisfied by the interpolant $z^N$, we have to be able to approximate test functions in the limiting constraint set. This is possible as the next lemma shows.
\begin{lem}[Recovery sequence]\label{lem:recoverySequence}
Assume the condition 
\begin{align}
\nonumber &\forall \epsilon > 0, \{w_N\} : w_N \weaklyto w \text{ in } L^2(0,T;V) \text{ and weakly-$*$ in }L^\infty(0,T;H), \exists N_0 \in \mathbb{N} : \\
&\quad\quad N \geq N_0 \implies  \sum_{n=1}^N\int_{T_{n-1}^N}\norm{\Phi(w_N(t))(t_{n-1}^N) - \Phi(w)(t)}{V}^2 \leq \epsilon.\label{ass:CA}
\end{align}
Then for every $v \in L^2(0,T;V)$ with $v(t) \leq \Phi(z)(t)$, there exists a $v^N \in L^2(0,T;V)$ such that 
\begin{align*}
&v^N|_{T^N_{n-1}} \leq \Phi(z_n^N)(t_{n-1}^N)\\
&v^N \to v \quad \text{in $L^2(0,T;V)$}.
\end{align*}
\end{lem}
\begin{proof}
Let $v \in L^2(0,T;V)$ with $v(t) \leq \Phi(z)(t)$ and define
\[v^N_n(t) := v(t) + \Phi(z_n^N)(t_{n-1}^N) - \Phi(z)(t)\] which satisfies $v_n^N(t) \leq \Phi(z_n^N)(t_{n-1}^N)$  and set
\[v^N(t) := \sum_{n=1}^N \chi_{T_{n-1}}(t) (v(t) + \Phi(z_n^N)(t_{n-1}^N) - \Phi(z)(t)).\]
Take $\epsilon > 0$. We have
\begin{align*}
\int_0^T \norm{v^N(t)-v(t)}{V}^2 &= \sum_{n=1}^N\int_{T_{n-1}^N}\norm{\Phi(z_n^N)(t_{n-1}^N) - \Phi(z)(t)}{V}^2\\
 &= \sum_{n=1}^N\int_{T_{n-1}^N}\norm{\Phi(z^N(t))(t_{n-1}^N) - \Phi(z)(t)}{V}^2\\
 &\leq \epsilon
\end{align*}
by assumption \eqref{ass:CA} as long as $N \geq N_0$. 
%
%
This shows that $v^N \to v$.
\end{proof}
Let us consider two cases under which the assumption \eqref{ass:CA} of the previous lemma holds.

\medskip

\noindent\textsc{1. Superposition case}. In case where $\Phi(v)(t) := \hat \Phi(v(t))$, \eqref{ass:CA} translates to a complete continuity assumption. Indeed, the sum term in \eqref{ass:CA} is
\begin{align*}
\sum_{n=1}^N\int_{T_{n-1}^N}\norm{\Phi(w_N(t))(t_{n-1}^N) - \Phi(w)(t)}{V}^2 &= \sum_{n=1}^N\int_{T_{n-1}^N}\norm{\hat \Phi(w_N(t)) - \hat\Phi(w(t))}{V}^2\\
&= \int_0^T\norm{\hat \Phi(w_N(t)) - \hat\Phi(w(t))}{V}^2,
\end{align*}
so that \eqref{ass:CA} simply asks for $\Phi(w_N) \to \Phi(w)$ in $L^2(0,T;V)$ whenever $w_N \weaklyto w$ in $L^2(0,T;V)$ and weakly-* in $L^\infty(0,T;H)$.

\medskip

\noindent\textsc{2. VI case}. When $\Phi(v)(t) \equiv \psi(t)$ for some obstacle $\psi$ 
and if $\psi \in C^0([0,T];V)$ then for every $\epsilon > 0$, there exists $\delta > 0$ such that $|t-s| \leq \delta$ implies $\norm{\psi(t)-\psi(s)}{V} \leq \sqrt{\epsilon\slash T}.$  When $t \in T_{n-1}^N$, we have $|t^N_{n-1} - t| \leq |t^N_{n-1}-t^N_n| \leq h^N \to 0$ as $N \to \infty$. So for sufficiently large $N$, say $N \geq N_0$, we have $|t^N_{n-1}-t| \leq \delta$ and thus
\[\epsilon \geq \sum_n \int_{T_{n-1}^N}\norm{\psi(t^N_{n-1})-\psi(t)}{V}^2 = \sum_n \int_{T_{n-1}^N}\norm{\Phi(w_N(t))(t_{n-1}^N) - \Phi(w)(t)}{V}^2\]
and so \eqref{ass:CA} also holds.

\begin{theorem}\label{thm:discretisationExistence}
Let Assumption \ref{ass:OnPhiForTimeDisc}, \eqref{ass:dataNonnegativeAndfInc}, \eqref{ass:initialDataNewAss}, \eqref{eq:PhiIncreasingInTimeForNonNegV}, \eqref{ass:forFeasibility}, \eqref{ass:ensuresFeasibility2} and \eqref{ass:CA} hold. Then there exists a non-negative solution $z \in W_s(0,T)$ to \eqref{eq:qvi} which is the limit of the interpolants $\{z^N\}, \{\hat z^N\}.$ Furthermore, the map $t \mapsto z(t)$ is increasing.
\end{theorem}
\begin{proof}
Let $v \in L^2(0,T;V)$ satisfy $v(t) \leq \Phi(z)(t)$ and let us take $v^N$ as defined in the proof of Lemma \ref{lem:recoverySequence}. Then by \eqref{eq:approximatingQVI},
\begin{align*}
&\int_0^T \langle \partial_t \hat z^N(t) + Az^N(t) - f^N(t), z^N(t) - v^N(t) \rangle\\
&\quad= \sum_{n=1}^N \int_{t_{n-1}^N}^{t_{n}^N} \langle \partial_t \hat z^N(t) + Az^N(t) - f^N(t), z^N(t) - v^N(t) \rangle\\
&\quad= \sum_{n=1}^N \int_{t_{n-1}^N}^{t_{n}^N} \left\langle \frac{z_{n}^N - z_{n-1}^N}{h} + Az^N_n - f^N_n, z^N_n - v^N_n(t)\right\rangle\\
&\quad\leq 0.
\end{align*}
Writing the duality product involving the time derivative as an inner product, we have, using the convergences in \eqref{eq:listOfConvergences} and the weak lower semicontinuity of the bilinear form generated by $A$,
\begin{align*}
0 &\geq 
\int_0^T (\partial_t \hat z^N(t), z^N(t) - v^N(t)) + \langle Az^N(t) - f^N(t), z^N(t) - v^N(t) \rangle\\
&\to
\int_0^T (z'(t), z(t)-v(t))_{H} + \langle Az(t) - f(t), z(t)-v(t) \rangle
\end{align*}
so that $z \in \mathbf{P}(f) \cap W_s(0,T)$. Since $\{z_n^N\}$ are non-negative, it follows that $z$ is too.

By \eqref{eq:listOfConvergences}, it follows that $z^{N_j}(t) \to z(t)$ in $H$ for almost every $t \in [0,T]$. Let now $s \leq r$ and suppose that $s \in T_{m-1}^{N_j}$ and $r \in T_{n-1}^{N_j}$ with $m \leq n$. Then we have $z^{N_j}(s) = z^{N_j}_m \leq z^{N_j}_n = z^{N_j}(t)$ since the sequence $\{z^{N_j}_i\}_{i \in \mathbb{N}}$ is increasing. Passing to the limit, we find for almost every $r$ and $s$ with $s \leq r$ that $z(s) \leq z(r)$, i.e., $t \mapsto z(t)$ is increasing.
\end{proof}

\section{Parabolic VI iterations of parabolic QVIs}\label{sec:approximations}
In this section, we will show the existence of sequences of solutions to VIs that converge to solutions of QVIs. We begin with collecting some facts regarding parabolic VIs.

Consider the parabolic VI
\begin{equation}\label{eq:Pvi}
\begin{aligned}
z(t) \leq \psi(t) : \int_0^T \langle z'(t) + Az(t) - f(t), z(t) - v(t) \rangle &\leq 0 \quad \forall v \in L^2(0,T;V) \text{ s.t. } v(t) \leq \psi(t),\\
z(0) &= z_0.
\end{aligned}
\end{equation}
We write the solution as $z:=\sigma_{z_0}(f, \psi)$ when it exists. Given $f \in L^2(0,T;H)$ and $\psi \in V$ independent of time, the solution $z \in W_s(0,T)$ exists uniquely, see e.g. \cite{MR742624, LionsBensoussan, Bensoussan1974}.

The problem \eqref{eq:Pvi} can be transformed to a parabolic VI with zero initial data with right-hand side $f-Az_0$ and obstacle $\psi-z_0$ with the substitution $u(t)=z(t)-z_0$:
\[\sigma_0(f-Az_0, \psi-z_0) = \sigma_{z_0}(f,\psi) - z_0.\] 
We often write simply $\sigma$ rather than $\sigma_{z_0}$ when we do not need to emphasise the initial data. The next lemma shows that $\sigma$ is increasing in its arguments.
\begin{lem}[I. Comparison principle for parabolic VIs]\label{lem:comparisonPrincipleNoPhi}
Suppose for $i=1, 2$ that $z_i \in W(0,T)$ is a solution of \eqref{eq:Pvi} with data $f_i \in L^2(0,T;V^*)$ and obstacle $\psi_i$ such that $f_1 \leq f_2$ and $\psi_1 \leq \psi_2$. Then $z_1 \leq z_2$. 
\end{lem}
\begin{proof}
The $z_i$ satisfy the inequalities
\begin{equation*}
\begin{aligned}
&z_1(t) \leq \psi_1(t) : \int_0^T \langle z_1'(t) + Az_1(t) - f_1(t), z_1(t) - v_1(t) \rangle \leq 0 ,\\
&z_2(t) \leq \psi_2(t) : \int_0^T \langle z_2'(t) + Az_2(t) - f_2(t), z_2(t) - v_2(t) \rangle \leq 0,
\end{aligned}
\end{equation*}
for all $v_1, v_2 \in L^2(0,T;V)$ such that $v_1(t) \leq \psi_1(t)$ and $v_2(t) \leq \psi_2(t)$. Let us take here $v_1 = z_1 - (z_1-z_2)^+$, which is clearly a feasible test function, and take $v_2= z_2 + (z_1-z_2)^+$ which is also feasible since 
\begin{align*}
v_2 \leq \begin{cases}
z_2 \leq \psi_2 &: \text{if $z_1 \leq z_2$},\\
z_1 \leq \psi_1 \leq \psi_2 &: \text{if $z_1 \geq z_2$}.
\end{cases}
\end{align*}
This gives us
\begin{align*}
\int_0^T \langle z_1'(t) + Az_1(t) - f_1(t), (z_1(t) - z_2(t))^+ \rangle &\leq 0,\\
\int_0^T \langle z_2'(t) + Az_2(t) - f_2(t), -(z_1(t) - z_2(t))^+ \rangle &\leq 0 .
\end{align*}
Adding yields
\[\int_0^T \langle (z_1-z_2)'(t) + Az_1(t)-Az_2(t), (z_1(t) - z_2(t))^+ \rangle \leq \int \langle f_1(t)-f_2(t) , (z_1(t)-z_2(t))^+ \rangle \leq 0\]
whence using T-monotonicity and coercivity, we obtain $z_1 \leq z_2$.
\end{proof}
We start by giving some existence results for  \eqref{eq:Pvi} now in the general case when $\psi$ is not necessarily independent of time. The first one is a result of applying Theorem \ref{thm:discretisationExistence} of \S \ref{sec:discretisation} using the obstacle mapping $\Phi(v)(t) \equiv \psi(t)$ (it can be seen that all of the assumptions of the theorem are satisfied, refer also to the remarks below the proof of Lemma \ref{lem:recoverySequence}).
\begin{prop}[Existence via time discretisation]\label{prop:existenceUsingDiscretisation}Let $\psi \in C^0([0,T];V)$ be non-negative with $t \mapsto \psi(t)$ increasing and let \eqref{ass:dataNonnegativeAndfInc} and
\[z_0 \leq \psi(0) \text{ and } \langle Az_0 - f(t), v\rangle \leq 0\text{ for all $v \in V_+$ and a.e. $t$}\]
hold. Then \eqref{eq:Pvi} has a unique non-negative solution $z=\sigma(f,\psi) \in W_s(0,T)$ which is increasing in time.
\end{prop}
The next proposition applies a result due to Brezis--Stampacchia (see \cite[\S 2.9.6.1, p.~286]{Lions1969}) to obtain existence of a very weak solution and then a further argument is required to obtain additional regularity.
\begin{prop}[Existence II]\label{prop:existenceUsingBrezis}Let $\psi \in W_{\mathrm r}(0,T)$ be such that $t \mapsto \psi(t)$ is increasing with $z_0 \leq \psi(0)$. Then \eqref{eq:Pvi} has a unique solution $z=\sigma(f,\psi) \in W_{\mathrm r}(0,T)$.
\end{prop}
\begin{proof}
First observe that since $z_0 \leq \psi(0)$ and $t \mapsto \psi(t)$ is increasing, Theorem 7.1 of \cite[\S III]{Showalter} gives the existence of a weak solution $z \in L^2(0,T;V) \cap L^\infty(0,T;H)$ (see the introduction for the definition):
\[z(t) \leq \psi(t) : \int_0^T \langle v'(t) + Az(t)-f(t), z(t)-v(t) \rangle \leq 0 \quad \forall v \in W(0,T) : v \leq \psi(t), v(0) = z_0.\]
Indeed, if for simplicity we take $z_0 \equiv 0$, the domain of $L$ is $D(L):=\{ v \in H^1(0,T;V^*) : v(0) = 0\}$ and the condition (7.5) of \cite[\S III.7]{Showalter} follows since $\psi$ is increasing in time (see also \cite[p.~150]{Showalter}) and condition (7.7) of \cite[\S III.7]{Showalter} holds for the function $v_0 := \psi$. Hence the aforementioned theorem is applicable.

Now, given $v \in L^2(0,T;V)$ with $v(t) \leq \psi(t)$, consider the PDE
\begin{align*}
\epsilon v_\epsilon' + \epsilon Av_\epsilon + v_\epsilon &= v + \epsilon {L}\psi\,\
v_\epsilon(0) &= z_0,
\end{align*}
which, by standard parabolic theory, has a strong solution $v_\epsilon \in W_s(0,T)$ thanks to the regularity on $\psi$. A rearrangement and adding and subtracting the same term leads to
\begin{align*}
\epsilon v_\epsilon' + \epsilon Av_\epsilon -\epsilon {L}\psi + v_\epsilon-\psi &= v - \psi,\\
(v_\epsilon-\psi)(0) &= z_0 - \psi(0).
\end{align*}
Testing the equation above with $(v_\epsilon-\psi)^+$ and using the non-negativity of the right-hand side,
\begin{align*}
\int_0^T \langle {L}(v_\epsilon-\psi), (v_\epsilon-\psi)^+\rangle &\geq \frac 12 \norm{(v_\epsilon(T)-\psi(T))^+}{H}^2 - \frac 12 \norm{(z_0 - \psi(0))^+}{H}^2\\
&\quad + C_a\int_0^T \norm{(v_\epsilon(t)-\psi(t))^+}{V}^2 
\end{align*}
we find
\[\int_0^T \norm{(v_\epsilon(t)-\psi(t))^+}{H}^2 \leq \int_0^T (v(t)-\psi(t), (v_\epsilon(t)-\psi(t))^+)_H \leq 0\]
which implies that $v_\epsilon(t) \leq \psi(t)$. Then \cite[\S III, Proposition 7.2]{Showalter} implies that the solution is actually strong, i.e., \eqref{eq:Pvi} holds and it belongs to $W(0,T)$ with the additional regularity ${L}z \in L^2(0,T;H)$, i.e., $z \in W_{\mathrm r}(0,T)$. 
\end{proof}
Some related regularity results given sufficient smoothness for $f$ can be found in e.g. \cite[Theorem 2.1]{MR653144}. 

\subsection{Parabolic VIs with obstacle mapping}
Take $\Phi$ as in the introduction and fix $\psi \in L^2(0,T;H)$. Define the map 
\[S_{z_0}(f,\psi) := \sigma_{z_0}(f,\Phi(\psi)),\] that is, $z=S_{z_0}(f,\psi)$ solves 
\begin{equation}\label{eq:parabolicVIObstacleMapping}
\begin{aligned}
z(t) \leq \Phi(\psi)(t) : \int_0^T \langle z'(t) + Az(t) - f(t), z(t) - v(t) \rangle &\leq 0 \\
&\!\!\!\!\!\!\!\!\!\!\!\!\!\!\!\!\!\!\!\!\!\!\!\!  \forall v \in L^2(0,T;V) : v(t) \leq \Phi(\psi)(t),\\
z(0) &= z_0.
\end{aligned}
\end{equation}
We will often omit the subscript $z_0$ in $S_{z_0}(f,\psi)$ and simply write $S(f,\psi)$.
Let us translate the content of Propositions \ref{prop:existenceUsingDiscretisation} and \ref{prop:existenceUsingBrezis} to this setting. 
\begin{prop}[Existence for \eqref{eq:parabolicVIObstacleMapping}]\label{prop:existenceForParabolicVIObstacle}
Let
\begin{align}
&t \mapsto \Phi(\psi)(t) \text{ be increasing}\label{ass:PhiAtPsiIsIncreasingInTime},\\
&z_0 \leq \Phi(\psi)(0)\label{ass:InitialDataBoundedAboveByObstacleAtZero},
\end{align}
and either
\begin{empheq}[left={\empheqlbrace}]{align}
\nonumber &\eqref{ass:dataNonnegativeAndfInc},\\
&\langle Az_0 - f(t), v\rangle \leq 0\text{ for all $v \in V_+$ and a.e. $t$}\label{ass:NEWASSPHIPSI},\\
&\text{$\Phi(\psi) \in C^0([0,T];V)$ and $\Phi(\psi) \geq 0$}\label{ass:PhiIsNonNegativeAtPsi},
\end{empheq}
or
\begin{align}
&\text{$\Phi(\psi) \in W_{\mathrm r}(0,T)$.}\label{ass:existenceForPVIOBstacleViaBrezis}
\end{align}
Then \eqref{eq:parabolicVIObstacleMapping} has a solution $z=S(f, \psi)\in W(0,T)$ with the regularity that, in the first case, $z \in W_s(0,T)$ is non-negative and $t \mapsto z(t)$ is increasing, whereas in the second case, $z \in W_{\mathrm r}(0,T).$
\end{prop}
\begin{proof}
The first set of assumptions imply that the hypotheses of Proposition  \ref{prop:existenceUsingDiscretisation} hold for the obstacle $\Phi(\psi)$, whilst under the second set of assumptions, we apply Proposition \ref{prop:existenceUsingBrezis}.
\end{proof}
\begin{remark}
In this section, it suffices for $\Phi$ to be defined on $L^2(0,T;V)$ rather than $L^2(0,T;H)$ since that assumption was necessarly only to apply the Birkhoff--Tartar theory in the previous section. 
\end{remark}
The next lemma states that the solution mapping $S(f,\psi)=z$ is increasing with respect to the arguments. This follows simply by using Lemma \ref{lem:comparisonPrincipleNoPhi} and the fact that $\Phi$ is increasing.
\begin{lem}[II. Comparison principle for parabolic VIs]\label{lem:comparisonPrinciple}
Suppose for $i=1, 2$ that $z_i \in W(0,T)$ is a solution of \eqref{eq:parabolicVIObstacleMapping} with data $f_i \in L^2(0,T;V^*)$ and obstacle $\psi_i$ such that $f_1 \leq f_2$ and $\psi_1 \leq \psi_2$. Then $z_1 \leq z_2$. 
\end{lem}
\subsection{Iteration scheme to approximate a solution of the parabolic QVI}
We say that a function $z_{\mathrm{sub}} \in L^2(0,T;H)$ is a \textit{subsolution} for \eqref{eq:qvi} if $z_{\mathrm{sub}} \leq S(f, z_{\mathrm{sub}})$, and a \textit{supersolution} is defined with the opposite inequality.
\begin{remark}Let $\Phi(0) \geq 0$. If $f \geq 0$, the function $0$ is a subsolution, and the function $\bar z$ defined by
\begin{align*}
\bar z' + A\bar z -f &= 0\\
\bar z(0) &= z_0
\end{align*}
is a supersolution to \eqref{eq:qvi}. Both claims follow by the comparison principle: the first claim is clear and the second follows upon realising that $\bar z = S(f, \infty) \geq S(f,\bar z)$. We need the sign condition on $f$ for $z_{\mathrm{sub}}=0 \leq z^{sup}=\bar z$.
\end{remark}
\begin{lem}\label{lem:subsolutionForPerturbedQVI}If $d \geq 0$, any subsolution for $\mathbf{P}(f)$ is a subsolution for $\mathbf{P}(f+sd)$ where $s \geq 0.$
\end{lem}
\begin{proof}
This is obvious: if $w$ is a subsolution for $\mathbf{P}(f)$, then $w \leq S(f,w) \leq S(f+sd, w)$.
\end{proof}
The previous lemma tells us in particular that any element of $\mathbf{P}(f)$ is a subsolution for $\mathbf{P}(f+sd)$. The next theorem, which shows that a solution of the QVI can be approximated by solutions of VIs defined iteratively with respect to the obstacle, is based on an iteration idea of Bensoussan and Lions in \cite[Chapter 5.1]{LionsBensoussan} (there, the authors consider $\Phi$ to be of impulse control type). The theorem and its sister result Theorem \ref{thm:approximationOfQVIByVIIteratesDec} show in particular that the approximating sequences converge to extremal (the smallest or largest) solutions of the QVI in certain intervals.
\begin{theorem}[Increasing approximation of the minimal solution of QVI by solutions of VIs]\label{thm:approximationOfQVIByVIIterates}
Let $z_{\mathrm{sub}} \in L^2(0,T;V)$ be a subsolution for $\mathbf{P}(f)$ such that $\Phi(z_{sub
}) \geq 0$ and
\begin{equation}
z_0 \leq \Phi(z_{sub})(0) \label{ass:NEWASS}.
\end{equation}
Let either 
\begin{empheq}[left={\empheqlbrace}]{align}
&t \mapsto \Phi(z_{\mathrm{sub}})(t) \text{ is increasing}\label{ass:newIncreasingPropertyOfSub},\\
&\Phi(z_{\mathrm{sub}}) \in C^0([0,T];V)\label{ass:PhiAtSubsolnIsCts},\\
&\text{\eqref{ass:dataNonnegativeAndfInc}},\nonumber \\
&\eqref{ass:NEWASSPHIPSI},\nonumber\\
&\Phi(0) \geq 0,\label{ass:PhiNonNegativeForNonNegativeObs}\tag{O1a}\\
&\nonumber \text{$\forall \psi \in L^2(0,T;V), t \mapsto \psi(t)$ is increasing $\implies$}\\
& \qquad\qquad   \qquad \qquad \qquad \qquad \qquad \qquad  \text{$t \mapsto \Phi(\psi)(t)$ is increasing},\label{ass:newIncreasingProperty2}\\
&\Phi\colon W_s(0,T) \to C^0([0,T];V),\label{ass:PhiTakesWsIntoCts}\\
&\nonumber\text{$w_n \weaklyto w$ in $L^2(0,T;V)$ and weakly-* in $L^\infty(0,T;H) \implies$}\\
&\qquad\qquad   \qquad \qquad \qquad \qquad \qquad \text{$\Phi(w_n) \to \Phi(w)$ in $L^2(0,T;V)$,}
\tag{O2a}\label{ass:completeContinuityOfPhi}
\end{empheq}
or
\begin{empheq}[left={\empheqlbrace}]{align}
&\forall \psi \in L^2(0,T;V) : \psi \geq z_{\mathrm{sub}}, t \mapsto \Phi(\psi)(t) \text{ is increasing,}\label{ass:PhiIsIncreasingForAllFnsBiggerThanSubsoln}\\
&\Phi(z_{\mathrm{sub}}) \in W_{\mathrm r}(0,T),\label{ass:PhizSubInW}\\
&\Phi(\psi) \geq 0\quad \forall \psi \in L^2(0,T;V), \label{ass:PhiNonNegativeForAllPsiInBochnerSpace}\tag{O1b}\\
&\Phi\colon W_{\mathrm r}(0,T) \to W_{\mathrm r}(0,T),\tag{O2b}\label{ass:T2}\\
&\nonumber w \in W_{\mathrm r}(0,T) : w(0) = z_0 \implies  z_0 \leq  \Phi(w)(0)\text{, or }  \Phi(v) \leq \Phi(w) \implies \\ &\qquad\qquad\qquad\qquad\qquad\qquad\qquad\qquad\qquad\qquad\Phi(v)(0) \leq \Phi(w)(0),\label{ass:NEWASS2}\\
&\nonumber\text{$w_n \in W_{\mathrm r}(0,T)$, $w_n \weaklyto w$ in $L^2(0,T;V)$ and weak-* in $L^\infty(0,T;H)$}\\
&\qquad\qquad\qquad\qquad\qquad\qquad\qquad\text{$ \implies \Phi(w_n) \to \Phi(w)$ in $W(0,T)$},\tag{O3b}\label{ass:strongerCompleteContinuityOfPhi}
\end{empheq}
hold. Then the sequence $\{z^n\}_{n \in \mathbb{N}}$ denoted by
\begin{align*}
z^0 &:= z_{\mathrm{sub}},\\
z^n &:= S_{z_0}(f, z^{n-1}) \quad\text{ for } n = 1, 2, 3, ...,
\end{align*}
is well defined, monotonically increasing and satisfies
\begin{equation*}
\begin{aligned}
&z^n \nearrow z \text{ where } z \in \mathbf{P}_{z_0}(f) \text{ is the minimal very weak solution in $[z_{\mathrm{sub}}, \infty),$}\\
&z^n \weaklyto z \text{ in $L^2(0,T;V)$ and weakly-* in $L^\infty(0,T;H)$}.
\end{aligned}
\end{equation*}
Furthermore,
\begin{itemize}
\item in the first case,  $z^n \in W_s(0,T)$ with $z^n \geq 0$ , $\partial_t z^n \weaklyto \partial_t z$ in $L^2(0,T;H)$ and $z \in W_s(0,T)$ is a strong solution (i.e. it satisfies \eqref{eq:qvi} with additional regularity on the time derivative), and both $z^n$ and $z$ are increasing in time
\item or, in the second case, $z^n \in W_{\mathrm r}(0,T)$.
\end{itemize} 
\end{theorem}
\noindent Before we proceed, let us observe that 
\begin{enumerate}
\item since $\Phi$ is increasing, \eqref{ass:PhiNonNegativeForNonNegativeObs} is equivalent to $\psi \geq 0 \implies \Phi(\psi) \geq 0$
\item \eqref{ass:strongerCompleteContinuityOfPhi} implicitly implies that $\Phi(w) \in W(0,T)$.
\end{enumerate}
\begin{proof}
The proof is split into five steps.

\noindent \textsc{1. Monotonicity of $\{z^n\}$}. The $z^n$ (if they exist) satisfy for all $v \in L^2(0,T;V)$ with $v(t) \leq \Phi(z^{n-1})(t)$ the inequality
\begin{equation}\label{eq:VIforiterates}
\begin{aligned}
z^n(t) \leq \Phi(z^{n-1})(t) : \int_0^T \langle z_t^n(t) + Az^n(t) - f(t), z^n(t) - v(t) \rangle &\leq 0 \\
z^n(0) &= z_0.
\end{aligned}
\end{equation}
Since $z^0$ is a subsolution, $z^1 = S(f,z^0) \geq z^0$. Suppose that $z^j \geq z^{j-1}$ for some $j$; then $z^{j+1} = S(f,z^j) \geq S(f,z^{j-1}) = z^j$ (the inequality due to Lemma \ref{lem:comparisonPrinciple}). This shows that $z^n$ is a monotonically increasing sequence. 

\medskip

\noindent \textsc{2. Existence of $\{z^n\}$}. For the actual existence, we apply Proposition \ref{prop:existenceForParabolicVIObstacle} as we see now. 

\medskip

\noindent \textit{First case}. In case of \eqref{ass:dataNonnegativeAndfInc}, \eqref{ass:NEWASS}, \eqref{ass:NEWASSPHIPSI}, \eqref{ass:PhiAtSubsolnIsCts}, \eqref{ass:newIncreasingPropertyOfSub} and since $\Phi(z_{\mathrm{sub}}) \geq 0$, Proposition \ref{prop:existenceForParabolicVIObstacle} tells us that $z^1=S(f,z_{\mathrm{sub}}) \in W_s(0,T)$. We find $z^1 \geq S(0,0) = 0$ by Lemma \ref{lem:comparisonPrinciple}, and hence by \eqref{ass:PhiNonNegativeForNonNegativeObs}, $\Phi(z^1) \geq 0$. 

Let us also see why $z_0 \leq \Phi(z^1)(0)$. The monotonicity above implies that $\Phi(z_{sub}) \leq \Phi(z^1)$. As $z^1 \in W_s(0,T)$, \eqref{ass:PhiTakesWsIntoCts} implies that $\Phi(z^1) \in C^0([0,T];V)$, which along with \eqref{ass:PhiAtSubsolnIsCts} implies that we can take the trace of the previous inequality at time $0$, giving $z_0 \leq \Phi(z_{sub})(0) \leq \Phi(z^1)(0)$ where the first inequality is with the aid of \eqref{ass:InitialDataBoundedAboveByObstacleAtZero}. Making use of \eqref{ass:PhiTakesWsIntoCts}, the increasing property and \eqref{ass:newIncreasingProperty2} (which tells us that $t \mapsto \Phi(z^1)(t)$ is increasing, since $t \mapsto z^1(t)$ is), Proposition \ref{prop:existenceForParabolicVIObstacle} is again applicable and we use it to obtain $z^2$.

By bootstrapping this argument we get that $z^n \in W_s(0,T)$ is well defined. Furthermore we have $z^n \geq 0$ by the sign condition on the data.

\medskip

\noindent \textit{Second case}. In case of \eqref{ass:PhiIsIncreasingForAllFnsBiggerThanSubsoln} and \eqref{ass:PhizSubInW}, \eqref{ass:InitialDataBoundedAboveByObstacleAtZero}, \eqref{ass:PhiAtPsiIsIncreasingInTime} and \eqref{ass:existenceForPVIOBstacleViaBrezis} and \eqref{ass:InitialDataBoundedAboveByObstacleAtZero} hold for the obstacle $z_{\mathrm{sub}}$ and we get $z^1 = S(f, z_{sub}) \in W_{\mathrm r}(0,T)$. Applying \eqref{ass:T2} to this, $\Phi(z^1) \in W_{\mathrm r}(0,T)$. Suppose that the first part of \eqref{ass:NEWASS2} holds. Then since $z^1(0) = z_0$,  we find that $z_0 \leq \Phi(z^1)(0)$ and then again \eqref{ass:initialDataNewAss} is satisfied for the obstacle $\Phi(z^1)$, giving existence of $z^2=S(f,z^1)$. Repeating this, we get $z^n \in W_{\mathrm r}(0,T)$. If instead the second part of \eqref{ass:NEWASS2} holds, we get by monotonicity that $z^0 \leq z^1$ and using the increasing property of $\Phi$, $z_0 \leq \Phi(z^0)(0) \leq \Phi(z^1)(0)$ (the first inequality by \eqref{ass:NEWASS}) and again we can apply the existence and proceed in this manner for general $n$.


\medskip

\noindent \textsc{3. Uniform bounds on $\{z^n\}$.} By \eqref{ass:PhiNonNegativeForNonNegativeObs} and the fact that $z^n \geq 0$, or by \eqref{ass:PhiNonNegativeForAllPsiInBochnerSpace}, we find that $0$ is a valid test function in \eqref{eq:VIforiterates} and testing with it yields
\begin{align*}
\frac 12  \frac{d}{dt}\int_0^T\norm{z^n(t)}{H}^2 + C_a\int_0^T\norm{z^n(t)}{V}^2 
&\leq \epsilon\int_0^T \norm{z^n(t)}{H}^2 + C_\epsilon\int_0^T \norm{f(t)}{H}^2
\end{align*}
which immediately leads to bounds in $L^\infty(0,T;H)$ and $L^2(0,T;V)$ giving the weak convergences stated in the theorem 
to some $z$, for the full sequence (and not a subsequence) thanks to the monotonicity property. 

\medskip

\noindent \textit{3.1. Uniform bounds on $\partial_t z^n$ under first set of assumptions.} In this case, we can obtain a bound on the time derivative. Indeed, due to the work on the time discretisations in \S \ref{sec:discretisation}, from \eqref{eq:listOfConvergences} and Lemma \ref{lem:identificationOfHatuAndu} we know that for each $j$, the interpolant $(z^j)^N$ of the time-discretised solutions is such that $(z^j)^N \weaklyto z^{j}$ in $L^2(0,T;V)$ and $\partial_t (\hat {z}^j)^N \weaklyto \partial_t z^{j}$ in $L^2(0,T;H)$. One should bear in mind that the $\{(z^j)^N\}_N$ are interpolants constructed from solutions of elliptic VIs and not QVIs since  the $\{z^j\}_{j \in \mathbb{N}}$ are solutions of VIs. One observes the bound
\[\norm{\partial_t z^{j}}{L^2(0,T;H)} \leq \liminf_{N \to \infty} \norm{\partial_t (\hat{z}^j)^N}{L^2(0,T;H)} \leq C\]
where the constant $C$ ultimately arises from \eqref{eq:bdhinvDifferencesInH}. Evidently, it only depends on the initial data and the source term. Hence, in this case,
\begin{equation}\label{eq:derivativeLimit}
\partial_t z^j \weaklyto \partial_t z \text{ in $L^2(0,T;H)$}.
\end{equation}

\medskip

\noindent \textsc{4. Passage to the limit.} Either of the conditions \eqref{ass:completeContinuityOfPhi} and \eqref{ass:strongerCompleteContinuityOfPhi} allow us to pass to the limit in $z^n(t) \leq \Phi(z^{n-1})(t)$ to deduce the feasibility of $z$ since order is preserved in norm convergence. Now  let $v^* \in L^2(0,T;V)$ be a test function such that $v^* \leq \Phi(z)$. We use 
\[v^n(t) :=  v^*(t) + \Phi(z^{n-1})(t) - \Phi(z)(t)\] 
as the test function in the VI \eqref{eq:VIforiterates}. 

\medskip

\noindent \textit{4.1. Under first set of assumptions.} In this case, using the strong convergence $v^n \to v^*$ assured by \eqref{ass:completeContinuityOfPhi}, we can use \eqref{eq:derivativeLimit} and pass to the limit after writing the duality pairing for the time derivative as an inner product to get the inequality
\[\int_0^T (z'(t), z(t)-v^*(t))_H + \langle Az(t)-f(t), z(t)-v^*(t) \rangle \leq 0.\]

\medskip

\noindent \textit{4.2. Under second set of assumptions.} In this case, we take the limiting test function $v^* \in W(0,T)$ with $v^*(0)=z_0$ and rewrite \eqref{eq:VIforiterates}, using the monotonicity of the time derivative and assumption \eqref{ass:strongerCompleteContinuityOfPhi} (which guarantees that $\Phi(z) \in W(0,T)$, and hence $v^n \in W(0,T)$) as
\begin{equation*}
\begin{aligned}
z^n(t) \leq \Phi(z^{n-1})(t) : \int_0^T \langle v^n_t(t) + Az^n(t) - f(t), z^n(t) - v^n(t) \rangle &\leq 0 \\
z^n(0) &= z_0.
\end{aligned}
\end{equation*}
By \eqref{ass:strongerCompleteContinuityOfPhi}, we find that
$v^n \to v^*$ in $W(0,T)$, and hence we can pass to the limit in the above  to obtain \eqref{eq:qviWeak}.

\medskip

\noindent \textsc{5. Minimality of the solution.}
Suppose that $z^* \in \mathbf{P}(f)$ is the minimal solution on the interval $[z_{\mathrm{sub}}, \infty)$, which in particular implies $z^* \leq z.$ We see by the comparison principle and since $z^0=z_{\mathrm{sub}}$ is a subsolution that
\[z^* = S(f,z^*) \geq S(f,z^0) \geq z^0.\]
By this, we find $z^1=S(f,z^0) \leq S(f,z^*) = z^*$. Similarly, $z^2 = S(f,z^1) \leq S(f,z^*) = z^*$, and thus
\[z^n \leq z^*.\]
Passing to the limit shows that $z \leq z^*$ so that $z=z^*$.
\end{proof}
\begin{remark}
Note that the compactness assumptions \eqref{ass:completeContinuityOfPhi} and \eqref{ass:strongerCompleteContinuityOfPhi} are only required for identifying the limit point $z$ and showing that it is feasible.
\end{remark}
%
%
\begin{theorem}[Decreasing approximation of the maximal solution of QVI by solutions of VIs]\label{thm:approximationOfQVIByVIIteratesDec}
Let $z^0 := z^{\sup}$ be a supersolution of $\mathbf{P}(f)$ and assume that
\begin{equation*}
w \in W_{\mathrm r}(0,T) : w(0) = z_0 \implies  z_0 \leq  \Phi(w)(0)\text{, or }  \Phi(v) \geq \Phi(w) \implies  \Phi(v)(0) \leq \Phi(w)(0).
\end{equation*}
Under the assumptions of the previous theorem (except \eqref{ass:NEWASS2}) except with $z_{\mathrm{sub}}$ replaced with $z^{sup}$ and \eqref{ass:PhiIsIncreasingForAllFnsBiggerThanSubsoln} replaced with
\begin{align*}
&t \mapsto \Phi(\psi)(t) \text{ is increasing for all $\psi \in L^2(0,T;V)$ with $\psi \leq z^{sup}$},
\end{align*}
the sequence $\{z^n\}$ is monotonically decreasing and converges to a solution $z \in \mathbf{P}(f)$ with the same regularity and convergence results as stated in Theorem \ref{thm:approximationOfQVIByVIIterates}. Furthermore, $z$ is the maximal solution of \eqref{eq:qvi} or \eqref{eq:qviWeak} in the interval $(-\infty,z^{sup}]$.
\end{theorem}
\begin{proof}
It follows that $z \in \mathbf{P}(f) \cap (-\infty, z^{sup}]$ by the same argumentation as in the proof of the previous theorem. Let us prove the claim of the maximality of the solution. Suppose that there exists a maximal solution $z^* \in (-\infty, z^{sup}]$ so that $z^* \geq z$ where $z = \lim_n z^n $ with $z^0 = z^{sup}.$ We have $z^0=z^{sup} \geq S(f,z^{sup}) \geq S(f,z^*) = z^*$, and thus $z^1 = S(f,z^0) \geq S(f,z^*) = z^*.$ Iterating shows that
\[z^n \geq z^*,\]
whence passing to the limit, $z \geq z^*$, and thus $z = z^*$ is the maximal solution.
\end{proof}

\subsection{Transformation of VIs with obstacle to zero obstacle VIs}\label{sec:transformation}
It will become useful to relate solutions of the parabolic VI \eqref{eq:parabolicVIObstacleMapping} with non-trivial obstacle to solutions of parabolic VIs with zero (lower) obstacle.  We achieve this as follows. Take $w_0 \geq 0$ and define $\bar{S}_{w_0}\colon L^2(0,T;H) \to W_s(0,T)$ by $\bar{S}_{w_0}(g):=w$ the solution to the parabolic VI with lower obstacle
\begin{equation*}
\begin{aligned}
w(t) \geq 0 : \int_0^T \langle w'(t) + Aw(t) - g(t), w(t)-v(t) \rangle &\leq 0\quad \forall v \in L^2(0,T;V) : v(t) \geq 0,\\
w(0) &= w_0.
\end{aligned}
\end{equation*}
Omitting when convenient the subscript, we obtain the following estimate for $w_i = \bar{S}(g_i)$:
\[\frac 12 \norm{w_1(t)-w_2(t)}{H}^2 + C_a \norm{w_1-w_2}{L^2(0,t;V)}^2 \leq \int_0^t (g_1(r)-g_2(r),w_1(r)-w_2(r))_H\]
which then leads to
\begin{align}
\norm{\bar{S}(g_1)-\bar{S}(g_2)}{L^\infty(0,T;H)} &\leq 2\norm{g_1-g_2}{L^1(0,T;H)},\label{eq:S0LipschitzLinfty}
\end{align}
and (due to Young's inequality applied to the right-hand side)
\begin{equation*}
\norm{\bar{S}(g_1)-\bar{S}(g_2)}{L^\infty(0,T;H)}^2 + C_a\norm{\bar{S}(g_1)-\bar{S}(g_2)}{L^2(0,T;V)}^2 \leq \frac{1}{C_a}\norm{g_1-g_2}{L^2(0,T;V^*)}^2,
\end{equation*}
hence also
\begin{equation}\label{eq:S0LipschitzWeirdSpaces2}
\norm{\bar{S}(g_1)-\bar{S}(g_2)}{L^\infty(0,T;H)} \leq \frac{1}{\sqrt{C_a}}\norm{g_1-g_2}{L^2(0,T;V^*)}.
\end{equation}
Let us note that  \eqref{eq:S0LipschitzLinfty} in particular implies
\begin{equation}\label{eq:S0LipschitzLinfty1}
\norm{\bar{S}(g_1)-\bar{S}(g_2)}{L^p(0,T;H)} \leq 2T^{\frac 1p}\norm{g_1-g_2}{L^1(0,T;H)}.
\end{equation}
The relationship between solutions of VIs with non-trivial obstacles and VIs with zero obstacle is given in the next result.
\begin{prop}\label{prop:conditionsForFormulaToHold}Let \eqref{ass:PhiAtPsiIsIncreasingInTime}, \eqref{ass:InitialDataBoundedAboveByObstacleAtZero} hold and let $g \in L^2(0,T;H)$, $z_0 \in V$ and $\psi \in L^2(0,T;V)$ be such that $\Phi(\psi) \in W_{\mathrm r}(0,T)$. 

Then
\begin{equation}\label{eq:relationS0AndS}
S_{z_0}(g,\psi) = \Phi(\psi) - \bar{S}_{w_0}({L}\Phi(\psi)-g)
\end{equation}
holds in $W_{\mathrm r}(0,T)$ where $w_0 =  \Phi(\psi)(0) - z_0$.

Furthermore, if \eqref{ass:dataNonnegativeAndfInc}, \eqref{ass:NEWASSPHIPSI}, and \eqref{ass:PhiIsNonNegativeAtPsi} hold, then in fact $S_{z_0}(g,\psi) \in W_s(0,T)$ and hence the spatial regularity $AS_{z_0}(g,\psi) \in L^2(0,T;H)$.
\end{prop}
\begin{proof}
Under the hypotheses, Proposition \ref{prop:existenceForParabolicVIObstacle} can be applied to deduce that $z:=S_{z_0}(g,\psi)$ is well defined in  $W_{\mathrm r}(0,T)$ and it solves the VI \eqref{eq:parabolicVIObstacleMapping}.
Set $w := \Phi(\psi)- z$ (which belongs to $W_{\mathrm r}(0,T)$) and observe that
\begin{align*}
\int_0^T \langle \partial_t \Phi(\psi) + A\Phi(\psi) - w' - Aw - g, \Phi(\psi) - w - v  \rangle &\leq 0,\\
w(0) &= w_0 :=  \Phi(\psi)(0) - z_0 \in H.
\end{align*}
The upper bound on $z_0$ implies that $w_0 \geq 0$. Now define $\varphi(t) := \Phi(\psi)(t) - v(t)$. Then the above reads
\begin{align*}
w(t) \geq 0 : \int_0^T \langle w'(t) + Aw(t) + g(t)- \partial_t \Phi(\psi) - A\Phi(\psi) , w(t)-\varphi(t)  \rangle &\leq 0 \\
&\!\!\!\!\!\!\!\!\!\!\!\!\!\!\!\!\!\!\!\!  \forall \varphi \in L^2(0,T;V) : \varphi(t) \geq 0,\\
w(0) &= w_0.
\end{align*}
This shows the desired identity $\bar{S}_{w_0}({L}\Phi(\psi)-g)=w=\Phi(\psi)-z = \Phi(\psi) - S_{z_0}(g,\psi)$. Under the additional assumptions, Proposition \ref{prop:existenceForParabolicVIObstacle} yields $z \in W_s(0,T)$ and $w = \Phi(\psi) - z \in W_{\mathrm r}(0,T) + W_s(0,T)$.
\end{proof}

\section{Expansion formula for variations in the obstacle and source term}\label{sec:expansion}
The aim in this section is obtain differential expansion formulae for the solution mapping of parabolic VIs with respect to perturbations on the source term and the obstacle. This will form the backbone of our QVI differentiability result in the next section.

\subsection{Definitions and cones from variational analysis}
To state directional differentiability results for VIs, we need some concepts and notation which we shall collect in this subsection. Let us define the lower obstacle sets
\begin{align*}
K_0 &:= \{ v \in V : v \geq 0\}=V_+ \quad \text{ and }\quad 
\mathbb{K}_0 
:= \{ v \in W(0,T) : v(t) \in K_0 \text{ for a.e. $t \in [0,T]$}\},
\end{align*}
and for $y \in \mathbb{K}_0$, the \textit{radial cone} at $y$
\begin{align}
\nonumber T_{\mathbb{K}_0}^{\mathrm{rad}}(y) &:= \{ v \in W(0,T) : \exists \rho^* > 0 \text{ s.t. } y + \rho v \in \mathbb{K}_0 \text{ for all } \rho < \rho^*\}\\
&= \{ v \in W(0,T) : \exists \rho^* > 0 \text{ s.t. } y(t) + \rho v(t) \in K_0 \text{ for a.e. $t \in [0,T]$ for all } \rho < \rho^*\}\label{eq:dd1}.
\end{align}
We shall consider $\mathbb{K}_0$ as a subset of the Banach space $L^2(0,T;V)$. The \textit{tangent cone} is defined as the closure of the radial cone:
\[T_{\mathbb{K}_0, L^2(0,T;V)}^{\mathrm{tan}}(y) := \cl_{L^2(0,T;V)}T^{\mathrm{rad}}_{\mathbb{K}_0}(y).\]
Obviously, $T^{\mathrm{rad}}_{\mathbb{K}_0}(y) \subset T_{\mathbb{K}_0, L^2(0,T;V)}^{\mathrm{tan}}(y).$  We now show that the tangent cone is contained in a set which has a convenient description (see also the discussion after Remark 5.6 in \cite{Christof}).

We will need the notions of capacity of sets, quasi-continuity of functions and related concepts, consult \cite[\S 3]{MR0423155}, \cite[\S 3]{MR0481060} and \cite[\S 6.4.3]{MR1756264} for more details. Below, we shall use the abbreviation `q.e.' to mean quasi-everywhere; a statement holds quasi-everywhere if it holds everywhere except on a set of capacity zero. 
\begin{lem}\label{lem:tangentConeK0}
The tangent cone of $\mathbb{K}_0$ can be characterised as
\begin{equation}\label{eq:tangentConeInclusion}
T^{\mathrm{tan}}_{\mathbb{K}_0,L^2(0,T;V)}(y) \subset \{ v \in L^2(0,T;V) : v(t) \geq 0 \text{ q.e. on $\{\bar y(t)=0\}$ for a.e. $t \in [0,T]$}\}
\end{equation}
where $\bar y(t)$ is a quasi-continuous representative of $y(t)$.
\end{lem}
\begin{proof}
If $w \in T_{\mathbb{K}_0}^{\mathrm{rad}}(y)$, then from \eqref{eq:dd1}, $w \in L^2(0,T;V)$ and there exists $\rho^* \geq 0$ such that $y(t) + \rho w(t) \in K_0$ for almost every $t$ and for all $\rho \in [0,\rho^*)$, meaning that $w(t) \in T_{K_0}^{\mathrm{rad}}(y(t)) \subset T_{{K}_0}^{\mathrm{tan}}(y(t))$ for almost every $t$. This shows that
\[T_{\mathbb{K}_0}^{\mathrm{rad}}(y) \subset \{ w \in L^2(0,T;V) :  w(t) \in T_{{K}_0}^{\mathrm{tan}}(y(t)) \text{ a.e. $t \in [0,T]$}\}\]
and if we take the closure in $L^2(0,T;V)$ on both sides,
\begin{equation}\label{eq:dd2}
T_{\mathbb{K}_0, L^2(0,T;V)}^{\mathrm{tan}}(y) \subset cl_{L^2(0,T;V)}\left(\{ w \in L^2(0,T;V) :  w(t) \in T_{{K}_0}^{\mathrm{tan}}(y(t)) \text{ a.e. $t \in [0,T]$}\}\right).
\end{equation}
Suppose that $\{w_n\} \subset L^2(0,T;V)$ is a sequence that belongs to the set on the right-hand side above with $w_n \to w$ in $L^2(0,T;V)$. Thus, for a subsequence, $w_{n_j}(t) \to w(t)$ in $V$ and $w_{n_j}(t) \in T_{K_0}^{\mathrm{tan}}(y(t))$ for almost every $t$. Since the tangent cones are closed sets, the limit point $w(t) \in T_{K_0}^{\mathrm{tan}}(y(t))$. Hence $w \in \{ w \in L^2(0,T;V) :  w(t) \in T_{{K}_0}^{\mathrm{tan}}(y(t)) \text{ a.e. $t \in [0,T]$}\}$ and the closure can be omitted on the right-hand side of \eqref{eq:dd2}.

From \cite[Lemma 3.2]{MR0423155}, Mignot proves the following description of the tangent cone of $K_0$:
\[T_{{K}_0}^{\mathrm{tan}}(y) = \{ v \in V : v \geq 0 \text{ q.e. on } \{\bar y=0\}\},\]
with $\bar y$ a quasi-continuous representative of the function $y$. This provides the characterisation stated in the lemma. 
\end{proof}
The set $\mathbb{K}_0$ is said to be \emph{polyhedric at $(y,\lambda) \in \mathbb{K}_0 \times T_{\mathbb{K}_0, L^2(0,T;V)}^{\mathrm{tan}}(y)^{\circ}$} if
\[T_{\mathbb{K}_0, L^2(0,T;V)}^{\mathrm{tan}}(y)  \cap \lambda^\perp = cl_{L^2(0,T;V)}\left(T^{\mathrm{rad}}_{\mathbb{K}_0}(y) \cap \lambda^\perp \right),\]
where
\[T_{\mathbb{K}_0, L^2(0,T;V)}^{\mathrm{tan}}(y)^{\circ}:= \{ f \in L^2(0,T;V^*) : \langle f, w \rangle \leq 0 \text{ for all } w \in T_{\mathbb{K}_0, L^2(0,T;V)}^{\mathrm{tan}}(y)\}\]
is the known as the polar cone. The set is polyhedric at $y$ if it is polyhedric at $(y,\lambda)$ for all $\lambda$, and it is polyhedric if it is polyhedric at $y$ for all $y$. The concept of polyhedricity is useful because it is a sufficient condition guaranteeing  the directional differentiability of the metric projection associated to that set (see \cite{Haraux1977, MR0423155, Bonnans} and also \cite{WachsmuthGuidedTour}) and this fact ultimately enables one to obtain directional differentiability for solution mappings of variational inequalities.

Now, the set $\mathbb{K}_0$ is not polyhedric as a subset of the space $W(0,T)$ since $W(0,T)$ lacks certain smoothness properties due to the low regularity of the time derivative. However, $\tilde{\mathbb{K}}_0 := \mathbb{K}_0 \cap W_s(0,T)$ is indeed polyhedric.
\begin{lem}
The set $\tilde{\mathbb{K}}_0$ is polyhedric as a subset of $W_s(0,T)$ and for $(y,\lambda) \in \tilde{\mathbb{K}}_0 \times T_{\tilde{\mathbb{K}}_0, W_s(0,T)}^{\mathrm{tan}}(y)^\circ$, 
\begin{align*}
\cl_{W_s(0,T)}(T_{\tilde{\mathbb{K}}_0}^{\mathrm{rad}}(y) \cap \lambda^\perp) 
&= T_{\tilde{\mathbb{K}}_0, W_s(0,T)}^{\mathrm{tan}}(y)\cap \lambda^\perp\\
&= \{ z \in W_s(0,T) : z \geq 0 \text{ $W_s(0,T)$-q.e. in $\{\bar y=0\}$}\} \cap \lambda^\perp 
\end{align*}
where $\bar y$ is a quasi-continuous representative of $y$ and
\[\{\bar y=0\} := \{ p \in [0,T]\times \bar \Omega : \bar y(p) = 0\}.\]
\end{lem}
\begin{proof}
First note that if $v \in W_s(0,T)$, $\partial_t(v^+) = \chi_{\{v \geq 0\}}\partial_t v$ by the chain rule and hence we have the bound
\[\norm{v^+}{W_s(0,T)}^2 \leq C\norm{v}{L^2(0,T;V)}^2 + \norm{\partial_t v}{L^2(0,T;H)}^2,\]
which shows that $(\cdot)^+\colon W_s(0,T) \to W_s(0,T)$ is a bounded map. It follows that $W_s(0,T)$ is a vector lattice in the sense of Definition 4.6 of \cite{WachsmuthGuidedTour} when associated to the cone $\tilde{\mathbb{K}}_0$. The boundedness of $(\cdot)^+ \colon W_s(0,T) \to W_s(0,T)$ and Lemma 4.8 and Theorem 4.18 of \cite{WachsmuthGuidedTour} imply that $\tilde{\mathbb{K}}_0$ is polyhedric in $W_s(0,T)$ and hence
\begin{align*}
\cl_{W_s(0,T)}(T_{\tilde{\mathbb{K}}_0}^{\mathrm{rad}}(y) \cap \lambda^\perp) &= \cl_{W_s(0,T)}(T_{\tilde{\mathbb{K}}_0}^{\mathrm{rad}}(y))\cap \lambda^\perp\\
&= T_{\tilde{\mathbb{K}}_0, W_s(0,T)}^{\mathrm{tan}}(y)\cap \lambda^\perp.
\end{align*}
The space $W_s(0,T)$ is also a Dirichlet space in the sense of \cite[Definition 3.1]{MR0423155} on the set $[0,T]\times \bar \Omega$ and so, due to the characterisation of the tangent cone in \cite[Lemma 3.2]{MR0423155}, we find
\begin{align*}
\cl_{W_s(0,T)}(T_{\tilde{\mathbb{K}}_0}^{\mathrm{rad}}(y) \cap \lambda^\perp) 
&= \{ z \in W_s(0,T) : z \geq 0 \text{ $W_s(0,T)$-q.e. in $\{\bar y=0\}$}\} \cap \lambda^\perp.
\end{align*}
\end{proof}
\subsection{Directional differentiability for VIs}
We now specialise to the case where the pivot space $H$ is a Lebesgue space, a restriction which is needed for the results of \cite{Christof}.
\begin{ass}\label{ass:onSpaces}Set $H:=L^2(\Omega, \mu)$ where $(\Omega, \Sigma, \mu)$ is a complete measure space and let $V \subset H$ be a separable Hilbert space.
\end{ass}
Theorem 4.1 of \cite{Christof} states that for $w_0 \in V_+$, the map $\bar{S}_{w_0}\colon L^2(0,T;H) \to L^p(0,T;H)$ (defined in \S \ref{sec:transformation}) is directionally differentiable for all $p \in [1,\infty)$, i.e.,  
\begin{equation}\label{eq:S0Formula}
\bar{S}_{w_0}(g+sd) = \bar{S}_{w_0}(g) + s\bar{S}_{w_0}'(g)(d) + o(s,d;g)
\end{equation}
where $s^{-1}o(s) \to 0$ in $L^p(0,T;H)$ as $s \to 0^+$, and $\delta:=\bar{S}_{w_0}'(g)(d) \in L^2(0,T;V) \cap L^\infty(0,T;H)$ satisfies, with $w=\bar{S}_{w_0}(g)$, the inequality
\begin{equation}\label{eq:christoffVI}
\begin{aligned}
&\delta  \in T_{\mathbb{K}_0,L^2(0,T;V)}^{\mathrm{tan}}(w) \cap [w' + Aw - g]^\perp : \int_0^T  \langle \varphi' + A\delta - d, \delta-\varphi \rangle  \leq \frac 12 \norm{\varphi(0)}{H}^2\\
&\qquad \qquad \qquad \qquad \qquad \qquad \qquad  \forall \varphi \in \cl_{W(0,T)}(T_{\mathbb{K}_0}^{\mathrm{rad}}(w)\cap [w'+Aw-g]^\perp).
\end{aligned}
\end{equation}
Using this, we shall first work to deduce a differentiability formula for the map $S$ under perturbations of the right-hand side source with a fixed obstacle.
\begin{remark}\label{rem:basePoints}
Our notation emphasises the fact that the higher-order term in \eqref{eq:S0Formula} depends on the base point $g$. This is important because the behaviour of the higher-order terms is in general unknown with respect to the base points, such as for example whether there is any kind of uniformity of the convergence of the higher-order terms on compact or bounded subsets of the base points. Such uniform convergence does hold in cases where the map has more smoothness, namely if it possesses the so-called \emph{uniform Hadamard differentiability} property, but it is not clear whether this is the case for us when such issues become relevant in \S \ref{sec:otherApproaches}.

This is in stark contrast to the dependence on the direction: we know that the terms converge uniformly on compact subsets of the direction since $\bar{S}$ is Hadamard (and hence compactly) differentiable. 
\end{remark}
If $d(s) \to d$, we write
\begin{equation}\label{eq:S0HadamardFormula}
\bar{S}_{w_0}(g+sd(s)) = \bar{S}_{w_0}(g) + s\bar{S}_{w_0}'(g)(d) + \hat o(s,d, s(d(s)-d); g).
\end{equation}
Let us see why $\hat o$ above is a higher-order term. Let $h\colon (0,1) \to L^2(0,T;H)$ and take $d(s) = d + s^{-1}h(s)$ and $p \in [1,\infty)$. Subtracting \eqref{eq:S0Formula} from \eqref{eq:S0HadamardFormula}, we obtain from the Lipschitz nature of $\bar{S}$,
\begin{align*}
\norm{\hat o(s,d,h(s);g) - o(s,d;g)}{L^p(0,T;H)} &= \norm{\bar{S}(g+s(d + s^{-1}h(s)))-\bar{S}(g+sd)}{L^p(0,T;H)}\\
&\leq 2T^{\frac 1p}\norm{h(s)}{L^1(0,T;H)}\tag{by \eqref{eq:S0LipschitzLinfty1}}\\
&\leq 2T^{\frac 1p + \frac 12}\norm{h(s)}{L^2(0,T;H)}
\end{align*}
using $L^2(0,T;H) \cts L^1(0,T;H)$. 
The estimate 
\begin{align}
\norm{\hat o(s,d, h(s);g)}{L^p(0,T;H)} &\leq \frac{T^{\frac 1p}}{\sqrt{C_a}}\norm{h(s)}{L^2(0,T;V^*)} + \norm{o(s,d;g)}{L^p(0,T;H)}\label{eq:hatOBoundIntoL2VStar}
\end{align}
follows from applying instead the Lipschitz estimate \eqref{eq:S0LipschitzWeirdSpaces2} to the first line of the above calculation. 

The next proposition guarantees (under certain assumptions) the directional differentiability of one or both of the maps
\[S_{z_0}(\cdot,\psi)\colon L^2(0,T;H_+) \to L^p(0,T;H) \quad\text{and}\quad S_{z_0}(\cdot,\psi)\colon L^2(0,T;H) \to L^p(0,T;H).\] 
\begin{prop}\label{prop:hadamardDifferentiabilityForPVI}
Let $f,d \in L^2(0,T;H)$, $\psi \in L^2(0,T;V)$ with $\Phi(\psi) \in W_{\mathrm r}(0,T)$ and let \eqref{ass:PhiAtPsiIsIncreasingInTime} and  \eqref{ass:InitialDataBoundedAboveByObstacleAtZero} hold. 
Then
\begin{align}\label{eq:relationS0AndSDerivatives}
S_{z_0}(f+sd,\psi) = S_{z_0}(f,\psi) + s\partial S_{z_0}(f,\psi)(d) + h(s,d)\quad \text{in $W_{\mathrm r}(0,T)$}
\end{align}
where, with  $w_0 := \Phi(\psi)(0) - z_0$,
\begin{align}
\partial S_{z_0}(f,\psi)(d) := \bar{S}_{w_0}'({L}\Phi(\psi)-f)(d) \label{eq:relation2}
\end{align}
belongs to $L^2(0,T;V) \cap L^\infty(0,T;H)$, and $h$ is a higher-order term in $L^p(0,T;H)$ for $p \in [1,\infty)$ whose convergence is uniform in $d$ on compact subsets of $L^2(0,T;H)$. 

The directional derivative $\alpha:=\partial S_{z_0}(f,\psi)(d)$ satisfies
\begin{equation}\label{eq:myVI}
\begin{aligned}
&\alpha  \in T_{\mathbb{K}_0,L^2(0,T;V)}^{\mathrm{tan}}(w) \cap [w' + Aw - ({L}\Phi(\psi)-f)]^\perp : \\
&\qquad\qquad\qquad \int_0^T  \langle \varphi' + A\alpha - d, \alpha-\varphi \rangle  \leq \frac 12 \norm{\varphi(0)}{H}^2\\
&\quad \quad \quad \quad \quad \quad \quad \quad  \quad \quad \quad \forall \varphi \in \cl_{W}(T_{\mathbb{K}_0}^{\mathrm{rad}}(w)\cap [w'+Aw-({L}\Phi(\psi)-f)]^\perp),\\
&w=\bar{S}_{w_0}({L}\Phi(\psi)-f) = \Phi(\psi) - S_{z_0}(f,\psi).
\end{aligned}
\end{equation}
If additionally, for $s \geq 0,$
\begin{empheq}[left={\empheqlbrace}]{align}
&\text{$f+sd \geq 0$ and is increasing, $z_0 \geq 0$},\tag{$\text{D}_\text{s}$}\label{ass:perturbedDataNonnegativeAndInc}\\
\nonumber &\text{\eqref{ass:NEWASSPHIPSI}}, \\
\nonumber &\text{\eqref{ass:InitialDataBoundedAboveByObstacleAtZero}},\\
\nonumber &\text{\eqref{ass:PhiIsNonNegativeAtPsi}},\\
&\langle Az_0 - f(t)-sd(t), v\rangle \leq 0\text{ for all $v \in V_+$ and a.e. $t$,}\label{ass:NEWASSPHIPSIPERTURBED}
\end{empheq}
then \eqref{eq:relationS0AndSDerivatives} holds in $W_s(0,T)\cap W_{\mathrm r}(0,T).$
\end{prop}
\begin{proof}The assumptions imply that Proposition \ref{prop:existenceForParabolicVIObstacle} applies and we obtain existence of the left-hand side and the first term on the right-hand side of \eqref{eq:relationS0AndSDerivatives}. We can via Proposition \ref{prop:conditionsForFormulaToHold} utilise \eqref{eq:relationS0AndS} with the source terms $f$ and $f+sd$ to write $S_{z_0}(f,\psi)$ and $S_{z_0}(f+sd,\psi)$ in terms of $\bar{S}_{w_0}$, and then with the aid of the expansion formula \eqref{eq:S0Formula}, we find
\begin{align*}
S_{z_0}(f+sd, \psi) -S_{z_0}(f,\psi) &= \bar{S}_{w_0}({L}\Phi(\psi)-f) - \bar{S}_{w_0}({L}\Phi(\psi)-f-sd)\\
&= -s\bar{S}_{w_0}'({L}\Phi(\psi)-f)(-d) - o(s,-d;{L}\Phi(\psi)-f).
\end{align*}
\end{proof}
\subsection{Differentiability with respect to the obstacle and the source term}\label{sec:dd}
We clearly need some differentiability for the obstacle mapping to proceed the study further and this comes in the following assumption which we take to stand for the rest of the paper.
\begin{ass}\label{ass:PhiHadamardDifferentiable}
Suppose that $\Phi\colon L^2(0,T;H) \to W(0,T)$ is Hadamard differentiable.
\end{ass}
It is necessary for $\Phi$ to be defined on $L^2(0,T;H)$ and differentiable as stated because it implies the uniform convergence with respect to compact subsets of the direction in $L^2(0,T;H)$ of the difference quotients to the directional derivative of $\Phi$, which is a fact that we will use later in \S \ref{sec:HOTs} in the analysis of some higher-order terms that arise. 

We write, for $\rho(s) \to \rho$,
\begin{align}
\Phi(\psi+s\rho) &= \Phi(\psi) + s\Phi'(\psi)(\rho) + l(s,\rho;\psi),\label{eq:PhiFormula}\\
\Phi(\psi+s\rho(s)) &= \Phi(\psi) + s\Phi'(\psi)(\rho) + \hat l(s,\rho, s(\rho(s)-\rho); \psi)\label{eq:PhiHadamardFormula}.
\end{align}
\begin{remark}
Assumption \ref{ass:PhiHadamardDifferentiable} implies, thanks to $W(0,T) \cts C^0([0,T];H) \cts L^p(0,T;H)$, that 
\begin{equation}\label{ass:PhiIsDifferentiable}
\text{$\Phi \colon L^2(0,T;H) \to L^p(0,T;H)$ is Hadamard differentiable for any $p \in [1,\infty]$}.
\end{equation}
In this section, we could have merely assumed \eqref{ass:PhiIsDifferentiable} instead of Assumption \ref{ass:PhiHadamardDifferentiable} and most results would carry through all the way up to the identification of the term $r$ in Proposition \ref{prop:formulaForSDiff} as a higher-order term (and hence the differentiability)
\end{remark}
Now if $h\colon (0,1) \to L^2(0,T;H)$ satisfies $s^{-1}h(s) \to 0$ as $s \to 0^+$, then from \eqref{ass:PhiIsDifferentiable} and the mean value theorem \cite[\S2, Proposition 2.29]{Penot}, 
\begin{align*}
\norm{\hat l(s,\rho, h(s);\psi) - l(s,h;\psi)}{L^p(0,T;H)} &= \norm{\Phi(\psi+s(\rho + s^{-1}h(s)))-\Phi(\psi+s\rho)}{L^p(0,T;H)}\\
&\leq \sup_{\lambda \in [0,1]}\norm{\Phi'(\psi +s\rho + \lambda h(s))h(s)}{L^p(0,T;H)},
\end{align*}
which leads to the estimate
\begin{equation}\label{eq:hatLBound}
\norm{\hat l(s,\rho, h(s);\psi)}{L^p(0,T;H)} \leq \sup_{\lambda \in [0,1]}\norm{\Phi'(\psi+s\rho + \lambda h(s))h(s)}{L^p(0,T;H)} + \norm{l(s,\rho;\psi)}{L^p(0,T;H)}.
\end{equation}
Recall that $Lv:=v' + Av$.

\begin{lem}\label{lem:AAA}The map ${L}(\Phi(\cdot))\colon L^2(0,T;H) \to L^2(0,T;V^*)$ is Hadamard differentiable with derivative ${L}(\Phi'(\psi)(\rho))$ at the point $\psi$ in direction $\rho$, and its higher-order term ${L}(l(s,\rho;\psi))$ satisfies
\begin{align*}
\norm{{L}\hat l(s,\rho, h(s); \psi) }{L^2(0,T;V^*)} 
&\leq \sup_{\lambda \in (0,1)}\norm{{L}(\Phi'(\psi+s\rho+\lambda h(s))(h(s)))}{L^2(0,T;V^*)}\\
&\quad+ \norm{{L}l(s,\rho; \psi)}{L^2(0,T;V^*)}.
\end{align*}
\end{lem}
\begin{proof}
Applying the operator ${L}$ to \eqref{eq:PhiFormula} we get the following equality in $L^2(0,T;V^*)$:
\[{L}\Phi(\psi+s\rho) = {L}\Phi(\psi) + s{L}(\Phi'(\psi)(\rho)) + {L}l(s, \rho; \psi).\]
Due to the estimate
\begin{align*}
\norm{{L}l(s,\rho;\psi)}{L^2(0,T;V^*)} 
&\leq \norm{\partial_t l(s,\rho;\psi)}{L^2(0,T;V^*)} + C_b\norm{l(s,\rho;\psi)}{L^2(0,T;V)},
\end{align*}
we see that ${L}\Phi$ is Hadamard differentiable in the stated spaces since $\Phi\colon L^2(0,T;H) \to W(0,T)$ is Hadamard differentiable. Subtracting the expansion
\begin{align*}
{L}\Phi(\psi+s\rho+h(s)) &= {L}\Phi(\psi) + s{L}\Phi'(\psi)(\rho) + {L}\hat l(s,\rho, h(s); \psi)
\end{align*}
from the equality above, we obtain
\begin{align*}
{L}\hat l(s,\rho, h(s); \psi) - {L}l(s,\rho;\psi) &= {L}(\Phi(\psi+s\rho+h(s))- \Phi(\psi+s\rho)).
\end{align*}
Since the composite mapping ${L}\Phi \colon L^2(0,T;H) \to L^2(0,T;V^*)$ is Hadamard differentiable, the mean value theorem applied to the right-hand side above implies
\begin{align*}
\norm{{L}\hat l(s,\rho, h(s); \psi) - {L}l(s,\rho; \psi)}{L^2(0,T;V^*)} 
&\leq \sup_{\lambda \in (0,1)}\norm{{L}(\Phi'(\psi+s\rho+\lambda h(s))(h(s)))}{L^2(0,T;V^*)}.
\end{align*}
\end{proof}
\begin{prop}\label{prop:formulaForSDiff}
Assume \eqref{ass:PhiAtPsiIsIncreasingInTime} and let $f, d \in L^2(0,T;H)$, $\psi, \rho \in L^2(0,T;V)$ and $h\colon (0,1) \to L^2(0,T;V)$ with $h(0) = 0$ be such that, for $s \geq 0$,
\begin{align}
\nonumber &t \mapsto \Phi(\psi+s\rho + h(s))(t) \text{ is increasing},\\
&\Phi(\psi+s\rho+h(s)), \Phi'(\psi)(\rho) 
\in W_{\mathrm r}(0,T),\label{ass:LPhiObstacleAndPerutbredInL2H}\\
&z_0 \leq \Phi(\psi+s\rho + h(s))(0)\label{ass:perturbedObstacleInWszo}.
\end{align}
Then
\[S_{z_0}(f+sd, \psi + s\rho + h(s)) = S_{z_0}(f,\psi) + sS_{z_0}'(f,\psi)(d,\rho) + r(s,\rho,h(s);\psi)\]
holds in $W_{\mathrm r}(0,T)$ where 
\begin{align*}
S_{z_0}'(f,\psi)(d,\rho)&:= \Phi'(\psi)(\rho) + \partial S_{z_0}(f,\psi)(d-{L}\Phi'(\psi)(\rho)),\\
r(s,\rho,h;\psi) &:= \hat l(s,\rho, h;\psi) -\hat o(s,{L}\Phi'(\psi)(\rho)-d, {L}l(s,\rho,h;\psi); {L}\Phi(\psi)-f),
\end{align*}
and $\alpha:=S_{z_0}'(f,\psi)(d,\rho) \in \Phi'(\psi)(\rho) + L^2(0,T;V) \cap L^\infty(0,T;H)$ satisfies the VI
\begin{equation*}
\begin{aligned}
&\alpha - \Phi'(\psi)(\rho)\in T_{\mathbb{K}_0,L^2}^{\mathrm{tan}}(\Phi(\psi)-y) \cap [y' + Ay - f]^\perp :\\
&\qquad  \int_0^T  \langle \varphi' + A\alpha- d, \alpha-\varphi \rangle \leq \frac 12 \norm{\varphi(0) - \Phi'(\psi)(\rho)(0)}{H}^2 \\
&\qquad  \forall \varphi \in L^p(0,T;H) : \varphi - \Phi'(\psi)(\rho) \in \cl_{W}(T_{\mathbb{K}_0}^{\mathrm{rad}}(\Phi(\psi)-y)\cap [y'+Ay-f]^\perp),\\
&y:=S_{z_0}(f,\psi).
\end{aligned}
\end{equation*}
If additionally
\begin{empheq}[left={\empheqlbrace}]{align}\nonumber 
&\eqref{ass:perturbedDataNonnegativeAndInc},\\
&\Phi(\psi+s\rho + h(s)) \in C^0([0,T];V)\text{ and }\Phi(\psi+s\rho + h(s)) \geq 0\label{ass:perturbedObstacleNonNegative},\\
&\langle Az_0 - f(t) - sd(t), v \rangle \leq 0 \text{ for all } v \in V_+\text{ a.e. $t$}\label{ass:perturbedPhiAtZeroIsGEQID},
\end{empheq}
then the formula above holds in $W_s(0,T)$.

\noindent If also for $p \in [1,\infty)$,
\begin{empheq}[left={\empheqlbrace}]{equation}\label{ass:PhiDerivativeConvergenceInLpH}
\begin{aligned}
&\sup_{\lambda \in [0,1]}\frac{\norm{\Phi'(\psi + s\rho + \lambda h(s))h(s)}{L^p(0,T;H)}}{s} \to 0 \quad \text{ as $s \to 0^+$},\\
&\sup_{\lambda \in (0,1)}\frac{\norm{{L}(\Phi'(\psi+s\rho+\lambda h(s))(h(s)))}{L^2(0,T;V^*)}}{s} \to 0 \quad \text{ as $s \to 0^+$},
\end{aligned}
\end{empheq}
then 
\[\frac{r(s,\rho,h(s);\psi)}{s} \to 0 \quad \text{in $L^p(0,T;H)$ as $s \to 0$},\]
that is, $S_{z_0}\colon L^2(0,T;H) \times L^2(0,T;V) \to L^p(0,T;H)$ is directionally differentiable.


\end{prop}
\begin{proof}
Due to assumptions \eqref{ass:LPhiObstacleAndPerutbredInL2H} and \eqref{ass:perturbedObstacleInWszo}, the left-hand side of the expansion formula to be proved can be written using  \eqref{eq:relationS0AndS} in  Proposition \ref{prop:conditionsForFormulaToHold}:
\begin{equation}\label{eq:0}
S_{z_0}(f+sd, \psi + s\rho + h(s)) = \Phi(\psi+s\rho+h(s)) - \bar{S}_{w_0}({L}(\Phi(\psi + s\rho + h(s)))-f-sd),
\end{equation}
where $w_0 = \Phi(\psi+s\rho+h(s))(0)-z_0$. The first term on the right-hand side here can be expanded through the formula \eqref{eq:PhiHadamardFormula} for $\Phi$:
\begin{equation}
\Phi(\psi+s\rho+h(s)) = \Phi(\psi) + s\Phi'(\psi)(\rho) + \hat l(s,\rho, h(s); \psi)\label{eq:2}.
\end{equation}
This is an equality in $L^p(0,T;H)$ for all $p$ (and in fact in $W(0,T)$ by assumption). Note that \eqref{ass:LPhiObstacleAndPerutbredInL2H} implies that we can apply ${L}$ to all terms in \eqref{eq:2} and doing so yields
\[{L}\Phi(\psi+s\rho+h(s)) = {L}\Phi(\psi) + s{L}\Phi'(\psi)(\rho) + {L}\hat l(s,\rho, h(s); \psi) \in L^2(0,T;H).\]
Using this and the expansion formula \eqref{eq:S0HadamardFormula} for $\bar{S}$, the second term on the right-hand side of \eqref{eq:0} becomes
\begin{align}
\nonumber &\bar{S}_{w_0}({L}(\Phi(\psi + s\rho + h(s)))-f-sd) \\
\nonumber &\quad= \bar{S}_{w_0}({L}\Phi(\psi) - f) + s\bar{S}_{w_0}'({L}\Phi(\psi) - f)({L}\Phi'(\psi)(\rho) - d)\\
&\quad\quad  + \hat o (s,{L}\Phi'(\psi)(\rho) - d, {L}\hat l(s,\rho, h(s);\psi); {L}\Phi(\psi)-f),\label{eq:3}
\end{align}
where the second equality holds since every term inside $\bar{S}_{w_0}$ on the left-hand side is in $L^2(0,T;H)$ and so \eqref{eq:S0Formula} applies. Now, plugging \eqref{eq:2} and \eqref{eq:3} into \eqref{eq:0} we find
\begin{align*}
&S_{z_0}(f+sd, \psi + s\rho + h(s))\\
&= \Phi(\psi)+s\Phi'(\psi)(\rho)+\hat l(s,\rho,h(s);\psi) - \bar{S}_{w_0}({L}\Phi(\psi) - f)\\
&\quad - s\bar{S}_{w_0}'({L}\Phi(\psi) - f)({L}\Phi'(\psi)(\rho) - d)\\
&\quad -\hat o(s,{L}\Phi'(\psi)(\rho)-d, {L}\hat l(s,\rho,h(s);\psi); {L}\Phi(\psi)-f)\\
&= S_{z_0}(f,\psi) +s(\Phi'(\psi)(\rho)+\partial S_{z_0}(f,\psi)(d-{L}\Phi'(\psi)(\rho)))+\hat l(s,\rho,h(s);\psi)\\
&\quad -\hat o(s,{L}\Phi'(\psi)(\rho)-d, {L}\hat l(s,\rho,h(s);\psi); {L}\Phi(\psi)-f),
\end{align*}
where for the final equality, we again used the formula \eqref{eq:relationS0AndS} which is applicable because \eqref{ass:perturbedObstacleInWszo} implies that $z_0 \leq \Phi(\psi)(0)$, and we used the relation \eqref{eq:relation2} between the directional derivatives of $\bar{S}$ and $S$:
\[\bar{S}_{w_0}'({L}\Phi(\psi)-f)({L}\Phi'(\psi)(\rho)-d)=-\partial S_{z_0}(f,v)(d-{L}\Phi'(\psi)(\rho)).\] We then set $\alpha:= \Phi'(\psi)(\rho)+\partial S(f,\psi)(d-{L}\Phi'(\psi)(\rho))$. From \eqref{eq:christoffVI}, \eqref{eq:relationS0AndSDerivatives}, \eqref{eq:myVI}, the function $\delta :=\partial S(f,\psi)(d-{L}\Phi'(\psi)(\rho)) \in L^2(0,T;V)\cap L^\infty(0,T;H)$ satisfies 
\begin{equation*}
\begin{aligned}
&\delta\in T_{\mathbb{K}_0,L^2}^{\mathrm{tan}}(w) \cap [w' + Aw - ({L}\Phi(\psi)-f)]^\perp : \\
&\qquad\int_0^T  \langle \varphi' + A\delta - (d-{L}\Phi'(\psi)(\rho)), \delta-\varphi \rangle \leq \frac 12 \norm{\varphi(0)}{H}^2 \\
&\qquad\qquad\qquad\qquad\qquad\qquad\forall \varphi \in \cl_{W}(T_{\mathbb{K}_0}^{\mathrm{rad}}(w)\cap [w'+Aw-({L}\Phi(\psi)-f)]^\perp),
\end{aligned}
\end{equation*}
where $w=\bar{S}({L}\Phi(\psi)-f) = \Phi(\psi)-S(f,\psi)$.
Recalling the definition of $\alpha$ and making the substitution $\varphi := \Phi'(\psi)(\rho)+ z$ in the above variational formulation for $\delta$ yields the formulation for $\alpha$ stated in the proposition. If additionally \eqref{ass:perturbedDataNonnegativeAndInc}, \eqref{ass:perturbedPhiAtZeroIsGEQID}, \eqref{ass:perturbedObstacleNonNegative}, then Proposition \ref{prop:conditionsForFormulaToHold} gives the stated regularity.


Dropping now the dependence on the base points for clarity, we estimate the remainder term $r$ (which is defined as in the statement of this proposition) as follows, making use of  \eqref{eq:hatOBoundIntoL2VStar} and Lemma \ref{lem:AAA},
\begin{align*}
&\norm{r(s,\rho,h(s))}{L^p(0,T;H)} \\
&\quad\leq \norm{\hat l(s,\rho,h(s))}{L^p(0,T;H)} + \norm{\hat o(s,{L}\Phi'(\psi)(\rho)-d, {L}\hat l(s,\rho,h(s)))}{L^p(0,T;H)}\\
&\quad\leq \sup_{\lambda \in [0,1]}\norm{\Phi'(\psi + s\rho + \lambda h(s))h(s)}{L^p(0,T;H)} + \norm{l(s,\rho)}{L^p(0,T;H)}\\
&\quad\quad + \frac{T^{\frac 1p}}{\sqrt{C_a}}\norm{{L}\hat l(s,\rho,h(s))}{L^2(0,T;V^*)} + \norm{o(s,{L}\Phi'(\psi)(\rho)-d)}{L^p(0,T;H)}\\
&\quad\leq \sup_{\lambda \in [0,1]}\norm{\Phi'(\psi + s\rho + \lambda h(s))h(s)}{L^p(0,T;H)} + \norm{l(s,\rho)}{L^p(0,T;H)}\\
&\quad\quad + \frac{T^{\frac 1p}}{\sqrt{C_a}}\left(\sup_{\lambda \in (0,1)}\norm{{L}(\Phi'(\psi+s\rho+\lambda h(s))(h(s)))}{L^2(0,T;V^*)} + \norm{{L}l(s,\rho)}{L^2(0,T;V^*)} 
\right)\\
&\quad\quad + \norm{o(s,{L}\Phi'(\psi)(\rho)-d)}{L^p(0,T;H)}.
\end{align*}
Dividing by $s$ and taking the limit $s \to 0^+$, we see that the remainder term vanishes in the limit due to assumption \eqref{ass:PhiDerivativeConvergenceInLpH}.

Furthermore the convergence to zero is uniform in $d$ on compact subsets since $d$ appears only in the final term which we know has the same property as $S(\cdot, \psi)$ is Hadamard differentiable.
\end{proof}
\begin{remarks}
The assumption $h(0)=0$ in the proposition implies that all assumptions that hold for the perturbed data also hold for the non-perturbed data (i.e. at $s=0$). Without this assumption, we would have to assume in addition $\Phi(\psi) \in W_{\mathrm r}(0,T)$ and \eqref{ass:PhiIsNonNegativeAtPsi} and \eqref{ass:dataNonnegativeAndfInc} along with \eqref{ass:perturbedObstacleNonNegative}.
\end{remarks}
\section{Directional differentiability}\label{sec:DDMain}
Fix $f,d \in L^2(0,T;H)$. We begin by choosing an element of $\mathbf{P}(f)$ with sufficient regularity as described in the following assumption.
\begin{ass}\label{ass:onU}
Take $u_0 \in V_+$ and let $u \in \mathbf{P}_{u_0}(f) \cap W(0,T)$ be such that $t \mapsto \Phi(u)(t)$ is increasing.
\end{ass}
Refer to Theorems \ref{thm:discretisationExistence} and \ref{thm:approximationOfQVIByVIIterates} which give conditions for the existence of such a $u \in \mathbf{P}_{u_0}(f)$ and for the increasing property of $t \mapsto u(t)$; if $\Phi$ then preserves the increasing-in-time property then one would have the satisfaction of this assumption.

Picking $u \in \mathbf{P}_{u_0}(f)$ satisfying Assumption \ref{ass:onU}, define the sequence 
\begin{equation*}\label{eq:sequenceuns}
\begin{aligned}
u^n_s &:= S_{u_0}(f+sd, u^{n-1}_s) \quad\text{ for } n = 1, 2, 3, ...,\\
u^0_s &:= u.
\end{aligned}
\end{equation*}
Our aim will be to apply Theorem \ref{thm:approximationOfQVIByVIIterates} or Theorem \ref{thm:approximationOfQVIByVIIteratesDec} to this sequence in order to show that, under additional assumptions, it is well defined and has the right convergence properties. Furthermore, we also want to obtain expansion formulae for each $u^n_s$.
\subsection{Expansion formula for the VI iterates}
The following sets of assumptions are to ensure that Theorem \ref{thm:approximationOfQVIByVIIterates} (or Theorem \ref{thm:approximationOfQVIByVIIteratesDec}) can be applied for our sequence $\{u_s^n\}_{n \in \mathbb{N}}$.
\begin{ass}\label{ass:newAss}
Assume
\begin{align}
&d \leq 0 \text{ or } d \geq 0\label{ass:signOnTheDirection},\\
&\Phi'(u)\colon W_{\mathrm r}(0,T) + L^2(0,T;V) \cap L^\infty(0,T;H) \to  W_{\mathrm r}(0,T),\tag{L2}\label{ass:inL2VLinftyHimples}\\
\nonumber &\intertext{if $\rho \in L^2(0,T;V),$ and  $h\colon (0,1) \to L^2(0,T;H)$ is a higher-order term, then for some $p \in [1,\infty)$, as $s \to 0^+$,} 
&\qquad\qquad\begin{cases}\sup_{\lambda \in [0,1]}s^{-1}\norm{\Phi'(u + s\rho + \lambda h(s))h(s)}{L^p(0,T;H)} \to 0, \\
\sup_{\lambda \in [0,1]}s^{-1}\norm{{L}(\Phi'(u+s\rho+\lambda h(s))(h(s)))}{L^2(0,T;V^*)} \to 0,
\end{cases}\tag{L3}\label{ass:LFour}
\end{align}
and either 
\begin{empheq}[left={\empheqlbrace}]{align}
\nonumber &\eqref{ass:perturbedDataNonnegativeAndInc}, 
\eqref{ass:PhiNonNegativeForNonNegativeObs}, \eqref{ass:completeContinuityOfPhi},\\
&\langle Au_0 - f(t)- sd(t), v \rangle \leq 0 \text{ $\forall v \in V_+$ and for a.e. $t$,}\label{ass:NEWASSconditionOnIDPerturbed}\\
&\Phi \colon W_s(0,T) \to W_{\mathrm r}(0,T) \cap C^0([0,T];V),\label{ass:inWImpliesLPhiInL2H}\tag{O3a}\\
&t \mapsto \Phi(w)(t) \text{ is increasing $\forall w \in L^2(0,T;V_+) : t \mapsto w(t)$ increasing,}\tag{O4a}\label{ass:PhiIsIncreasingInTimeForAllNonNegBochner}\\
&\text{$u \in W_s(0,T)$ with $\Phi(u) \geq 0$,}\tag{L1a}\label{ass:L2}\\
&\text{if $d \leq 0$, $w(0) = u_0 \implies u_0 \leq \Phi(w)(0)\; \forall w \in W_s(0,T)$,}\label{ass:NEWASSFOR}
\end{empheq}
or  
\begin{empheq}[left={\empheqlbrace}]{align}
\nonumber &\eqref{ass:PhiNonNegativeForAllPsiInBochnerSpace}, \eqref{ass:strongerCompleteContinuityOfPhi}, \eqref{ass:T2},\\
&t \mapsto \Phi(w)(t) \text{ is increasing $\forall w \in L^2(0,T;V)$,}\tag{O4b}\label{ass:PhiIsIncreasingInTimeForAllBochner}\\
&\Phi(u) \in W_{\mathrm r}(0,T) \text{ and } \Phi(u)(0)\geq u_0,\tag{L1b}\label{ass:LZero}\\
\nonumber &\text{Either } w \in W_{\mathrm r}(0,T) : w(0) = z_0 \implies  z_0 \leq  \Phi(w)(0)\text{, or, } \\
\nonumber &\text{if $d \geq 0$, }\Phi(v) \leq \Phi(w) \implies  \Phi(v)(0) \leq \Phi(w)(0),\\
&\text{whereas if $d \leq 0$, }  \Phi(v) \geq \Phi(w) \implies  \Phi(v)(0) \leq \Phi(w)(0).
\label{ass:NEWASS22}
\end{empheq}
\end{ass}
\begin{remark}
Regarding \eqref{ass:signOnTheDirection}, observe that if $d \geq 0$, the initial element $u_s^0 = u$ is a subsolution for $\mathbf{P}(f+sd)$ since $u = S(f,u) \leq S(f+sd, u)$ (see also Lemma \ref{lem:subsolutionForPerturbedQVI}) whilst if $d \leq 0$ then $u_s^0$ is instead a supersolution.
\end{remark}

Define $\alpha^1 = \delta^1 := \partial S(f,u)(d)$ and for $n\geq 2$, we make the recursive definitions:
\begin{align}
\nonumber \delta^n &:= \partial S(f,u)(d-{L}\Phi'(u)(\alpha^{n-1})),\\
\alpha^n &:=
\Phi'(u)[\alpha^{n-1}] + \delta^{n},\label{eq:idForAlphaN}\\
o^n(s) &:= r(s,\alpha^{n-1},o^{n-1}(s)).\label{eq:identityforOnAndRn}
\end{align}
To ease the notation on the higher-order terms, we did not write the base point $u$ in the $r$ term (which originates from Proposition \ref{prop:formulaForSDiff}) above. 
We now give two results (with varying assumptions) in the next proposition concerning convergence behaviour and an expansion formula for the sequence $\{u_s^n\}$.

\begin{prop}\label{prop:expansionFormulan}
Let Assumption \ref{ass:newAss} hold. 
Then $\{u_s^n\}_{n \in \mathbb{N}} \subset W(0,T)$ is a well defined non-negative monotone sequence (increasing if $d \geq 0$, decreasing if $d \leq 0$) such that
\begin{equation}\label{eq:eqForqn}
u_s^n = u + s\alpha^n + o^n(s), \quad \text{as $s \downarrow 0$},
\end{equation}
where
\begin{enumerate}[label={(\arabic*)}]
\item $\alpha^n$ satisfies the VI
\begin{equation}
\begin{aligned}
&\alpha^n - \Phi'(u)(\alpha^{n-1}) \in T^{\tan}_{\mathbb{K}_0, L^2}(\Phi(u)-u)\cap [u' + Au - f]^\perp :\\ 
& \qquad\qquad\int_0^T \langle \varphi' + A\alpha^n - d, \alpha^n - \varphi \rangle \leq
 \frac 12 \norm{\varphi(0)-\Phi'(u)\alpha^{n-1}(0)}{H}^2 \\
&   \qquad \qquad  \forall \varphi : \varphi- \Phi'(u)(\alpha^{n-1}) \in \cl_{W}(T^{\mathrm{rad}}_{\mathbb{K}_0}(\Phi(u)-u)\cap [u' + Au - f]^\perp);
\end{aligned}\label{eq:alphanTestFunctionSpace}
\end{equation}
\item ${L}\Phi(u_s^n) \in L^2(0,T;H)$, $\alpha^n-\Phi'(u)(\alpha^{n-1}) \in L^2(0,T;V) \cap L^\infty(0,T;H)$, ${L}\Phi'(u)(\alpha^{n-1}) \in L^2(0,T;H);$ 
\item $s^{-1}o^n(s) \to 0$ in $L^p(0,T;H)$ as $s \to 0^+$ with $p \in [1,\infty)$ from Assumption \ref{ass:newAss};
\item under the first set of assumptions,
\[u_s^n \weaklyto u_s \text{ in $W_s(0,T)$ where $u_s^n \geq 0$ and $u_s \in \mathbf{P}_{u_0}(f+sd)$ is a non-negative  solution};\]
\item under the second set of assumptions, $u_s^n \in W_{\mathrm r}(0,T)$  and
\begin{align*}
&u_s^n \weaklyto u_s \text{ in $L^2(0,T;V)$ and weakly-star in $L^\infty(0,T;H)$ where}\\
&\text{$u_s \in \mathbf{P}_{u_0}(f+sd)$ is a very weak solution.}
\end{align*}
\end{enumerate}
\end{prop}
\begin{proof}
Let us first show monotonicity of the sequence assuming existence. First take $d \geq 0$. Then since $u$ is a subsolution, $u \leq S(f,u) \leq S(f+sd, u) = u^1_s$. By the comparison principle, we find again that $u^1_s = S(f+sd, u) \leq S(f+sd, u^1_s) = u^2_s$ and in this we obtain that $\{u^n_s\}$ is an increasing sequence. Likewise if $d \leq 0$, the sequence is decreasing. 

\medskip

\noindent \textsc{1. First case}.  Observe that since $u \leq \Phi(u)$ and $\Phi(u) \in W_s(0,T)$, by \eqref{ass:inWImpliesLPhiInL2H}, we can take the trace at $t=0$ to get $\Phi(u)(0) \geq u_0$. Since \eqref{ass:perturbedDataNonnegativeAndInc}, \eqref{ass:L2} and \eqref{ass:NEWASSconditionOnIDPerturbed} hold and as \eqref{ass:PhiIsNonNegativeAtPsi} is satisfied for the obstacle $u$ (thanks to \eqref{ass:L2} and \eqref{ass:inWImpliesLPhiInL2H}), Proposition \ref{prop:existenceForParabolicVIObstacle} implies that $u_s^1=S(f+sd,u) \in W_s(0,T)$ exists and is increasing in time and non-negative.

Regarding the upper bound for the initial data in terms of $u_s^1$, we argue as follows. For $d \geq 0$, we may apply $\Phi$ to the inequality $u_s^1 \geq S(f,u) = u$ and use the regularity offered by \eqref{ass:inWImpliesLPhiInL2H} to obtain $u_0 \leq \Phi(u)(0) \leq \Phi(u_s^1)(0).$ If $d \leq 0$, we use the condition \eqref{ass:NEWASSFOR} to obtain the same conclusion.  Then making use of \eqref{ass:PhiNonNegativeForNonNegativeObs}, \eqref{ass:inWImpliesLPhiInL2H}, and \eqref{ass:PhiIsIncreasingInTimeForAllNonNegBochner} we apply Proposition \ref{prop:existenceForParabolicVIObstacle} to  obtain the existence for each $u^2_s$  and subsequently, using these arguments, existence for each $u^n_s$. 

\medskip

We now show the expansion formula \eqref{eq:eqForqn} by induction.

\medskip

\noindent \textit{1.1 Base case.} Using Proposition \ref{prop:hadamardDifferentiabilityForPVI} to expand $u^1_s = S(f+sd, u)$, which is applicable due to \eqref{ass:L2}, \eqref{ass:inWImpliesLPhiInL2H}, and the increasing-in-time property and the non-negativity of $u_0$ from Assumption \ref{ass:onU}, we obtain a
$\delta^1 := \partial S(f,u)(d) \in L^2(0,T;V) \cap L^\infty(0,T;H)$ such that
\begin{equation}\label{eq:idForu1}
u^1_s = S(f+sd, u) = u + s\delta^1 + o(s,d),
\end{equation}
where 
\begin{equation*}
\begin{aligned}
&\delta^1  \in T_{\mathbb{K}_0,L^2(0,T;V)}^{\mathrm{tan}}(w) \cap [u' + Au - f]^\perp : \int_0^T  \langle \varphi' + A\delta^1 - d, \delta^1-\varphi \rangle  \leq \frac 12 \norm{\varphi(0)}{H}^2\\
&\quad \quad \quad \quad \quad \quad\quad \quad \quad \quad \quad  \quad \quad \quad \forall \varphi \in \cl_{W}(T_{\mathbb{K}_0}^{\mathrm{rad}}(\Phi(u)-u)\cap [u'+Au-f]^\perp),
\end{aligned}
\end{equation*}
and 
\[o^1=r(\cdot, 0,0)=\hat l(\cdot,0,0) - \hat o(\cdot,-d,-d) = l(\cdot,0)-o(\cdot,-d)\]
is clearly a higher-order term. Finally, assumption \eqref{ass:inWImpliesLPhiInL2H} implies that ${L}\Phi(u_s^1) \in L^2(0,T;H)$, whilst \eqref{ass:inL2VLinftyHimples} gives ${L}(\Phi'(u)(\delta^1)) \in L^2(0,T;H)$. 

\medskip

\noindent \textit{1.2 Inductive step}. Now assume the statement is true for $n=k$.  By definition,
\begin{equation}\label{eq:idForus}
u_s^{k+1} := S(f+sd, u_s^k) = S(f+sd, u + s\alpha^k + o^k(s)).
\end{equation}
This object is again non-negative since $u_s^k \geq 0$ implies that $\Phi(u_s^k) \geq \Phi(0) \geq 0$ by \eqref{ass:PhiNonNegativeForNonNegativeObs}.
We have $u_s^k \in W_s(0,T)$, and thus by \eqref{ass:inWImpliesLPhiInL2H}, $\Phi(u_s^k) \in W_{\mathrm r}(0,T)$, and since $\alpha^k = \Phi'(u)(\alpha^{k-1}) + \delta^k \in W_{\mathrm r}(0,T) + L^2(0,T;V)\cap L^\infty(0,T;H)$, \eqref{ass:inL2VLinftyHimples} implies that $\Phi'(u)(\alpha^k) \in W_{\mathrm r}(0,T)$. Hence \eqref{ass:LPhiObstacleAndPerutbredInL2H} holds for obstacle $u$, direction $\alpha^k$ and higher-order term $o^k(s)$. By \eqref{ass:PhiIsIncreasingInTimeForAllNonNegBochner}, the obstacle $\Phi(u^k_s)$ is increasing in time. Then, since $\Phi$ is increasing, 
\[\Phi(u_k^s)(0) = \Phi(u+s\alpha^k + o^k(s))(0) \geq \Phi(u)(0) \geq u_0.\]
Proposition \ref{prop:formulaForSDiff} can now be applied and we find
\begin{align*}
u^{k+1}_s &= u + s(\Phi'(u)(\alpha^k) + \partial S_{u_0}(f,u)(d-{L}\Phi'(u)(\alpha^k))) + r(s,\alpha^k, o^k(s))\\
&= u+ s(\Phi'(u)(\alpha^k) + \delta^{k+1})) + r(s,\alpha^k, o^k(s))\\
&= u + s\alpha^{k+1} + o^{k+1}(s)
\end{align*}
with $u_s^{k+1} \in W_s(0,T) \cap W_{\mathrm r}(0,T)$ and $\delta^{k+1}=\alpha^{k+1}-\Phi'(u)(\alpha^k) \in L^\infty(0,T;H)\cap L^2(0,T;H)$ (we already argued above that ${L}\Phi'(u)(\alpha^{k}) \in L^2(0,T;H)$), meaning that $\alpha^{k+1} \in W_{\mathrm r}(0,T) + L^2(0,T;V)\cap L^\infty(0,T;H)$ as desired. Under Assumption \eqref{ass:LFour}, by the same argument as in the proof of Proposition \ref{prop:formulaForSDiff}, $o^{k+1}$ is a higher-order term given that $o^k$ is.

Regarding the expression for the derivative, 
we know that $\alpha^{k+1}-\Phi'(u)(\alpha^k) = \partial S(f,u)(d-L\Phi'(u)(\alpha^k))$ solves the VI that appears in Proposition \ref{prop:hadamardDifferentiabilityForPVI}, i.e.,
\begin{equation*}
\begin{aligned}
&\alpha^{k+1}-\Phi'(u)(\alpha^k)  \in T_{\mathbb{K}_0,L^2(0,T;V)}^{\mathrm{tan}}(w) \cap [u' + Au - f]^\perp :\\ &\int_0^T  \langle \phi' + A(\alpha^{k+1}-\Phi'(u)(\alpha^k)) - d + L\Phi'(u)(\alpha^k), \alpha^{k+1}-\Phi'(u)(\alpha^k)-\phi \rangle \\
& \quad \quad \quad \quad \quad\quad \quad \quad \quad \quad \quad \quad  \quad  \quad \quad \quad \quad \quad\quad \quad \quad \quad \quad \quad \quad  \quad  \leq \frac 12 \norm{\phi(0)}{H}^2,\\
&\quad \quad \quad \quad \quad\quad \quad \quad \quad \quad \quad \quad  \quad  \forall \phi \in \cl_{W}(T_{\mathbb{K}_0}^{\mathrm{rad}}(\Phi(u)-u)\cap [u'+Au-f]^\perp),
\end{aligned}
\end{equation*}
whence setting $\varphi:= \phi + \Phi'(u)(\alpha^k)$ yields
\begin{equation*} 
\begin{aligned}
&\alpha^{k+1}-\Phi'(u)(\alpha^k)  \in T_{\mathbb{K}_0,L^2(0,T;V)}^{\mathrm{tan}}(w) \cap [u' + Au - f]^\perp : \\&\int_0^T  \langle \varphi'  + A\alpha^{k+1} - d, \alpha^{k+1}-\varphi \rangle  \leq \frac 12 \norm{\varphi(0)-\Phi'(u)(\alpha^k)(0)}{H}^2\\
&  \quad \quad \quad \quad \quad  \quad \quad \quad \forall \varphi : \varphi -\Phi'(u)(\alpha^k)\in \cl_{W}(T_{\mathbb{K}_0}^{\mathrm{rad}}(\Phi(u)-u)\cap [u'+Au-f]^\perp)
\end{aligned}
\end{equation*}
as desired.

\medskip

\noindent \textsc{2. Second case}. We will not repeat some of the same techniques used in the above case and simply focus on the differences under the different set of assumptions. Due to \eqref{ass:LZero}, $u^1_s$ exists by Proposition \ref{prop:existenceForParabolicVIObstacle}. Using \eqref{ass:NEWASS22} we find $u_0 \leq \Phi(u^1_s)(0)$. The monotonicity of $\{u_s^n\}_{n \in \mathbb{N}}$ and \eqref{ass:NEWASS22} shows this bound for all $u^n_s$. Using  \eqref{ass:T2}, \eqref{ass:PhiIsIncreasingInTimeForAllBochner} and \eqref{ass:NEWASS22}, we infer the existence for all $u_s^n$ by the same proposition.

We prove the remaining claims again by induction. For the base case, we can expand $u_s^1=S(f+sd,u)$ by using \eqref{ass:LZero} and Proposition \ref{prop:hadamardDifferentiabilityForPVI} directly and we obtain $\delta^1 := \partial S(f,u)(d) \in L^2(0,T;V) \cap L^\infty(0,T;H)$ such that \eqref{eq:idForu1} holds. 
Furthermore, $u^1_s \in W_{\mathrm r}(0,T)$. Now assume the statement is true for $n=k$ so that \eqref{eq:idForus} holds. 
Since $u^k_s \in W_{\mathrm r}(0,T)$, ${L}\Phi(u_s^k) \in L^2(0,T;H)$, and by \eqref{ass:inL2VLinftyHimples}, since 
\[\alpha^k = \alpha^k - \Phi'(u)(\alpha^{k-1}) + \Phi'(u)(\alpha^{k-1}) \in L^2(0,T;V) \cap L^\infty(0,T;H) + W_{\mathrm r}(0,T),\] we have ${L}\Phi'(u)(\alpha^k) \in L^2(0,T;H)$ and so assumption \eqref{ass:LPhiObstacleAndPerutbredInL2H} of Proposition \ref{prop:formulaForSDiff} holds, and as does \eqref{ass:perturbedObstacleInWszo} as shown above, and the proposition can applied to give
\begin{align*}
u^{k+1}_s 
&= u + s\alpha^{k+1} + o^{k+1}(s)
\end{align*}
(just like in the proof of the first case) with $u_s^{k+1} \in W_{\mathrm r}(0,T)$ and $\alpha^{k+1}-\Phi'(u)(\alpha^k) \in L^2(0,T;V)\cap L^\infty(0,T;H)$ and ${L}\Phi'(u)(\alpha^{k}) \in L^2(0,T;H)$ as desired. Note that we used the fact that \eqref{ass:PhiIsIncreasingInTimeForAllBochner} implies the increasing property of all obstacles considered in the proof.

\medskip

\noindent \textsc{3. Conclusion.}
The claim of the VI satisfied by the $\alpha^n$ follows from Proposition \ref{prop:formulaForSDiff} whilst the convergence behaviour stated in the result is a consequence of either Theorem \ref{thm:approximationOfQVIByVIIterates} (if $d \geq 0$) or Theorem \ref{thm:approximationOfQVIByVIIteratesDec} (if $d \leq 0$), using the fact that \eqref{ass:signOnTheDirection} implies that $u^0_s$ is either a subsolution or supersolution.
\end{proof}
\begin{remark}
Everything up to the convergence of the $\{u_s^n\}$ stated in the above result holds if we do not assume \eqref{ass:signOnTheDirection} and either \eqref{ass:completeContinuityOfPhi} or \eqref{ass:strongerCompleteContinuityOfPhi} respectively. Also, \eqref{ass:LFour} was necessary only to prove that each $o^n$ is a higher-order term.
\end{remark}
\subsection{Properties of the iterates}
In this section, we give some basic attributes of the directional derivatives $\alpha^n$ and the higher-order terms $o^n$. One should not forget that these objects are time-dependent, and we will always denote the time component by $t$; this should not be confused with the perturbation parameter $s$.
\begin{lem}\label{lem:properties}The following properties hold:
\begin{enumerate}
\item For a.e. $t \in [0,T]$,
\begin{equation*}
\begin{aligned}
\alpha^1(t) &\geq 0, &&\text{ q.e. on $\{u(t)=\Phi(u)(t)\}$}, \\
\alpha^n(t) &\geq \Phi'(u)(\alpha^{n-1})(t), &&\text{ q.e. on $\{u(t)=\Phi(u)(t)\}$ for $n > 1$.}
\end{aligned}
\end{equation*}
\item The sequences 
\[\{\alpha^n\}_{n \in \mathbb{N}} \text{ and } \{\alpha^n + s^{-1}o^n(s)\}_{n \in \mathbb{N}}\]
are monotone (increasing if $d \geq 0$ and decreasing if $d \leq 0$) and have the same sign as $d$.
\item $(\alpha^n+s^{-1}o^n(s))|_{t=0} = 0.$
\item $\Phi'(u)(\alpha^n)$ has the same sign as $d$.
\end{enumerate}
\end{lem}
\begin{proof}
The first claim follows from the set that the $\alpha^n$ belong to and the characterisation of Lemma \ref{lem:tangentConeK0}. The second claim is true since the sequence $u^n_s$ is increasing or decreasing in $n$ and due to \eqref{eq:eqForqn} and the vanishing behaviour of $s^{-1}o^n(s)$ and the fact that $u^n_s-u$ has the same sign as $d$ (since if $d \geq 0$, $u^n_s$ is increasing and hence greater than $u$ whilst if $d \leq 0$, $u^n_s$ is decreasing and smaller than $u$). For the third claim, in  \eqref{eq:eqForqn}, take the trace $t=0$ (which is valid since $u^n_s, u$ were defined to have trace $u_0$ at $t=0$) to obtain
\[u_0 = u_0 + (s\alpha^n+ o^n(s))|_{t=0}.\]
Finally, we have that
\[\Phi'(u)(\alpha^n) = \lim_{h \to 0^+}\frac{\Phi(u+h\alpha^n) - \Phi(u)}{h},\]
where the limit is in $L^p(0,T;H)$, and hence, passing to a subsequence, the limit also holds at almost every time strongly in $H$:
\[\Phi'(u)(\alpha^n)(t) = \lim_{h_j \to 0^+}\frac{\Phi(u+h_j\alpha^n)(t) - \Phi(u)(t)}{h_j},\]
which is either greater than or less than zero depending on the sign of $\alpha^n$ which in turn depends on the sign of $d$ (see part 2 of this lemma).
\end{proof}
The first part of the previous result tells us about the quasi-everywhere behaviour of the directional derivatives on the coincidence set. We can say a little more about them in an almost everywhere sense.
\begin{lem}\label{lem:alphaNOnI}We have
\[\alpha^1 \leq 0 \quad \text{a.e. on $\{u = \Phi(u)\}$ with equality if $d \geq 0$.}\]
If $\Phi$ is a superposition operator, then for each $n$,
\[\alpha^n \leq 0 \quad\text{ a.e. on $\{u = \Phi(u)\}$ with equality if $d \geq 0$.}\]
\end{lem}
\begin{proof}
From $s\alpha^n = u^n_s-u-o^n(s)$, since $u^n_s \leq \Phi(u^{n-1}_s) \leq ... \leq \Phi^n(u^0_s) = \Phi^n(u)$, we find
\begin{equation}\label{eq:4a}
s\alpha^n \leq \Phi^n(u)-u - o^n(s).
\end{equation}
On the set $\{u=\Phi(u)\}$, we get $s\alpha^1 \leq -o^1(s)$
and dividing here by $s$ and sending to zero, we see by the sandwich theorem that if $d \geq 0$, $\alpha^1 = 0$ on $\{u=\Phi(u)\}$. If $\Phi$ is a superposition operator, observe that if $t$ is such that $u(t) = \Phi(u(t))$, then in fact 
\begin{equation*}
u(t) = \Phi^n(u(t))\quad\text{for any $n \in \mathbb{N}$}.
\end{equation*}
Using this fact on the right-hand side of \eqref{eq:4a} gives us the result.
\end{proof}
\subsection{Uniform bounds on the iterates}
We give a result on the boundedness of the directional derivatives $\alpha^n$ under two different sets of assumptions. The first set requires some boundedness conditions on the obstacle mapping including a smallness condition, whilst the second requires instead some regularity and complementarity (for the latter, see \cite[\S 7.3.1]{MR3726880} for the parabolic VI case) for the system. 

\begin{ass}\label{ass:forBoundednessOfAlphas}
Suppose that either
\begin{empheq}[left={\empheqlbrace}]{align}
\norm{\Phi'(u)(\alpha^{n-1})}{L^2(0,T;V)} &\leq C_1^*\norm{\alpha^{n-1}}{L^2(0,T;V)},\label{ass:lb1}\tag{L4a}\\
\norm{\partial_t \Phi'(u)(\alpha^{n-1})}{L^2(0,T;V^*)} &\leq C_2^*\norm{\alpha^{n-1}}{L^2(0,T;V)},\label{ass:lb2}\tag{L5a}\\
\norm{\Phi'(u)(\alpha^{n-1})(T)}{H}^2 &\leq C_3^*\norm{\alpha^{n-1}}{L^2(0,T;V)}^2 + C,\label{eq:boundednessOfAlphaDerivativeAtT}\tag{L6a}\\
C_bC_1^*+C_2^* + C_3^* &< C_a,\tag{L7a}\label{ass:smallnessCondition}
\end{empheq}
or 
\begin{empheq}[left={\empheqlbrace}]{align}
(u'+Au-f)(u-\Phi(u)) &= 0 \text{ a.e. on $(0,T)\times \Omega$},\label{eq:complimentarity}\tag{L4b}\\
\Phi'(u)(\alpha^n) &= 0 \text{ a.e. on $\{u=\Phi(u)\}$},\label{eq:derivativesZeroOnCoincidenceSet}\tag{L5b}\\
\norm{\Phi'(u)(\alpha^n)(0)}{H} &\leq C,\label{eq:boundednessOfAlphaDerivativeAtZero}\tag{L6b}
\end{empheq}
where all constants are independent of $n$.
\end{ass}
Regarding the fulfillment of assumption \eqref{eq:derivativesZeroOnCoincidenceSet}, Lemma \ref{lem:properties} may be helpful for certain choices of $\Phi$ in applications.
\begin{prop}[Bound on $\{\alpha^n\}$]\label{prop:boundOnAlphaN}
Let Assumption \ref{ass:forBoundednessOfAlphas} hold.  Then $\alpha^n \weaklyto \alpha \text{ in $L^2(0,T;V)$}.$
\end{prop}
\begin{proof}
\noindent \textsc{1. Under boundedness assumptions.}  First suppose that \eqref{ass:lb1} --- 
\eqref{ass:smallnessCondition} hold. In \eqref{eq:alphanTestFunctionSpace} take $\varphi=\Phi'(u)(\alpha^{n-1})$ as test function (this is admissible since zero is contained in the radial cone and the orthogonal space that the test function space is obtained from) to get
\[\int_0^T \langle \partial_t \Phi'(u)(\alpha^{n-1}) + A\alpha^n - d, \alpha^n - \Phi'(u)(\alpha^{n-1})\rangle \leq 0.\]
We can neglect the term $(d,\Phi'(u)(\alpha^{n-1}))_H$ due to part 4 of Lemma \ref{lem:properties}  which tells us that both $d$ and $\Phi'(u)(\alpha^n)$ have the same sign. Hence the above inequality becomes
\begin{align*}
&C_a\norm{\alpha^n}{L^2(0,T;V)}^2\\
&\leq \int_0^T \langle A\alpha^n, \Phi'(u)(\alpha^{n-1})\rangle + (d,\alpha^n)_H  - \langle \partial_t \Phi'(u)(\alpha^{n-1}), \alpha^n - \Phi'(u)(\alpha^{n-1})\rangle\\
&\leq C_b\norm{\alpha^n}{L^2(0,T;V)}\norm{\Phi'(u)(\alpha^{n-1})}{L^2(0,T;V)} + \norm{d}{L^2(0,T;H)}\norm{\alpha^n}{L^2(0,T;H)}\\
&\quad+ \norm{\partial_t \Phi'(u)(\alpha^{n-1})}{L^2(0,T;V^*)}\norm{\alpha^n}{L^2(0,T;V)}  + \frac 12 \norm{\Phi'(u)(\alpha^{n-1})(T)}{H}^2\\
&\quad - \frac 12 \norm{\Phi'(u)(\alpha^{n-1})(0)}{H}^2\\
&\leq C_bC_1^*\norm{\alpha^n}{L^2(0,T;V)}\norm{\alpha^{n-1}}{L^2(0,T;V)} + \norm{d}{L^2(0,T;H)}\norm{\alpha^n}{L^2(0,T;V)}  \\
&\quad + C_2^*\norm{\alpha^{n-1}}{L^2(0,T;V)}\norm{\alpha^n}{L^2(0,T;V)} + C_3^*\norm{\alpha^{n-1}}{L^2(0,T;V)}^2 + C\tag{by \eqref{ass:lb1}, \eqref{ass:lb2} and \eqref{eq:boundednessOfAlphaDerivativeAtT}}\\
&= (C_bC_1^*+C_2^*)\norm{\alpha^n}{L^2(0,T;V)}\norm{\alpha^{n-1}}{L^2(0,T;V)} + \norm{d}{L^2(0,T;H)}\norm{\alpha^n}{L^2(0,T;V)}\\
&\quad + C_3^*\norm{\alpha^{n-1}}{L^2(0,T;V)}^2 + C\\
&\leq \frac{(C_bC_1^*+C_2^*)}{2}\left(\norm{\alpha^n}{L^2(0,T;V)}^2 + \norm{\alpha^{n-1}}{L^2(0,T;V)}^2\right) + C_\rho\norm{d}{L^2(0,T;H)}^2 + \rho\norm{\alpha^n}{L^2(0,T;V)}^2\\
&\quad + C_3^*\norm{\alpha^{n-1}}{L^2(0,T;V)}^2 + C.
\end{align*}
Defining $a_n := \norm{\alpha^n}{L^2(0,T;V)}$, this reads
\begin{align*}
\left(C_a - \frac 12(C_bC_1^*+C_2^*) - \rho\right)a_n^2 \leq \left(\frac 12(C_bC_1^*+C_2^*)  + C_3^*\right) a_{n-1}^2 + C_\rho \norm{d}{L^2(0,T;H)}^2 + C
\end{align*}
which we write as
\begin{align*}
a_n^2 \leq \frac{A_2}{A_1} a_{n-1}^2 + \frac{C_\rho \norm{d}{L^2(0,T;H)}^2 + C}{A_1}
\end{align*}
where we have denoted
\[A_1 := C_a - \frac 12(C_bC_1^*+C_2^*) - \rho \quad\text{and}\quad A_2:= \frac 12(C_bC_1^*+C_2^*) +  C_3^* .\]
Solving this recurrence inequality leads to 
\begin{align*}
a_n^2 \leq \left(\frac{A_2}{A_1} \right)^{n-1}a_{1}^2 + \frac{C_\rho \norm{d}{L^2(0,T;H)}^2 + C}{A_1}\left(\frac{1-\left(\frac{A_2}{A_1}\right)^{n-1}}{1-\frac{A_2}{A_1}}\right).
\end{align*}
We evidently need $A_2 < A_1$ for this sequence to be bounded uniformly, that is, 
\begin{align*}
\frac 12(C_bC_1^*+C_2^*)  + C_3^* < C_a - \frac 12(C_bC_1^*+C_2^*) - \rho  \iff 
C_bC_1^*+C_2^* + C_3^* < C_a - \rho 
\end{align*}
i.e., \eqref{ass:smallnessCondition}. Under this condition, the bound is uniform and there is a weak limit for a subsequence of $\{\alpha^n\}_{n \in \mathbb{N}}$. Since the $\alpha^n$ are monotone, they have a pointwise a.e. monotone limit which must agree with $\alpha$ so indeed $\alpha^n \weaklyto \alpha$ in $L^2(0,T;V)$. 

\medskip

\noindent \textsc{2. Under regularity assumptions.} Now assume instead that \eqref{eq:complimentarity} --- 
\eqref{eq:boundednessOfAlphaDerivativeAtZero} hold. We want to show that $\varphi \equiv 0$ is a valid test function in the VI \eqref{eq:alphanTestFunctionSpace} for $\alpha^n$. Thus we need to prove that $-\Phi'(u)(\alpha^{n-1}) \in \cl_{W}(T^{\mathrm{rad}}_{\mathbb{K}_0}(\Phi(u)-u)\cap [u'+Au-f]^\perp)$. On this note, observe that 
\begin{align*}
\int_0^T \int_\Omega (u'+Au-f)\Phi'(u)(\alpha^{n-1}) &= \int_{\{u=\Phi(u)\}} (u'+Au-f)\Phi'(u)(\alpha^{n-1})\\
&\quad + \int_{\{u < \Phi(u)\}} (u'+Au-f)\Phi'(u)(\alpha^{n-1})\\
&= \int_{\{u=\Phi(u)\}} (u'+Au-f)\Phi'(u)(\alpha^{n-1}) \tag{by \eqref{eq:complimentarity}}\\
&=0 \tag{by \eqref{eq:derivativesZeroOnCoincidenceSet}}.
\end{align*}
The assumption \eqref{eq:derivativesZeroOnCoincidenceSet} implies that $-\Phi'(u)(\alpha^{n-1}) \geq 0$ a.e. on $\{u=\Phi(u)\}$ and thus it belongs to $T^{\mathrm{rad}}_{\mathbb{K}_0}(u-\Phi(u)) \cap [u'+Au-f]^\perp$ (see \eqref{eq:dd1})  and this is obviously a subset of its closure in $W$. Therefore, $0$ is a valid test function in \eqref{eq:alphanTestFunctionSpace} and testing with this we find
\[\int_0^T \langle A\alpha^n - d, \alpha^n \rangle \leq \frac 12 \norm{\Phi'(u)(\alpha^{n-1})(0)}{H}^2\]
which easily leads to the desired bound due to the assumption \eqref{eq:boundednessOfAlphaDerivativeAtZero}. As before, the monotonicity of the sequence implies the convergence for the whole sequence.
\end{proof}

\subsection{Characterisation of the limit of the directional derivatives}
We now want to study the limiting objects associated to the sequences $\{\alpha^n\}$ and $\{\delta^n\}$. First, we introduce the notation $B_\epsilon^q(u)$ to stand for a closed ball in $L^q(0,T;H)$ around $u$ of radius $\epsilon$
\[B_\epsilon^q(u) := \{ v \in L^q(0,T;H) : \norm{v-u}{L^q(0,T;H)} \leq \epsilon\},\] 
and introduce some assumptions that will be of use here and in further sections.
\begin{ass}\label{ass:requiredLocalAssumptions}Suppose that
\begin{align}
&\Phi'(u)(\cdot)\colon L^2(0,T;V) \to L^2(0,T;V) \text{ is completely continuous},\tag{L8}\label{ass:compContOfDerivOfPhi}\\
&{L}(\Phi'(u)(\cdot))\colon L^2(0,T;V) \to L^2(0,T;H) \text{ is completely continuous},\tag{L9}\label{ass:strongCCofDerivative}
\end{align}
and assume that there exists $\epsilon > 0$ such that for all $z \in L^2(0,T;H) \cap B_\epsilon^p(u) + B_\epsilon^2(0)$ and $v \in L^2(0,T;V)\cap L^p(0,T;H)$, 
\begin{align}
\norm{\Phi'(z)v}{L^p(0,T;H)} &\leq K_1^*\norm{v}{L^p(0,T;H)},\tag{L10}\label{ass:big0}\\
\norm{\Phi'(z)v}{L^2(0,T;V)} &\leq K_2^*\norm{v}{L^p(0,T;H)},\tag{L11}\label{ass:bigA}\\
\norm{\partial_t(\Phi'(z)v)}{L^2(0,T;V^*)} &\leq K_3^*\norm{v}{L^p(0,T;H)},\tag{L12}\label{ass:bigB}\\
\intertext{where}
K_1^*+\frac{T^{\frac 1p}(K_2^*C_b+K_3^*)}{\sqrt{C_a}} &< 1\label{ass:smallnessOfDerivOfPhi}.\tag{L13}
\end{align}
Here, $p \in [1,\infty)$ is as in Assumption \ref{ass:newAss}.
\end{ass}
Regarding \eqref{ass:compContOfDerivOfPhi}, it may be helpful to note that Assumption \ref{ass:PhiHadamardDifferentiable} implies that  $\Phi\colon L^2(0,T;V) \to W(0,T)$ is completely continuous. As a precursor to characterising the directional derivative $\alpha$, we study the limit of $\{\delta^n\}_{n \in \mathbb{N}}$ in the next lemma. 
Since $\delta^n = \alpha^n + \Phi'(u)(\alpha^{n-1})$, if $\Phi'(u)(\cdot)\colon L^2(0,T;V) \to L^p(0,T;H)$ is bounded, we can find a subsequence of $\{\delta^n\}_{n \in \mathbb{N}}$ such that $\delta^{n_j} \weaklyto \delta$ in $L^p(0,T;H)$ for some $\delta$. In fact under additional assumptions the convergence holds for the full sequence as shown below.


%
\begin{lem}If  \eqref{ass:compContOfDerivOfPhi} holds, then 
$\delta^n \weaklyto \delta$ in $L^2(0,T;V)$.
\end{lem}
\begin{proof}
With the aid of Proposition \ref{prop:boundOnAlphaN}, we can pass to the limit in \eqref{eq:idForAlphaN}, which is $\alpha^{n+1}=\Phi'(u)(\alpha^{n}) - \delta^{n+1}$, to find the weak convergence in $L^2(0,T;V)$ of the whole sequence $\{\delta^n\}$ to some $\delta \in L^2(0,T;V)$.
\end{proof}
To characterise $\delta$ as the solution of a VI in itself, it becomes useful to define the set
\[C_{\mathbb{K}_0}(y) := \{ v \in L^2(0,T;V) : v(t) \geq 0 \text{ q.e. on $\{y(t)=0\}$ for a.e. $t \in [0,T]$} \}.\]
\begin{lem}\label{lem:strongConvergenceOfDirDerv}Under the conditions of the previous lemma, $\delta$ satisfies
\begin{equation*}\label{eq:VIforDelta}
\begin{aligned}
&\delta  \in C_{\mathbb{K}_0}(u-\Phi(u)) \cap [u' + Au - f]^\perp : \int_0^T  \langle z' + A\delta - (d-{L}\Phi'(u)(\alpha)), \delta-z \rangle  \leq \frac 12 \norm{z(0)}{H}^2\\
&\quad  \quad \quad \quad \quad\quad \quad \quad \quad \quad \quad \quad  \quad  \quad \forall z \in \cl_{W}(T_{\mathbb{K}_0}^{\mathrm{rad}}(\Phi(u)-u)\cap [u'+Au-f]^\perp).
\end{aligned}
\end{equation*} 
\end{lem}
\begin{proof}
From \eqref{eq:myVI}, $\delta^n$ satisfies the VI
\begin{equation*}\label{eq:VIforDeltaN}
\begin{aligned}
&\delta^n  \in T_{\mathbb{K}_0,L^2}^{\mathrm{tan}}(\Phi(u)-u) \cap [u' + Au - f]^\perp : \\
&\quad \quad \quad \quad \quad \quad \quad \int_0^T  \langle z' + A\delta^n - (d-{L}\Phi'(u)(\alpha^{n-1})), \delta^n-z \rangle  \leq \frac 12 \norm{z(0)}{H}^2\\
&\quad \quad \quad \quad \quad\quad \quad \quad \quad \quad \quad \quad  \quad   \forall z \in \cl_{W}(T_{\mathbb{K}_0}^{\mathrm{rad}}(\Phi(u)-u)\cap [u'+Au-f]^\perp).
\end{aligned}
\end{equation*}
We can pass to the limit here using the convergence result of the previous lemma and the continuity of $\Phi'(u)(\cdot)\colon L^2(0,T;H) \to W(0,T)$ (noting that $\alpha^n \to \alpha$ in $L^2(0,T;H)$ thanks to $V \ctsCompact H$) and the limiting object $\delta$ satisfies the inequality stated in the lemma. 

We must check that $\delta \in C_{\mathbb{K}_0}(\Phi(u)-u) \cap [u' + Au - f]^\perp$ too. It is clear that the orthogonality condition is satisfied due to the convergence in the previous lemma. By \eqref{eq:tangentConeInclusion}, 
\[\delta^n(t) \geq 0 \text{ q.e. on $\{u(t)=\Phi(u)(t)\}$ a.e. $t$}.\]
Due to Mazur's lemma, there is a convex combination 
\[v_k = \sum_{j=k}^{N(k)} a(k)_j \delta^j\] 
of $\{\delta^n\}_{n \in \mathbb{N}}$ such that $v_k \to \delta$ in $L^2(0,T;V)$. Since this convergence is strong, for a subsequence, $v_{k_m}(t) \to \delta(t)$ in $V$ and hence pointwise q.e. for a.e. $t \in [0,T]$.

By definition, $\delta^n(t) \geq 0$ everywhere on $\{u(t)=\Phi(u)(t)\} \setminus A_n(u)(t)$ where $A_n(u)(t) \subset \{u(t)=\Phi(u)(t)\}$ is a set of capacity zero; this implies that
\begin{equation}\label{eq:helping1}
v_{k_m}(t) \geq 0 \text{ q.e. on $\{u(t)=\Phi(u)(t)\} \setminus \cup_{j=k_m}^{N(k_m)} A_{j}(u)(t)$},
\end{equation}
and using the fact that a countable union of capacity zero sets has capacity zero and the inequality \eqref{eq:helping1}, we can pass to the limit to deduce that $\delta(t) \geq 0$ quasi-everywhere on $\{u(t)=\Phi(u)(t)\}$ for a.e. $t \in [0,T]$. 
\end{proof}
\begin{prop}Under the conditions of the previous lemma, $\alpha$ satisfies the QVI
\begin{align*}
&\alpha-\Phi'(u)(\alpha) \in C_{\mathbb{K}_0}(\Phi(u)-u) \cap [u' + Au - f]^\perp :\\
&\quad \quad \quad \quad \quad \quad \quad \quad \quad  \int_0^T  \langle w' + A\alpha - d, \alpha-w \rangle  \leq \frac 12 \norm{w(0)-\Phi'(u)(\alpha)(0)}{H}^2\\
&\quad \quad \quad \quad \quad \quad \quad \quad \forall w : w -\Phi'(u)(\alpha) \in \cl_{W}(T_{\mathbb{K}_0}^{\mathrm{rad}}(\Phi(u)-u)\cap [u'+Au-f]^\perp).
\end{align*}
\end{prop}
\begin{proof}
From the definition of $\alpha^n$ in terms of $\delta^n$ in \eqref{eq:idForAlphaN}, we obtain
\[\alpha = \Phi'(u)(\alpha)+ \delta.\]
Using this fact in the QVI for $\delta$ given in Lemma \ref{lem:strongConvergenceOfDirDerv} yields
\[\int_0^T  \langle z' + A\alpha + \partial_t \Phi'(u)(\alpha)- d, \delta-z \rangle  \leq \frac 12 \norm{z(0)}{H}^2 \quad \forall z \in \cl_{W}(T_{\mathbb{K}_0}^{\mathrm{rad}}(\Phi(u)-u)\cap [u'+Au-f]^\perp)
\]
which translates into the desired result after setting $w := \Phi'(u)(\alpha) + z$.
\end{proof}

\subsection{Dealing with the higher-order term}\label{sec:HOTs}
We come to the final part which consists in showing that the limit of the higher-order terms $o^n$ is indeed a higher-order term itself. The idea is to be able to commute the two limits
\[\lim_{n \to \infty}\lim_{s \to 0^+} \frac{o^n(s)}{s} \quad\text{and}\quad \lim_{s \to 0^+} \lim_{n \to \infty}\frac{o^n(s)}{s},\]
and this can be done typically under a \textit{uniform} convergence of one of the limits. Such a uniform convergence is assured by the next proposition but first we need a technical result which tells us that the sequence $\{u^n_s\}$ stays close to $u$ for small $s$. Define \[C^*:=K_1^*+\frac{T^{\frac 1p}C_bK_2^*}{\sqrt{C_a}}+ \frac{K_3^*T^{\frac 1p}}{\sqrt{C_a}}.\]
\begin{lem}\label{lem:ball}Under assumptions \eqref{ass:big0}--\eqref{ass:smallnessOfDerivOfPhi}, if $\bar\epsilon \leq \epsilon$ and
\[s \leq \frac{\bar\epsilon\sqrt{C_a}(1-C^*)}{\norm{d}{L^2(0,T;V^*)}T^{\frac 1p}},\]
then $\{u^n_s\}_{n \in \mathbb{N}}\subset B_{\bar\epsilon}^p(u)$.  
\end{lem}
\begin{proof}
The continuous dependence estimate \eqref{eq:S0LipschitzWeirdSpaces2} for the map $\bar S$ along with $L^\infty(0,T) \cts L^q(0,T)$ for any $q \geq 1$ yields the following estimate for $f, g \in L^2(0,T;H)$ and $\psi, \phi \in L^2(0,T;V)$ with $\Phi(\psi), \Phi(\phi) \in W_{\mathrm r}(0,T)$:
\begin{align}
\nonumber &\norm{S(g,\psi) - S(f,\phi)}{L^q(0,T;H)}\\
\nonumber &\quad\leq \norm{\Phi(\psi) -\Phi(\phi)}{L^q(0,T;H)} + \frac{T^{\frac 1q}}{\sqrt{C_a}}\norm{L(\Phi(\phi)-\Phi(\psi))}{L^2(0,T;V^*)} \\
\nonumber &\quad\quad+ \frac{T^{\frac 1q}}{\sqrt{C_a}}\norm{g-f}{L^2(0,T;V^*)}\\
\nonumber &\quad\leq \norm{\Phi(\psi) -\Phi(\phi)}{L^q(0,T;H)} + \frac{C_bT^{\frac 1q}}{\sqrt{C_a}}\norm{\Phi(\phi)-\Phi(\psi)}{L^2(0,T;V)}\\
 &\quad\quad + \frac{T^{\frac 1q}}{\sqrt{C_a}}\norm{\partial \Phi(\psi) -\partial \Phi(\phi)}{L^2(0,T;V^*)}+ \frac{T^{\frac 1q}}{\sqrt{C_a}}\norm{g-f}{L^2(0,T;V^*)}.\label{eq:pre1}
	\end{align}
Since $u \in W(0,T)$, Assumption \ref{ass:newAss} implies that $\Phi(u) \in W_{\mathrm{r}}(0,T)$ and we can apply \eqref{eq:pre1} for $u^1_s = S(f+sd,u)$ and $u=S(f,u)$ to get
\[\norm{u^{1}_s - u}{L^q(0,T;H)} \leq \frac{sT^{\frac 1q}}{\sqrt{C_a}}\norm{d}{L^2(0,T;V^*)},\]
so that $u^1_s \in B^q_{\bar\epsilon}(u)$ if $s$ is sufficiently small. Let us fix $q=p$. Applying the mean value theorem to the first three terms on the right-hand side of \eqref{eq:pre1} and taking in addition $\phi, \psi \in L^2(0,T;H) \cap B_{\bar\epsilon}^p(u)$ so that for any $\lambda \in [0,1]$, $\lambda \phi + (1-\lambda)\psi \in L^2(0,T;H) \cap B_{\bar\epsilon}^p(u)$ as well and we can use assumptions \eqref{ass:big0}--\eqref{ass:bigB} to get
\begin{align*}
&\norm{S(g,\psi) - S(f,\phi)}{L^p(0,T;H)}\\
&\quad\leq \left(K_1^*+\frac{T^{\frac 1p}C_bK_2^*}{\sqrt{C_a}}+ \frac{K_3^*T^{\frac 1p}}{\sqrt{C_a}}\right)\norm{\psi -\phi}{L^p(0,T;H)} + \frac{T^{\frac 1p}}{\sqrt{C_a}}\norm{g-f}{L^2(0,T;V^*)}\\
&\quad= C^*\norm{\psi -\phi}{L^p(0,T;H)} + \frac{T^{\frac 1p}}{\sqrt{C_a}}\norm{g-f}{L^2(0,T;V^*)}
\end{align*}
for all such $\phi, \psi$.
Since $u^1_s, u \in L^2(0,T;H) \cap B_{\bar\epsilon}^p(u)$ with $L\Phi(u_s^n) \in L^2(0,T;H)$ from Proposition \ref{prop:expansionFormulan}, the above formula is applicable for $u^2_s = S(f+sd,u^1_s)$ and $u$, and we get for sufficiently small $s$ that
\begin{align*}
\norm{u^{2}_s - u}{L^p(0,T;H)} &\leq C^*\norm{u^{1}_s -u}{L^p(0,T;H)} + \frac{sT^{\frac 1p}}{\sqrt{C_a}}\norm{d}{L^2(0,T;V^*)}\\
&\leq \left(C^*+ 1\right)\frac{sT^{\frac 1p}}{\sqrt{C_a}}\norm{d}{L^2(0,T;V^*)},
\end{align*}
i.e., $u_s^2 \in B_{\bar\epsilon}(u)$. Arguing by induction, the general case satisfies the estimate
\[\norm{u^{n}_s - u}{L^p(0,T;H)} \leq C^*\norm{u^{n-1}_s -u}{L^p(0,T;H)} + \frac{sT^{\frac 1p}}{\sqrt{C_a}}\norm{d}{L^2(0,T;V^*)}\]
and by solving the above recurrence inequality, we find
\begin{align*}
\norm{u^{n}_s - u}{L^p(0,T;H)} &\leq ((C^*)^{n-1} + (C^*)^{n-2} + ... + C^* + 1)\frac{sT^{\frac 1p}}{\sqrt{C_a}}\norm{d}{L^2(0,T;V^*)}\\
&\leq \frac{1}{1-C^*}\frac{sT^{\frac 1p}}{\sqrt{C_a}}\norm{d}{L^2(0,T;V^*)}\\
&\leq \bar\epsilon
\end{align*}
where we used the formula for geometric series (recall that \eqref{ass:smallnessOfDerivOfPhi} says precisely that $C^* < 1$).
\end{proof}

\begin{prop}
Suppose Assumptions \ref{ass:forBoundednessOfAlphas} and \ref{ass:requiredLocalAssumptions} hold. 
Then $s^{-1}o^n(s) \to 0$ in $L^p(0,T;H)$ as $s \to 0^+$ uniformly in $n$ and thus the limit $\lim_{n \to \infty} o^n$ is also a higher-order term.
\end{prop}
\begin{proof}
The proof is in four steps. 

\medskip

\noindent \textsc{Step 1.} 
We first collect some estimates. From the expansions
\begin{align*}
\nonumber \Phi(u+s\alpha^{n-1}) &= \Phi(u) + s\Phi'(u)(\alpha^{n-1}) + l(s,\alpha^{n-1}),\\
\Phi(u+s\alpha^{n-1} + o^{n-1}(s)) &= \Phi(u) + s\Phi'(u)(\alpha^{n-1}) + \hat l(s,\alpha^{n-1}, o^{n-1}(s)),
\end{align*}
we get, from subtracting one from the other and using the mean value theorem,
\begin{align}
&\nonumber \norm{\hat l(s,\alpha^{n-1}, o^{n-1}(s)) }{L^2(0,T;V)}\\ 
\nonumber &\quad \leq \sup_{\lambda \in (0,1)}\norm{\Phi'(u+s\alpha^{n-1}+\lambda o^{n-1}(s))(o^{n-1}(s))}{L^2(0,T;V)} + \norm{l(s,\alpha^{n-1})}{L^2(0,T;V)}\\
&= \sup_{\lambda \in (0,1)}\norm{\Phi'(\lambda u^{n-1}_s + (1-\lambda)(u+s\alpha^{n-1}))(o^{n-1}(s))}{L^2(0,T;V)} + \norm{l(s,\alpha^{n-1})}{L^2(0,T;V)}.\label{eq:aaa}
\end{align}
Now take
\[ s \leq \tau_0 := \min\left(\frac{\epsilon\sqrt{C_a}(1-C^*)}{\norm{d}{L^2(0,T;V^*)}T^{\frac 1p}}, \frac{\epsilon}{\sup_n \norm{\alpha^n}{L^2(0,T;H)}}\right)\]
(the supremum over $n$ is sensible due to Proposition \ref{prop:boundOnAlphaN}). Then $u_s^{n-1} \in B_{\epsilon}^p(u)$ for all $n$ by Lemma \ref{lem:ball}, so that $\lambda u_s^{n-1} + (1-\lambda)u \in B_{\epsilon}^p(u)$ for $\lambda \in (0,1)$, and we also have $s\alpha^{n-1} \in B_{\epsilon}^2(0)$ and hence $\lambda u^{n-1}_s + (1-\lambda)(u+s\alpha^{n-1}) \in B_{\epsilon}^p(u) + B_{\epsilon }^2(0)$. Thus assumption \eqref{ass:bigA} is applicable on the right-hand side of \eqref{eq:aaa} and we obtain
\begin{align}
\norm{\hat l(s,\alpha^{n-1}, o^{n-1}(s)) }{L^2(0,T;V)}\leq K_2^*\norm{o^{n-1}(s)}{L^p(0,T;H)} + \norm{l(s,\alpha^{n-1})}{L^2(0,T;V)}.\label{eq:aa}
\end{align}
Secondly, arguing as before, by differentiating in time the expansions for $\Phi(u+s\alpha^{n-1})$ and $\Phi(u+s\alpha^{n-1}+o^{n-1}(s))$, taking the difference 
and using the mean value theorem (bearing in mind that $(\partial_t \Phi)'(z)(b) = \partial_t \Phi'(z)(b)$ by linearity) and applying assumption \eqref{ass:bigB},
\begin{align}
\norm{\partial_t \hat l(s,\alpha^{n-1}, o^{n-1}(s))}{L^2(0,T;V^*)} 
&\leq K_3^*\norm{o^{n-1}(s)}{L^p(0,T;H)} + \norm{\partial_t l(s,\alpha^{n-1})}{L^2(0,T;V^*)} \label{eq:bb}.
\end{align}

\medskip

\noindent \textsc{Step 2.} By definition of $o^n$ in \eqref{eq:identityforOnAndRn},
\begin{align*}
o^n(s) 
&=\hat l(s,\alpha^{n-1},o^{n-1}(s)) - \hat o(s,{L}\Phi'(u)(\alpha^{n-1})-d, {L}\hat l(s,\alpha^{n-1}, o^{n-1}(s))),
\end{align*}
and using the boundedness estimates \eqref{eq:hatLBound}, \eqref{eq:hatOBoundIntoL2VStar} and assumption \eqref{ass:big0},
\begin{align*}
&\norm{o^n(s)}{L^p(0,T;H)}\\
&\leq \sup_{\lambda \in (0,1)}\norm{\Phi'(u+s\alpha^{n-1}+\lambda o^{n-1}(s))(o^{n-1}(s))}{L^p(0,T;H)} + \norm{l(s,\alpha^{n-1})}{L^p(0,T;H)}\\
&\quad+ \frac{T^{\frac 1p}}{\sqrt{C_a}}\norm{{L}\hat l(s,\alpha^{n-1},o^{n-1}(s))}{L^2(0,T;V^*)}+ \norm{o(s,{L}\Phi'(u)(\alpha^{n-1})-d)}{L^p(0,T;H)}\\
&\leq K_1^*\norm{o^{n-1}(s)}{L^p(0,T;H)} + \norm{l(s,\alpha^{n-1})}{L^p(0,T;H)} + \frac{T^{\frac 1p}}{\sqrt{C_a}}\norm{{L}\hat l(s,\alpha^{n-1},o^{n-1}(s))}{L^2(0,T;V^*)}\\
&\quad + \norm{o(s,{L}\Phi'(u)(\alpha^{n-1})-d)}{L^p(0,T;H)}.
\end{align*}
We estimate the third term above using \eqref{eq:aa} and \eqref{eq:bb} of Step 1:
\begin{align*}
&\norm{{L}\hat l(s,\alpha^{n-1},o^{n-1}(s))}{L^2(0,T;V^*)}\\ 
&\leq C_b\norm{\hat l(s,\alpha^{n-1},o^{n-1}(s))}{L^2(0,T;V)} + \norm{\partial_t\hat l(s,\alpha^{n-1},o^{n-1}(s))}{L^2(0,T;V^*)}\\
&\leq K_2^*C_b\norm{o^{n-1}(s)}{L^p(0,T;H)} + C_b\norm{l(s,\alpha^{n-1})}{L^2(0,T;V)} + K_3^*\norm{o^{n-1}(s)}{L^p(0,T;H)}\\
&\quad + \norm{\partial_t l(s,\alpha^{n-1})}{L^2(0,T;V^*)}.
\end{align*}
Inserting this back above, we find
\begin{align*}
&\norm{o^n(s)}{L^p(0,T;H)}\\
&\leq (K_1^*+\frac{T^{\frac 1p}(K_2^*C_b+K_3^*)}{\sqrt{C_a}})\norm{o^{n-1}(s)}{L^p(0,T;H)} + \norm{l(s,\alpha^{n-1})}{L^p(0,T;H)} \\
&\quad + \frac{T^{\frac 1p}C_b}{\sqrt{C_a}}\norm{l(s,\alpha^{n-1})}{L^2(0,T;V)}+ \frac{T^{\frac 1p}}{\sqrt{C_a}}\norm{\partial_t l(s,\alpha^{n-1})}{L^2(0,T;V^*)}\\
&\quad + \norm{o(s,{L}\Phi'(u)(\alpha^{n-1})-d)}{L^p(0,T;H)}\\
&<  C\norm{o^{n-1}(s)}{L^p(0,T;H)} + \norm{l(s,\alpha^{n-1})}{L^p(0,T;H)} + \frac{T^{\frac 1p}C_b}{\sqrt{C_a}}\norm{l(s,\alpha^{n-1})}{L^2(0,T;V)} \\
&\quad + \frac{T^{\frac 1p}}{\sqrt{C_a}}\norm{\partial_t l(s,\alpha^{n-1})}{L^2(0,T;V^*)} + \norm{o(s,{L}\Phi'(u)(\alpha^{n-1})-d)}{L^p(0,T;H)}
\end{align*}
for some $C<1$ by \eqref{ass:smallnessOfDerivOfPhi}. The above can be recast as
\[a^n(s) \leq Ca_{n-1}(s) + b_{n-1}(s),\]
where 
\begin{align*}
a_n(s) &:= \norm{o^n(s)}{L^p(0,T;H)},\\
b_{n-1}(s)&:=\norm{l(s,\alpha^{n-1})}{L^p(0,T;H)} + \frac{T^{\frac 1p}C_b}{\sqrt{C_a}}\norm{l(s,\alpha^{n-1})}{L^2(0,T;V)} \\
&\quad + \frac{T^{\frac 1p}}{\sqrt{C_a}}\norm{\partial_t l(s,\alpha^{n-1})}{L^2(0,T;V^*)}+ \norm{o(s,{L}\Phi'(u)(\alpha^{n-1})-d)}{L^p(0,T;H)}.
\end{align*}
The recurrence inequality can be solved for $a^n$ in terms of $a_1$ and the $b_i$:
\begin{align}
a^n 
&\leq C^{n-1}a_{1} + C^{n-2}b_1 + C^{n-3}b_2 + ... + Cb_{n-2} + b_{n-1}\label{eq:prelim2}.
\end{align}

\medskip

\noindent \textsc{Step 3.} Let us see why 
\[\frac{b_{n-1}(s)}{s}  \to 0 \quad\text{uniformly in $n$}.\]
By the weak convergence of $\{\alpha^n\}$ in $L^2(0,T;V)$ (and hence strong convergence in $L^2(0,T;H)$), $\{\alpha^{n}\}$ is a compact set in $L^2(0,T;H)$. Since $\Phi\colon L^2(0,T;H) \to W(0,T)$ is Hadamard differentiable, it is compactly differentiable, meaning that $s^{-1}l(s,\alpha^n) \to 0$ in $W(0,T)$ uniformly, thus the first three terms in the definition of $b_n$ are taken care of. For the final term, we note that $\{{L}\Phi'(u)(\alpha^{n})\}_{n \in \mathbb{N}}$ is also compact in $L^2(0,T;H)$ by \eqref{ass:strongCCofDerivative}. Therefore, for any $\epsilon  > 0$, there exists a $\tau_1 > 0$ independent of $j$ such that 
\begin{equation}\label{eq:cc}
s \leq \tau_1 \implies \frac{b_{j}(s)}{s} \leq \frac{(1-C)\epsilon}{2}\qquad \forall j.
\end{equation}

\medskip

\noindent \textsc{Step 4.} As $o_1(s) = r(s,0,0)=o(s,d)$ is a higher-order term, we know that there is a $\tau_2 > 0$ such that
\begin{equation}\label{eq:dd}
s \leq \tau_2 \implies \frac{\norm{o(s,d)}{L^p(0,T;H)}}{s} \leq \frac{\epsilon}{2}.
\end{equation}
Now recalling \eqref{eq:prelim2}, for $s \leq \min(\tau_0, \tau_1,\tau_2)$,
\begin{align*}
\frac{\norm{o^n(s)}{L^p(0,T;H)}}{s}
&\leq C^{n-1}\frac{\norm{o(s,d)}{L^p(0,T;H)}}{s} +
C^{n-2}\frac{b_1(s)}{s} + C^{n-3}\frac{b_2(s)}{s} + ...  + \frac{b_{n-1}(s)}{s}\\
&\leq \frac{\norm{o(s,d)}{L^p(0,T;H)}}{s} + C^{n-2}\frac{b_1(s)}{s} + C^{n-3}\frac{b_2(s)}{s}+ ...  + \frac{b_{n-1}(s)}{s}\\
&\leq \frac{\epsilon}{2} + \frac{\epsilon(1-C)}{2}\left(C^{n-2} + C^{n-3} + ...  + C + 1\right)\tag{for any $\epsilon > 0$ by \eqref{eq:cc} and \eqref{eq:dd}}\\
&= \frac{\epsilon}{2} + \frac{\epsilon(1-C)(1-C^{n-1})}{2(1-C)}\\
&\leq \epsilon.
\end{align*}
This shows that $o^n(s)\slash s$ tends to zero uniformly in $n$.
\end{proof}
We are finally in position to state the main theorem which condenses the various intermediary results we obtained above.
\begin{theorem}\label{thm:directionalDifferentiability}
Let Assumptions \ref{ass:onSpaces} and \ref{ass:PhiHadamardDifferentiable} hold and take $u \in \mathbf{P}_{u_0}(f)$ satisfying Assumption \ref{ass:onU}. Let also Assumptions  \ref{ass:newAss},  
\ref{ass:forBoundednessOfAlphas} and \ref{ass:requiredLocalAssumptions} hold.

Then there exists a $u_s \in \mathbf{P}(f+sd)$ such that, under the first set of assumptions of Assumption \ref{ass:newAss}, $u_s \in W_s(0,T)$ is a (strong) solution whereas under the second set of assumptions of Assumption \ref{ass:newAss}, $u_s \in L^2(0,T;V) \cap L^\infty(0,T;H)$ is a very weak solution, and $u_s$ satisfies
\[u_s = u + s\alpha + o_s, \quad \text{as $s \downarrow 0$},\]
where 
\[\frac{o_s}{s} \to 0 \text{ in $L^p(0,T;H)$ as $s \to 0^+$},\]
with $p \in [1,\infty)$ is as in Assumption \ref{ass:newAss}, and $\alpha$ satisfies the parabolic QVI
\begin{align*}
&\alpha-\Phi'(u)(\alpha) \in C_{\mathbb{K}_0}(\Phi(u)-u) \cap [u' + Au - f]^\perp : \\
&\qquad\qquad\qquad\int_0^T  \langle w' + A\alpha - d, \alpha-w \rangle  \leq \frac 12 \norm{w(0)-\Phi'(u)(\alpha)(0)}{H}^2\\
&\qquad\qquad\qquad\qquad\qquad\forall w : w -\Phi'(u)(\alpha) \in \cl_{W}(T_{\mathbb{K}_0}^{\mathrm{rad}}(\Phi(u)-u)\cap [u'+Au-f]^\perp).
\end{align*}
\end{theorem}
Let us relate this result to notions of differentiability commonly used in set-valued and multi-valued analysis. Recall that the map $\mathbf{P}\colon F \rightrightarrows U$ has a \textit{contingent derivative} $\beta$ at $(f,u)$ (with $u \in \mathbf{P}(f)$) in the direction $d$, written $\beta \in D\mathbf{P}(f,u)(d)$,  if there exist sequences $\beta_n \to \beta$, $d_n \to d$ and $s_n \to 0$ such that
\[u +s_n\beta_n \in \mathbf{P}(f+s_nd_n).\]
We claim that $\mathbf{P}$ does indeed possess contingent derivatives and our main results furnishes us with one such contingent derivative.
\begin{prop}\label{prop:contingentDerivative}The map $\mathbf{P}\colon F \rightrightarrows U$ has a contingent derivative $\alpha \in D\mathbf{Q}(f,u)(d)$ given by the previous theorem with either of the following choices:
\begin{align*}
&F=L^2_{inc}(0,T;H_+) \quad \text{ and }\quad U=W_s(0,T),
\intertext{(where $L^2_{inc}(0,T;X)$ means increasing-in-time elements of $L^2(0,T;X)$) or}
&F=L^2(0,T;H) \quad \text{ and } \quad U=L^2(0,T;V)\cap L^\infty(0,T;H),
\end{align*}
\end{prop}
\begin{proof}
Indeed, given $u \in \mathbf{P}(f)$ and sequences $s_n \to 0$ and $d_n \equiv d$, we can, by Theorem \ref{thm:directionalDifferentiability}, find $u_n^s \in \mathbf{P}(f+s_nd)$ such that
\[u_n^s = u + s_n\alpha_n + o(s_n; d)\text{ and hence }u+s_n\beta_n = u_n^s \in \mathbf{P}(f+s_nd_n),\]
where the sequence 
\[\beta_n := \alpha_n + \frac{o(s_n;d)}{s_n} = \frac{u_n^s - u}{s_n}\] is such that $\beta_n \to \alpha$ as $n \to \infty$ because the term $o(s_n;d)$ is higher order.
\end{proof}
\section{Other approaches to differentiability}\label{sec:otherApproaches}
We now discuss some possible alternative approaches to deriving Theorem \ref{thm:directionalDifferentiability} (or a variant of this theorem). The idea here is to bootstrap by using the already-achieved results on elliptic QVIs in \cite{AHR}, however, we shall see that this is not at all straightforward.
\subsection{Time discretisation and elliptic QVI theory}
One idea is to apply the elliptic QVI theory of \cite{AHR} to the time-discrete problems and then pass to the limit there. With the definition
\[\Psi_{n,N}(z):=\Phi(z)(t_{n}^N),\]recall the notations defined in \S \ref{sec:discretisation}; in particular $\mathbf{Q}_{t_{n}^N}$ is the solution mapping defined as $z \in \mathbf{Q}_{t_n^N}(g)$ if and only if
\begin{equation*}
z \in V, z \leq \Psi_{n-1,N}(z) :  \langle z + hAz -g, z-v \rangle \leq 0 \quad \forall v  \in V : v \leq \Psi_{n-1,N}(z).
\end{equation*}
With $u_0^N := u_0 \in V$ for a given initial data, set \[u_n^N := \mathbf{Q}_{t_{n-1}^N}(hf_n^N + u_{n-1}^N).\]
Under the assumptions on the source $f$ and the direction $d$ in \S \ref{sec:discretisation}, we know by \cite[Theorem 1]{AHR} that there exists $u_{s,1}^N \in \mathbf{Q}_{t_{0}^N}(hf_1^N + u_0^N + shd_1^N)$ and $\gamma_1^N := \mathbf{Q}'_{t_{0}^N}(hf_1^N + u_0)(hd_1^N)$ such that
\begin{align*}
u_{s,1}^N &= u_1^N + s\gamma_1^N + o_1^N(s),
\end{align*}
where $o_1^N(s,hd_1^N;hf_1^N+u_0^N)$ is a higher-order term. Since $u_{s,0}^N = u_0^N$,
\[\langle u_{s,1}^N -u_0^N + hAu_{s,1}^N - (hf_1^N + shd_1^N), u_{s,1}^N - v \rangle \leq 0 \quad \forall v \in V : v \leq \Phi(u_{s,1}^N)(t_0^N).\]
In a similar way, since $u_2^N = \mathbf{Q}_{t_{1}^N}(hf_2^N + u_1^N)$, we know that there exists a function $u_{s,2}^N \in \mathbf{Q}_{t_{1}^N}(hf_2^N + u_1^N + s(hd_2^N + \alpha_1^N) + o_1^N(s))$ and $\gamma_2^N = \mathbf{Q}_{t_{1}^N}'(hf_2^N+u_1^N)(hd_2^N + \gamma_1^N)$ such that
\begin{align*}
u_{s,2}^N &= u_2^N + s\gamma_2^N + o_2^N(s),
\end{align*}
where $o_2^N(s, hd_2^N + \alpha_1^N, o_1^N(s);hf_2^N + u_1^N)$ is a higher-order term and since $u_1^N + s\alpha_1^N + o_1^N(s) = u_{s,1}^N$,
\[\langle u_{s,2}^N - u_{s,1}^N + hAu_{s,2}^N - (hf_2^N + shd_2^N), u_{s,2}^N - v \rangle \leq 0 \quad \forall v \in V : v \leq \Phi(u_{s,2}^N)(t_1^N).\]
Along these lines, we obtain the following lemma.
\begin{lem}
Given $u_n^N$ defined as above, there exists $u_{s,n}^N \in \mathbf{Q}_{t_{n-1}^N}(hf_n^N + u_{n-1}^N + shd_n^N)$ and $\gamma_n^N$ 
such that
\begin{equation}\label{eq:idForQNn}
u_{s,n}^N = u_n^N + s\gamma_n^N + o_n^N(s),
\end{equation}
where $\gamma_n^N$ satisfies
\begin{equation}\label{eq:QVIforGammas}
\begin{aligned}
&\langle \gamma_n^N - \gamma_{n-1}^N + hA \gamma_n^N - hd_n^N , \gamma_n^N -v \rangle \leq 0 \quad \forall v \in \mathbb{K}^{n}(\gamma_n^N),\\
&\mathbb{K}^{n}(w) := \{ \varphi \in V : \varphi \leq \Psi_{n,N}'(u_n^N)(w) \text{ q.e. on } \mathcal{A}(u_n^N)\\ & \qquad\qquad \qquad\qquad \text{ \& } \langle u_n^N - u_{n-1}^N + hAu_n^N - hf_n^N , \varphi - \Psi_{n,N}'(u_n^N)(w) \rangle = 0 \},
\end{aligned}
\end{equation}
and
\[o_n^N(s) = o_n^N(s, hd_n^N + \gamma_{n-1}^N, o_{n-1}^N(s);hf_n^N+u_{n-1}^N)\]
is a higher-order term. Here, $\mathcal{A}(u_n^N) = \{u_n^N = \Psi_{n-1,N}(u_n^N)\}.$
\end{lem}
\begin{proof}   
We prove this by induction. The base case has been shown above. Suppose it holds for the $n$th case. Then given $u_{n+1}^N = \mathbf{Q}_{t_{n-1}^N}(hf_{n+1}^N+u_n^N)$, there exists 
\[u_{s,n+1}^N \in 
\mathbf{Q}_{t_{n-1}^N}(hf_{n+1}^N+u_n^N + s(hd_{n+1}^N + \gamma_n^N) + o_n^N(s)) = \mathbf{Q}(hf_{n+1}^N+u_{s,n}^N + shd_{n+1}^N)\] (thanks to the formula relating $u_{s,n}^N$ and $u_n^N$) such that 
\begin{align*}
u_{s,n+1}^N &= \mathbf{Q}_{t_{n-1}^N}(hf_{n+1}^N+u_n^N) + s\mathbf{Q}_{t_{n-1}^N}'(hf_{n+1}^N+u_n^N)(hd_{n+1}^N + \gamma_n^N) + o_{n+1}^N(s)\\
&= u_{n+1}^N + s\mathbf{Q}_{t_{n-1}^N}'(hf_{n+1}^N+u_n^N)(hd_{n+1}^N + \gamma_n^N) + o_{n+1}^N(s)\\
&= u_{n+1}^N + s\gamma_{n+1}^N + o_{n+1}^N(s),
\end{align*}
which ends the proof.
\end{proof}
By definition, $u_{s,n}^N$ solves the QVI
\begin{equation}\label{eq:qviForQnN}
u_{s,n}^N \leq \Phi(u_{s,n}^N)(t_{n-1}^N) : \langle u_{s,n}^N -u_{s,n-1}^N+ hAu_{s,n}^N - (hf_n^N + shd_n^N), u_{s,n}^N - v \rangle \leq 0 
\end{equation}
for all $v \in V$ s.t. $v \leq \Phi(u_{s,n}^N)(t_{n-1}^N)$. It is then clear from \eqref{eq:qviForQnN} (since the structure of the time-discretised QVI is the same as those considered in \S \ref{sec:discretisation} and we can apply the same argumentation as there) that the interpolants $u_s^N$, $\hat u_s^N$ constructed from $\{u_{s,n}^N\}_{n \in \mathbb{N}}$ and which satisfy
\begin{align*}
\int_0^T \langle \partial_t \hat u_s^N(t) + Au_s^N(t) - f^N(t) - sd^N(t), u_s^N(t) - v^N(t) \rangle \leq 0,
\end{align*}
are bounded in the appropriate spaces and hence there exists a $u_s$ such that $u_s^N \to u_s$ and
\begin{align*}
\int_0^T \langle \partial_t  u_s(t) + Au_s(t) - f(t) - sd(t), u_s(t) - v(t) \rangle &\leq 0\\
&\forall v \in L^2(0,T;V) : v \leq \Phi(u_s),
\end{align*}
i.e., $u_s \in \mathbf{P}(f+sd)$. Clearly, similar claims are true for $u^N$, $\hat u^N$ constructed as interpolants from $\{u^N_n\}_{n \in \mathbb{N}}$. We now need to obtain uniform bounds on the $\gamma_n^N$ and the higher-order term.

\begin{lem}Suppose that $v=0$ is a valid test function in \eqref{eq:QVIforGammas} (which is the case in the VI setting). Then 
\begin{align*}
\norm{\gamma_n^N}{H} + h\sum_{i=1}^n\norm{\gamma_i^N}{V}^2 &\leq C,
\end{align*}
and hence
\begin{align*}
\norm{\gamma^N}{L^\infty(0,T;H)} +\norm{\gamma^N}{L^2(0,T;V)} &\leq C.
\end{align*}
\end{lem}
\begin{proof}
Since $0$ is feasible, the proof for the boundedness of the first two estimates follows that of the first two estimates of Lemma \ref{lem:boundsOnunN}. The proof of Corollary \ref{cor:interpolantBounds} shows that these bounds imply the boundedness in the Bochner spaces above.
\end{proof}
Multiplying \eqref{eq:idForQNn} by $\chi_{T_{n-1}^N}$ and summing up over $N$, we find 
\[u_s^N = u^N + s\gamma^N + o^N(s).\]
Due to the bounds on the left-hand side and the first two terms on the right-hand side, we can pass to the limit in $o^N$ too and we find
\[u_s = u + s\gamma^* + o^*(s),\]
for some $\gamma^* \in L^2(0,T;V) \cap L^\infty(0,T;H)$, and the equality also holds in this space. It remains to be seen that $o^*$ is a higher-order term and to characterise the term $\gamma^*$.

Here we come to a roadblock since we do not know how the terms $o^N$ depend on the moving base point and thus we cannot say anything about the uniform convergence of the $o^N$ and are unable to identify the limit of the higher-order terms as higher order; recall that we warned the reader of this issue in Remark \ref{rem:basePoints}. Alternatively, if we were able to identify $\gamma^*$ as $\alpha$ from the previous sections, this would suffice to show the higher-order behaviour. But it seems difficult to handle the constraint set satisfied by $\gamma^N_n$ and obtain a type of Mosco convergence for recovery sequences to approximate the limiting test function space.

\subsection{Elliptic regularisation of the parabolic (Q)VI}
Another possible approach is through regularising the parabolic QVI as an elliptic QVI, applying the elliptic theory of \cite{AHR} to the regularised problem and then passing to the limit in the regularisation parameter. This elliptic regularisation of parabolic problems can be seen in the work of Lions \cite[p.~407]{MR0600331}.

The idea is to include in the parabolic inequality a term involving the second time derivative of the solution, i.e., $-\epsilon u''$ with $\epsilon > 0$; more precisely we wish to consider the QVI
\begin{equation}\label{eq:regularisedqvi}
\begin{aligned}
u(t) \leq \Phi(u)(t) : \int_0^T \langle -\epsilon u''(t) + u'(t) + Au(t) - f(t), u(t) - v(t) \rangle &\leq 0\\
&\!\!\!\!\!\!\!\!\!\!\!\!\!\!\!\!\! \!\!\!\!\!\!\!\!\!\!\!\!\!\!\!\!\!\!\!\!\!\!\!\!\!\!\!\!\!\!\!\!\!\!\forall v \in L^2(0,T;V) : v(t) \leq \Phi(u)(t),\\
u(0) &= 0,\\
u'(T) &= 0.
\end{aligned}
\end{equation}
The `final time' condition is necessary to have a well defined problem. Define $\mathcal{V} := \{v \in W_s(0,T) : v(0) = 0\}.$ Integrating by parts in \eqref{eq:regularisedqvi} and using the initial and final conditions on $u$ and the test function space, we obtain the weak form: find $u \in \mathcal{V}$ with $
u(t) \leq \Phi(u)(t)$ such that
\begin{equation}\label{eq:regularisedqvi2}
\begin{aligned} : \int_0^T \epsilon (u'(t), u'(t)-v'(t))_H + \langle u'(t) + Au(t)-f(t), u(t) - v(t) \rangle &\leq 0\\ &\!\!\!\!\!\!\!\!\!\!\!\!\!\!\!\!\!\!\!\!\!\!\!\!\!\!\!\!\!\!\!\!\!\!\!\!\!\!\!\!\!\!\!\!\!\!\! \!\!\!\!\!\!\!\!\!\!\!\!\!\!\!\!\!\!\!\ \forall v \in \mathcal{V}\quad \text{s.t.}\quad v(t) \leq \Phi(u)(t),\\
u(0) &= 0,\\
u'(T) &= 0.
\end{aligned}
\end{equation}
Thinking of the time component as simply another space dimension, the resulting elliptic operator $\hat A\colon \mathcal{V} \to \mathcal{V}^*$ is
\begin{align*}
\langle \hat Au, w \rangle &:= \int_0^T \epsilon(u',w')_H + (u', w)_H + \langle Au, w \rangle_{V^*,V},
\end{align*}
which is clearly T-monotone, bounded and coercive in the Sobolev--Bochner space $\mathcal{V}$; for the latter, observe that
\[\int_0^T (u',u)_H = \frac 12 \norm{u(T)}{H}^2 \geq 0.\]
Clearly the solution of \eqref{eq:regularisedqvi2} depends on $\epsilon$ so let us write it as $u^\epsilon$. Working formally, we may apply \cite[Theorem 1.6]{AHR} and we find existence of a $u_s^\epsilon$ solving \eqref{eq:regularisedqvi2} with source term $f+sd$ and a $\alpha^\epsilon$ such that
\[u_s^\epsilon = u^\epsilon + s\alpha^\epsilon + o^\epsilon(s),\]
where $o^\epsilon$ is a higher-order term. As for $\alpha^\epsilon$, it satisfies $\alpha^{\epsilon} \in \mathbb{K}^{u^\epsilon}(\alpha^\epsilon)$ and
\begin{equation*}
\begin{aligned}
 &\int_0^T (\epsilon \partial_t \alpha^\epsilon, \partial_t\alpha^\epsilon-v')_H + (\partial_t \alpha^\epsilon, \alpha^\epsilon-v)_H + \langle A\alpha^\epsilon, \alpha^\epsilon-v \rangle - \langle d, \alpha^\epsilon -v \rangle \leq 0\\
&\qquad\qquad\qquad\qquad\qquad\qquad\qquad\qquad\qquad\qquad\qquad\qquad\qquad \forall v \in \mathbb{K}^{u^\epsilon}(\alpha^\epsilon),\\
&\mathbb{K}^{u^\epsilon}(w) := \{\varphi \in \mathcal{V} : \varphi \leq \Phi'(u^\epsilon)(w) \text{ q.e. on $\mathcal{A}(u^\epsilon)$ and $\langle \hat Au^\epsilon-f, \varphi - \Phi'(u^\epsilon)(w) \rangle = 0$}\}.
\end{aligned}
\end{equation*}
Testing \eqref{eq:regularisedqvi2} with $v=0$, we obtain the bound
\[\epsilon\norm{\partial_t u^\epsilon}{L^2(0,T;H)}^2 + \frac 12 \norm{u}{L^\infty(0,T;H)}^2 + C_a\norm{u}{L^2(0,T;V)}^2 \leq C,\]
so $u^\epsilon$ is bounded in at least $L^2(0,T;V) \cap L^\infty(0,T;H)$ uniformly in $\epsilon$. In a similar way, $u^\epsilon_s$ is also bounded in this space uniformly in $\epsilon$ and $s$. Let us also suppose that we have a uniform bound on $\alpha^\epsilon$. Then it remains to be seen that the limit of the $o^\epsilon$ is a higher-order term, and this is where we once again run into problems. A monotonicity in $\epsilon$ result on the solutions of \eqref{eq:regularisedqvi2} would be useful.

It is worth remarking that this technical issues also arises in the VI case and in the simpler unconstrained (PDE) case.

\section{Example}\label{sec:example}
We study here an example where the obstacle mapping is given by the inverse of a parabolic differential operator. This example is motivated by applications in thermoforming (which is a process that manufactures shapes such as car panels from a given mould shape; see \cite{AHR} and references therein for more information): in \cite{AHR}, we studied an elliptic nonlinear PDE-QVI model that describes such a thermoforming process but in reality, the full model is a highly complicated evolutionary system of equations and QVIs involving obstacle maps that are inverses of differential operators. As a first initial step on the road to studying the full problem, we consider a simplification and study the following example.

 Let $H=L^2(0,\omega)$ and $V=H^1_0(0,\omega)$ and consider the following nonlinear heat equation: 
\begin{equation}\label{eq:exampleEqn}
\begin{aligned}
w' + Bw &= g(\psi) &&\text{on $(0,T) \times (0,\omega)$},\\
w(\cdot,0) &= w(\cdot,\omega) = 0&&\text{on $(0,T)$},\\
w(0,\cdot) &= w_0 &&\text{on $(0,\omega)$},
\end{aligned}
\end{equation}
where 
\begin{enumerate}[label={(\arabic*)}]
\item $g \colon \mathbb{R} \to \mathbb{R}$ is $C^1$, non-negative and increasing,
\item $g, g' \in L^\infty(\mathbb{R})$,
\item $g, g'\colon L^2(0,T;H) \to L^2(0,T;H)$ are continuous,
\item $g \colon L^2(0,T;H) \to L^2(0,T;H)$ is directionally differentiable,
\item $B\colon V \to V^*$ is a linear, bounded, coercive and T-monotone operator giving rise to a differentiable bilinear form,
\item  $w_0 \in V$ and $\psi \in L^2(0,T;V)$ are given data. 
\end{enumerate}
We denote by $C^B_b$ and $C^B_a$ the constants of boundedness and coercivity for the operator $B$ while $A\colon V \to V^*$ denotes an operator which is assumed to satisfy all assumptions given in the introduction of this paper and it should also give rise to a bilinear form which is smooth.  Letting $W=H^2(\Omega)$, we further assume that 
\begin{enumerate}[label={(\arabic*)}]  \setcounter{enumi}{6}
\item $A\colon L^2(0,T;W) \to L^2(0,T;H)$,
\item the following norms are equivalent:
\[\norm{u}{L^2(0,T;H)} + \norm{Au}{L^2(0,T;H)},\quad  \norm{u}{L^2(0,T;H)} + \norm{Bu}{L^2(0,T;H)},\quad \norm{u}{L^2(0,T;W)}.\]
\end{enumerate}
This equivalence of norms assumption is related to (boundary) regularity results which follow given enough smoothness of the boundary and depending on the specific elliptic operators, see \cite[Theorem 4, \S 6.3]{Evans}. The equivalence of norms implies
\begin{equation}
w \in L^2(0,T;V) : Bw \in L^2(0,T;H) \implies w \in L^2(0,T;W).\label{eg1}
\end{equation}

Define $\Phi(\psi):=w$ as the solution of \eqref{eq:exampleEqn}. By the standard theory of parabolic PDEs, $\Phi$ maps $L^2(0,T;H)$ into $W_s(0,T)$. This regularity implies (using the equation itself) that for $\psi \in L^2(0,T;H)$, $B\Phi(\psi) \in L^2(0,T;H)$ and hence by \eqref{eg1} and the assumption on the range of the operator $A$, $\Phi(\psi) \in L^2(0,T;W)$ and $A\Phi(\psi) \in L^2(0,T;H)$. It follows then that ${L}\Phi(\psi) \in L^2(0,T;H)$ too (recall that $L=\partial_t + A$). Since $\Phi(\psi) \in L^2(0,T;W) \cap H^1(0,T;H)$, the theorem of Aubin--Lions \cite[Theorem II.5.16]{Boyer2012} gives the continuity in time $\Phi(\psi) \in C^0([0,T];V)$. 

We proceed now with checking the assumptions that will eventually enable us to use Theorem \ref{thm:directionalDifferentiability}.
\begin{lem}
Assumption  \ref{ass:PhiHadamardDifferentiable} on the Hadamard differentiability of $\Phi$ holds and the directional derivative satisfies
\[\Phi'(\psi)(h) =  \Phi_0(g'(\psi)(h))\] 
where $\Phi_0$ is the solution mapping associated to the equation \eqref{eq:exampleEqn} with zero initial condition (i.e. $\Phi_0$ is the same as $\Phi$ except for the initial condition).
\end{lem}
\begin{proof}
Take $\psi, h \in L^2(0,T;H)$ and denote $w_s := \Phi(\psi + sh)$ and $w:=\Phi(\psi)$. Their difference satisfies
\begin{align*}
\left(\frac{w_s-w}{s}\right)' + B\left(\frac{w_s-w}{s}\right) &= \frac{g(\psi+sh)-g(\psi)}{s},\\
(w_s-w)(0) &= 0.
\end{align*}
Consider the solution $\eta$ of the PDE
\begin{align*}
\eta' + B\eta &= g'(\psi)(h),\\
\eta(0) &= 0.
\end{align*}
Letting $(w_s-w)\slash s =: \eta_s$ we observe
\begin{align*}
(\eta_s-\eta)' + B(\eta_s-\eta) = \frac{g(\psi+sh)-g(\psi)}{s} - g'(\psi)(h),
\end{align*}
whence, through standard energy estimates,
\[\norm{\eta_s-\eta}{L^\infty(0,T;H)} + \norm{\eta_s-\eta}{L^2(0,T;V)} \leq C\norm{\frac{g(\psi+sh)-g(\psi)}{s} - g'(\psi)(h)}{L^2(0,T;H)},\]
which implies that $\eta_s \to \eta$ in $L^2(0,T;V) \cap L^\infty(0,T;H)$ due to the differentiability assumption on $g$. We also have
\begin{align*}
\norm{\eta_s'-\eta'}{L^2(0,T;V^*)} &\leq C\norm{\frac{g(\psi+sh)-g(\psi)}{s} - g'(\psi)(h)}{L^2(0,T;V^*)} + CC_b^B\norm{\eta_s-\eta}{L^2(0,T;V)},
\end{align*}
and hence $\eta_s \to \eta$ in $W(0,T)$ showing that $\Phi\colon L^2(0,T;H) \to W(0,T)$ is differentiable as desired with  $\Phi'(\psi)(h) =  \Phi_0(g'(\psi)(h))$ being the solution of the same PDE with source term $g'(\psi)(h)$ and zero initial data. Along the same lines as the previous arguments, $\Phi'(\psi)(h) \in W_s(0,T)$.
\end{proof}
Given a source term $f \in L^2(0,T;H_+)$ and initial condition $u_0 \in V_+$ satisfying $u_0 \leq \Phi(0)(0) = w_0$ and \eqref{ass:NEWASSconditionOnIDPerturbed}, we fix a solution $u$ solving the QVI 
\begin{equation}\label{eq:egQVI}
\begin{aligned}
\text{Find $u \in W(0,T)$ with $u \leq \Phi(u)$} :  \int_0^T \langle u'(t) + Au(t) - f(t), u(t) - v(t) \rangle &\leq 0\\
&\!\!\!\!\!\!\!\!\!\!\!\!\!\!\!\!\!\!\!\!\!\!\!\!\!\!\!\!\!\!\!\!\!\!\!\!\!\!\!\!\!\!\!\!\!\!\!\!\!\!\!\!\!\!\!\!\!\!\!\!\!\! \forall v \in L^2(0,T;V) : v \leq \Phi(u),\\
u(0) &= u_0.
\end{aligned}
\end{equation}
The fact that such a solution indeed exists will be shown later.
\begin{lem}\label{lem:egLem1}Let $d \in L^2(0,T;H_+)$ be such that $f+sd$ is increasing in time and let $w_0 \in V_+$ and $Bw_0 \leq 0$. Then Assumption \ref{ass:newAss} holds.
\end{lem}
\begin{proof}We showed in the previous paragraph the satisfaction of \eqref{ass:inWImpliesLPhiInL2H}, and the condition \eqref{ass:NEWASSFOR} is irrelevant since $d \geq 0$. Let us check the remaining assumptions.
$ $\newline
\vspace{-.5cm}
\begin{itemize}[leftmargin=*]
\item Regarding the increasing property of $\Phi$ in time, we argue as follows. Making the substitution $\bar w = w- w_0$ in \eqref{eq:exampleEqn} in order to remove the initial condition, the transformed PDE is
\begin{equation*}
\begin{aligned}
\bar w' + B\bar w &= g(\psi) -Bw_0&&\text{on $(0,T) \times (0,\omega)$},\\
\bar w(\cdot,0) &= \bar w(\cdot,\omega) = 0&&\text{on $(0,T)$},\\
\bar w(0,\cdot) &= 0 &&\text{on $(0,\omega)$}.
\end{aligned}
\end{equation*}
The solution of this can be written in terms of the Green's function
\[G(x,y,t) := \frac{2}{L}\sum_{n=1}^\infty \sin\left(\frac{n\pi x}{L}\right)\sin\left(\frac{n\pi y}{L}\right)e^{-\frac{n^2\pi^2t}{L^2}}\]
through the integral representation
\[\bar w(t,x) := 
\int_0^t \int_0^L (g(\psi(s,y))-Bw_0(y))G(x,y,t-s)\;\mathrm{d}y\mathrm{d}s.\]
Take $\psi$ to be increasing in time. For $t_1 \leq t_2$,
\begin{align*}
&\bar w(t_1,x) - \bar w(t_2,x)\\
&= \int_0^{t_1} \int_0^L (g(\psi(s,y))-Bw_0(y))G(x,y,t_1-s)\;\mathrm{d}y\mathrm{d}s\\
&\quad - \int_0^{t_2} \int_0^L (g(\psi(s,y))-Bw_0(y))G(x,y,t_2-s)\;\mathrm{d}y\mathrm{d}s\\
&= \int_0^{t_1} \int_0^L (g(\psi(t_1-r,y))-Bw_0(y))G(x,y,r)\;\mathrm{d}y\mathrm{d}r \\
&\quad- \int_0^{t_2} \int_0^L (g(\psi(t_2-r,y))-Bw_0(y))G(x,y,r)\;\mathrm{d}y\mathrm{d}r\tag{where $r:=t_i -s$}\\
&= \int_0^{t_1} \int_0^L (g(\psi(t_1-r,y))-g(\psi(t_2-r,y))G(x,y,r)\;\mathrm{d}y\mathrm{d}r\\
&\quad- \int_{t_1}^{t_2} \int_0^L (g(\psi(t_2-r,y))-Bw_0(y))G(x,y,r)\;\mathrm{d}y\mathrm{d}r\\
&\leq 0
\end{align*}
since $\psi$ is increasing in time and $g$ is increasing which implies that the first term is less than zero, and the second term is also clearly non-positive as its integrand is non-negative (due to the assumption on $w_0$). This shows that $t \mapsto \bar w(t) = w(t) - w_0 = \Phi(\psi)(t) - w_0$ is increasing, implying \eqref{ass:PhiIsIncreasingInTimeForAllNonNegBochner}.

\item If $\psi_1 \leq \psi_2$ and $w_i:=\Phi(\psi_i)$, the difference solves
\[(w_1-w_2)' + B(w_1-w_2) = g(\psi_1)-g(\psi_2),\]
with zero initial data. Testing with $(w_1-w_2)^+$ and using the fact that $g$ is increasing, the right-hand side is non-positive, showing that $\Phi(\psi_1) \leq \Phi(\psi_2)$.

\item Multiplying the equation by $w^-$ leads to
\[\frac 12 \frac{d}{dt} \norm{w^-}{H}^2 + C_a^B\norm{w^-}{V}^2 \leq (g(\psi), -w^-)_H \leq 0,\]
since $g \geq 0$, which shows that $\Phi(\psi)$ is non-negative too as long as $w_0 \in V_+$, thus \eqref{ass:PhiNonNegativeForNonNegativeObs} and \eqref{ass:L2}. 

\item Let $\psi_n \weaklyto \psi$ in $L^2(0,T;V)$ and set $w_n := \Phi(\psi_n)$ and $w:= \Phi(\psi)$ so that
\begin{align*}
w_n' + Bw_n = g(\psi_n) \quad \text{and} \quad w' + Bw = g(\psi).
\end{align*}
Subtracting one from the other and testing with the difference, we obtain
\[\frac 12 \frac{d}{dt} \norm{w_n-w}{H}^2 + C_a^B\norm{w_n-w}{V}^2 \leq \norm{g(\psi_n)-g(\psi)}{H}\norm{w_n-w}{H}.\]
Making use of Young's inequality on the right-hand side and then integrating over time, we find due to the continuity of $g$ in $L^2(0,T;H)$ that $w_n \to w$ in $L^2(0,T;V) \cap L^\infty(0,T;H)$ and so \eqref{ass:completeContinuityOfPhi} is also valid. 

\item If $h,\psi  \in L^2(0,T;V)$, $\Phi'(\psi)(h) = \Phi_0(g'(\psi)(h)) \in W_s(0,T)$ and by the same reasoning as above, we find 
${L}\Phi'(\psi)(h) \in L^2(0,T;H)$, giving \eqref{ass:inL2VLinftyHimples}. The property \eqref{ass:LFour} will be shown below.
\end{itemize}
\vspace{-0.3cm}
\end{proof}
Let us return now to the QVI \eqref{eq:egQVI}.
\begin{lem}A solution $u$ of \eqref{eq:egQVI} exists with all the stated properties. Furthermore, the assumptions on $u$ in Assumption \ref{ass:onU} also hold.
\end{lem}
\begin{proof}
We aim to apply Theorem \ref{thm:approximationOfQVIByVIIterates}. Indeed, taking the function $0$ as a subsolution (by the comparison principle),  we showed above that $\Phi(\psi) \in C^0([0,T];V)$ for any $\psi \in L^2(0,T;H)$ which ensures \eqref{ass:PhiAtSubsolnIsCts} and \eqref{ass:PhiTakesWsIntoCts}. We also proved in the same lemma the validity of \eqref{ass:PhiNonNegativeForNonNegativeObs} and \eqref{ass:completeContinuityOfPhi} as well as the increasing-in-time properties \eqref{ass:newIncreasingPropertyOfSub}, \eqref{ass:newIncreasingProperty2} for $\Phi$, hence there is a $u \in W_s(0,T)$ solving the above QVI and Assumption \ref{ass:onU} is met.
\end{proof}
\begin{lem}
The local hypotheses comprising Assumptions \ref{ass:forBoundednessOfAlphas} and \ref{ass:requiredLocalAssumptions} hold.
\end{lem}
\begin{proof}
Let $h \in L^2(0,T;V)$ and $\eta := \Phi'(u)(h)$ so that
\begin{align*}
\eta' + B\eta&= g'(u)h,\\
\eta(0) &= 0.
\end{align*}
\begin{itemize}[leftmargin=*]
\item From the standard energy estimate
\begin{align*}
\frac 12 \norm{\eta(r)}{H}^2 + C_a^B\int_0^r \norm{\eta}{V}^2 &\leq \int_0^r (g'(u)h, \eta)_H
\leq \norm{g'}{\infty}\norm{\eta}{L^\infty(0,r;H)}\norm{\eta}{L^1(0,r;H)},
\end{align*}
we find
\begin{align*}
\frac 12 \norm{\eta}{L^\infty(0,T;H)}^2 + C_a^B \norm{\eta}{L^2(0,T;V)}^2 
&\leq \norm{g'}{\infty}\norm{\eta}{L^\infty(0,T;H)}\norm{\eta}{L^1(0,T;H)},
\end{align*}
which leads to
\begin{equation}\label{eg3}
\norm{\Phi'(u)(h)}{L^2(0,T;V)}^2 \leq \frac{T\norm{g'}{\infty}^2}{2C_a^B}\norm{h}{L^2(0,T;H)}^2,
\end{equation}
i.e., \eqref{ass:lb1} after using the continuity of the embedding $V \cts H$. 
\item The second estimate above also leads to 
\begin{align*}
\left(\frac 12-\epsilon\right) \norm{\eta}{L^\infty(0,T;H)}^2 + C_a^B \norm{\eta}{L^2(0,T;V)}^2 
&\leq \norm{g'}{\infty}^2C_\epsilon\norm{h}{L^1(0,T;H)}^2\\
&\leq \norm{g'}{\infty}^2C_\epsilon T \norm{h}{L^2(0,T;H)}^2,
\end{align*}
whence \eqref{eq:boundednessOfAlphaDerivativeAtT}:
\begin{equation}\label{eg4}
\norm{\Phi'(u)(h)}{L^\infty(0,T;H)} \leq \norm{g'}{\infty}\sqrt{\frac{TC_\epsilon}{1\slash 2-\epsilon}}\norm{h}{L^2(0,T;H)}.
\end{equation}
\item For \eqref{ass:lb2}, we simply use the equation itself and \eqref{eg3}:
\begin{align}
\nonumber \norm{\partial_t \Phi'(u)(h)}{L^2(0,T;V^*)} &\leq \norm{g'}{\infty}\norm{h}{L^2(0,T;V^*)} + C_b^B\norm{\Phi'(u)(h)}{L^2(0,T;V)}\\
&\leq \norm{g'}{\infty}\left(1+C_b^B\sqrt{\frac{T}{2C_a^B}}\right)\norm{h}{L^2(0,T;H)}\label{eg5}.
\end{align}
Therefore, if $T$ and/or $\norm{g'}{\infty}$ is sufficiently small, assumption \eqref{ass:smallnessCondition} holds.
\item The estimate \eqref{eg3} yields the first property in \eqref{ass:LFour}  whilst \eqref{eg3} and \eqref{eg5} yield the second property.

\item Regarding the completely continuity requirements of Assumption \ref{ass:requiredLocalAssumptions},  let us take $h_n \weaklyto h$ in $L^2(0,T;V)$ and define $w_n := \Phi'(u)(h_n) = \Phi_0(g'(u)(h_n))$ and $w:=\Phi'(u)(h) = \Phi_0(g'(u)h)$. We see that
\begin{align*}
w_n' + Bw_n = g'(u)h_n \quad\text{and}\quad w' + Bw = g'(u)h,
\end{align*}
which immediately implies 
\begin{align*}
C_a\norm{w_n-w}{L^2(0,T;V)}^2 \leq \int_0^T (g'(u)h_n - g'(u)h, w_n-w)_H,
\end{align*}
on which using the boundedness of $g'$ and the strong convergence $h_n \to h$ in $L^2(0,T;H)$, we get that $w_n \to w$ in $L^2(0,T;V)$, giving \eqref{ass:compContOfDerivOfPhi}. 

\item We also have that ${L}\Phi'(u)(h_n) = (\partial_t + A)\Phi_0(g'(u)h_n) = (\partial_t + A)w_n = (\partial_t + B)w_n + Aw_n - Bw_n$, which implies that
\begin{align}
\nonumber &\norm{{L}\Phi'(u)(h_n)-{L}\Phi'(u)(h)}{L^2(0,T;H)}\\
\nonumber &\quad= \norm{g'(u)h_n - g'(u)h + (A-B)(w_n-w)}{L^2(0,T;H)}\\
\nonumber &\quad\leq \norm{g'(u)(h_n - h)}{L^2(0,T;H)} + \norm{A(w_n-w)}{L^2(0,T;H)}\\
&\quad\quad+ \norm{B(w_n-w)}{L^2(0,T;H)}.\label{eq:eg4}
\end{align}
The first term on the right-hand side converges to zero because $g'$ is bounded. For the third term, we note  the following standard estimate for linear parabolic PDEs (using the differentiability of $B$):
\begin{align*}
\norm{w_n-w}{W_s(0,T)} \leq C\norm{g'(u)(h_n-h)}{L^2(0,T;H)}
\end{align*}
whence
\begin{align*}
\norm{Bw_n - Bw}{L^2(0,T;H)} &\leq \norm{g'(u)(h_n - h)}{L^2(0,T;H)} + \norm{w_n'-w'}{L^2(0,T;H)}\\
&\leq C\norm{g'(u)(h_n - h)}{L^2(0,T;H)}.
\end{align*}
Regarding the second term of \eqref{eq:eg4}, we manipulate using the equivalence of norms as following:
\begin{align*}
\norm{A(w_n-w)}{L^2(0,T;H)} &\leq C_1\norm{w_n-w}{L^2(0,T;W)}\\
&\leq C_2\left(\norm{w_n-w}{L^2(0,T;H)} + \norm{B(w_n-w)}{L^2(0,T;H)}\right)
\end{align*}
and we apply again the estimate from the previous step on the right-hand side and this eventually leads to \eqref{ass:strongCCofDerivative}.

\item The bounds \eqref{eg3}, \eqref{eg4} and \eqref{eg5} imply \eqref{ass:big0}, \eqref{ass:bigA}, \eqref{ass:bigB} and the smallness condition \eqref{ass:smallnessOfDerivOfPhi} holds again if $T$ or $\norm{g'}{\infty}$ is sufficiently small.
\end{itemize}
\end{proof}
Having met all the requirements, we may apply Theorem \ref{thm:directionalDifferentiability} to infer that the solution mapping taking $f \mapsto u$ in \eqref{eq:egQVI} is directionally differentiable in the stated sense and the associated derivative solves the QVI
\begin{align*}
&\alpha-\Phi_0(g'(u)(\alpha)) \in C_{\mathbb{K}_0}(\Phi(u)-u) \cap [u' + Au - f]^\perp :\\
& \qquad \qquad \int_0^T  \langle w' + A\alpha - d, \alpha-w \rangle  \leq \frac 12 \norm{w(0)}{H}^2\\
&\qquad \qquad \qquad\forall w : w -\Phi_0(g'(u)(\alpha)) \in \cl_{W}(T_{\mathbb{K}_0}^{\mathrm{rad}}(\Phi(u)-u)\cap [u'+Au-f]^\perp).
\end{align*}
\section*{Acknowledgements}
This research was carried out in the framework of {\sc MATHEON} supported by the Einstein Foundation Berlin within the
ECMath projects OT1, SE5 and SE19 and is funded by the Deutsche Forschungsgemeinschaft (DFG, German Research Foundation) under Germany's Excellence Strategy – The Berlin Mathematics Research Center MATH$^+$ (EXC-2046/1, project ID: 390685689), and the {\sc MATHEON} project A-AP24. The authors further acknowledge the
support of the DFG through the DFG--SPP 1962: Priority Programme `Non-smooth and Complementarity-based
Distributed Parameter Systems: Simulation and Hierarchical Optimization' within Projects 10 and 11.
\bibliographystyle{abbrv}
\bibliography{QVIPaper}

\begin{thebibliography}{10}

\bibitem{Adams}
R.~A. Adams and J.~J.~F. Fournier.
\newblock {\em Sobolev Spaces}.
\newblock Academic Press, second edition, 2003.

\bibitem{AHR}
A.~Alphonse, M.~Hinterm\"{u}ller, and C.~N. Rautenberg.
\newblock Directional differentiability for elliptic quasi-variational
  inequalities of obstacle type.
\newblock {\em Calc. Var. Partial Differential Equations}, 58(1):Art. 39, 47,
  2019.

\bibitem{MR742624}
V.~Barbu.
\newblock {\em Optimal control of variational inequalities}, volume 100 of {\em
  Research Notes in Mathematics}.
\newblock Pitman (Advanced Publishing Program), Boston, MA, 1984.

\bibitem{BarrettPrigozhinSuperconductivity}
J.~W. Barrett and L.~Prigozhin.
\newblock A quasi-variational inequality problem in superconductivity.
\newblock {\em Math. Models Methods Appl. Sci.}, 20(5):679--706, 2010.

\bibitem{BarrettPrigozhinSandpile}
J.~W. Barrett and L.~Prigozhin.
\newblock A quasi-variational inequality problem arising in the modeling of
  growing sandpiles.
\newblock {\em ESAIM Math. Model. Numer. Anal.}, 47(4):1133--1165, 2013.

\bibitem{MR3231973}
J.~W. Barrett and L.~Prigozhin.
\newblock Lakes and rivers in the landscape: a quasi-variational inequality
  approach.
\newblock {\em Interfaces Free Bound.}, 16(2):269--296, 2014.

\bibitem{MR3335194}
J.~W. Barrett and L.~Prigozhin.
\newblock Sandpiles and superconductors: nonconforming linear finite element
  approximations for mixed formulations of quasi-variational inequalities.
\newblock {\em IMA J. Numer. Anal.}, 35(1):1--38, 2015.

\bibitem{BensoussanLionsArticle}
A.~Bensoussan, M.~Goursat, and J.-L. Lions.
\newblock Contr\^ole impulsionnel et in\'equations quasi-variationnelles
  stationnaires.
\newblock {\em C. R. Acad. Sci. Paris S\'er. A-B}, 276:A1279--A1284, 1973.

\bibitem{Bensoussan1974}
A.~Bensoussan and J.-L. Lions.
\newblock Controle impulsionnel et in\'{e}quations quasi-variationnelles
  d'\'{e}volutions.
\newblock {\em C. R. Acad. Sci. Paris}, 276:1333--1338, 1974.

\bibitem{MR653144}
A.~Bensoussan and J.-L. Lions.
\newblock {\em Applications of variational inequalities in stochastic control},
  volume~12 of {\em Studies in Mathematics and its Applications}.
\newblock North-Holland Publishing Co., Amsterdam-New York, 1982.
\newblock Translated from the French.

\bibitem{LionsBensoussan}
A.~Bensoussan and J.-L. Lions.
\newblock {\em Impulse control and quasivariational inequalities}.
\newblock $\mu $. Gauthier-Villars, Montrouge; Heyden \& Son, Inc.,
  Philadelphia, PA, 1984.
\newblock Translated from the French by J. M. Cole.

\bibitem{Birkhoff}
G.~Birkhoff.
\newblock {\em Lattice theory}, volume~25 of {\em American Mathematical Society
  Colloquium Publications}.
\newblock American Mathematical Society, Providence, R.I., third edition, 1979.

\bibitem{MR1756264}
J.~F. Bonnans and A.~Shapiro.
\newblock {\em Perturbation analysis of optimization problems}.
\newblock Springer Series in Operations Research. Springer-Verlag, New York,
  2000.

\bibitem{Bonnans}
J.~F. Bonnans and A.~Shapiro.
\newblock {\em Perturbation analysis of optimization problems}.
\newblock Springer Series in Operations Research. Springer-Verlag, New York,
  2000.

\bibitem{Boyer2012}
F.~Boyer and P.~Fabrie.
\newblock {\em Mathematical Tools for the Study of the Incompressible
  Navier-Stokes Equations and}.
\newblock Applied mathematical sciences. Springer, 2012.

\bibitem{Christof}
C.~Christof.
\newblock Sensitivity analysis and optimal control of obstacle-type evolution
  variational inequalities.
\newblock {\em SIAM J. Control Optim.}, 57(1):192--218, 2019.

\bibitem{Evans}
L.~C. Evans.
\newblock {\em Partial Differential Equations}.
\newblock American Mathematical Society, 1998.

\bibitem{Facchinei2007}
F.~Facchinei and C.~Kanzow.
\newblock Generalized nash equilibrium problems.
\newblock {\em 4OR}, 5(3):173--210, Sep 2007.

\bibitem{ASNSP}
A.~Friedman and R.~Jensen.
\newblock A parabolic quasi-variational inequality arising in hydraulics.
\newblock {\em Annali della Scuola Normale Superiore di Pisa - Classe di
  Scienze}, Ser. 4, 2(3):421--468, 1975.

\bibitem{MR0415069}
A.~Friedman and D.~Kinderlehrer.
\newblock A class of parabolic quasi-variational inequalities.
\newblock {\em J. Differential Equations}, 21(2):395--416, 1976.

\bibitem{Glowinski1981}
R.~Glowinski, J.-P. Lions, and R.~Tr\'{e}moli\`{e}res.
\newblock {\em Numerical Analysis of Variational Inequalities}.
\newblock North-Holland, 1981.

\bibitem{MR3726880}
S.~C. Gupta.
\newblock {\em The classical {S}tefan problem}.
\newblock Elsevier, Amsterdam, 2018.
\newblock Basic concepts, modelling and analysis with quasi-analytical
  solutions and methods, New edition of [ MR2032973].

\bibitem{MR0481060}
A.~Haraux.
\newblock How to differentiate the projection on a convex set in {H}ilbert
  space. {S}ome applications to variational inequalities.
\newblock {\em J. Math. Soc. Japan}, 29(4):615--631, 1977.

\bibitem{Haraux1977}
A.~Haraux.
\newblock {How to differentiate the projection on a convex set in Hilbert
  space. Some applications to variational inequalities}.
\newblock {\em J. Math. Soc. Japan}, 29(4):615--631, 1977.

\bibitem{HARKER199181}
P.~T. Harker.
\newblock Generalized nash games and quasi-variational inequalities.
\newblock {\em European Journal of Operational Research}, 54(1):81 -- 94, 1991.

\bibitem{Hintermueller2012}
M.~Hinterm\"{u}ller and C.~N. Rautenberg.
\newblock Parabolic quasi-variational inequalities with gradient-type
  constraints.
\newblock {\em SIAM J. Optim.}, 23(4):2090--2123, 2013.

\bibitem{HRSQuasi}
M.~Hinterm{\"u}ller, C.~N. Rautenberg, and N.~Strogies.
\newblock Dissipative and non-dissipative evolutionary quasi-variational
  inequalities with gradient constraints.
\newblock {\em Set-Valued and Variational Analysis}, Jul 2018.

\bibitem{Jarusek}
J.~Jaru\v{s}ek, M.~Krbec, M.~Rao, and J.~Sokolowski.
\newblock Conical differentiability for evolution variational inequalities.
\newblock {\em J. Differential Equations}, 193(1):131--146, 2003.

\bibitem{MR1677798}
M.~Kunze and M.~D.~P. Monteiro~Marques.
\newblock On parabolic quasi-variational inequalities and state-dependent
  sweeping processes.
\newblock {\em Topol. Methods Nonlinear Anal.}, 12(1):179--191, 1998.

\bibitem{KunzeRodrigues}
M.~Kunze and J.~F. Rodrigues.
\newblock An elliptic quasi-variational inequality with gradient constraints
  and some of its applications.
\newblock {\em Math. Methods Appl. Sci.}, 23(10):897--908, 2000.

\bibitem{Lions1969}
J.-L. Lions.
\newblock {\em Quelques M\'{e}thodes de R\'{e}solutions des Probl\'{e}mes aux
  Limites non Lin\'{e}aires}.
\newblock Dunod, Gauthier-Villars, 1969.

\bibitem{MR0600331}
J.-L. Lions.
\newblock {\em Perturbations singuli\`eres dans les probl\`emes aux limites et
  en contr\^{o}le optimal}.
\newblock Lecture Notes in Mathematics, Vol. 323. Springer-Verlag, Berlin-New
  York, 1973.

\bibitem{Lions1973}
J.-L. Lions.
\newblock Sur le c\^{o}ntrole optimal des systemes distribu\'{e}es.
\newblock {\em Enseigne}, 19:125--166, 1973.

\bibitem{MR0423155}
F.~Mignot.
\newblock Contr\^ole dans les in{\'e}quations variationelles elliptiques.
\newblock {\em J. Functional Analysis}, 22(2):130--185, 1976.

\bibitem{Pang2005}
J.-S. Pang and M.~Fukushima.
\newblock Quasi-variational inequalities, generalized nash equilibria, and
  multi-leader-follower games.
\newblock {\em Computational Management Science}, 2(1):21--56, Jan 2005.

\bibitem{Penot}
J.-P. Penot.
\newblock {\em Calculus without derivatives}, volume 266 of {\em Graduate Texts
  in Mathematics}.
\newblock Springer, New York, 2013.

\bibitem{Prigozhin1994}
L.~Prigozhin.
\newblock Sandpiles and river networks: extended systems with non-local
  interactions.
\newblock {\em Phys. Rev. E}, 49:1161--1167, 1994.

\bibitem{Prigozhin}
L.~Prigozhin.
\newblock On the {B}ean critical-state model in superconductivity.
\newblock {\em European Journal of Applied Mathematics}, 7:237--247, 1996.

\bibitem{Prigozhin1996}
L.~Prigozhin.
\newblock Sandpiles, river networks, and type-ii superconductors.
\newblock {\em Free Boundary Problems News}, 10:2--4, 1996.

\bibitem{PrigozhinSandpile}
L.~Prigozhin.
\newblock Variational model of sandpile growth.
\newblock {\em European J. Appl. Math.}, 7(3):225--235, 1996.

\bibitem{MR1765540}
J.~F. Rodrigues and L.~Santos.
\newblock A parabolic quasi-variational inequality arising in a
  superconductivity model.
\newblock {\em Ann. Scuola Norm. Sup. Pisa Cl. Sci. (4)}, 29(1):153--169, 2000.

\bibitem{Showalter}
R.~E. Showalter.
\newblock {\em Monotone Operators in Banach Space and Nonlinear Partial
  Differential Equations}.
\newblock American Mathematical Society, 1997.

\bibitem{tartar1974inequations}
L.~Tartar.
\newblock In{\'e}quations quasi variationnelles abstraites.
\newblock {\em CR Acad. Sci. Paris S{\'e}r. A}, 278:1193--1196, 1974.

\bibitem{WachsmuthGuidedTour}
G.~Wachsmuth.
\newblock A guided tour of polyhedric sets: Basic properties, new results on
  intersections and applications.
\newblock {\em To appear in J. Convex Anal.}, 2018.

\end{thebibliography}
\end{document}